\documentclass{amsart}
\usepackage[dvipsnames]{xcolor}
\usepackage[utf8]{inputenc}
\usepackage{mathtools,tikz-cd,amsmath,amssymb,bbm,stmaryrd,mathabx,caption,mathrsfs}
\usepackage[shortlabels]{enumitem}
\usepackage[parfill]{parskip}
\usepackage[all,2cell]{xy}
\xyoption{2cell}

\usepackage{tikz}
\tikzset{anchorbase/.style={baseline={([yshift=-0.5ex]current bounding box.center)}},tinynodes/.style={font=\tiny,text height=0.75ex,text depth=0.15ex},
  cross line/.style={preaction={draw=white,line width=4.75pt,-}},
  cross line thick/.style={preaction={draw=white,line width=6pt,-}},
  cross line thin/.style={preaction={draw=white,line width=3.5pt,-}},
  colored/.style={line width=2.25},
  uncolored/.style={thin},
  ccolored/.style={line width=1.4},
  dotbullet/.style={fill,circle,inner sep=0.5pt},
}
\usetikzlibrary{decorations.pathmorphing}
\newcommand\numeq[1]%
  {\stackrel{\scriptscriptstyle(\mkern-1.5mu#1\mkern-1.5mu)}{=}}

\pgfdeclarelayer{nodelayer}
\pgfdeclarelayer{edgelayer}
\pgfsetlayers{edgelayer,nodelayer,main}

\tikzset{none/.style={thick}}

\allowdisplaybreaks

\newcommand{\mone}[1][1]{{\,\text{-}#1}}

\newtheorem{thm}{Theorem}[section]

\newtheorem{lem}[thm]{Lemma}
\newtheorem{cor}[thm]{Corollary}

\newtheorem{prop}[thm]{Proposition}
\theoremstyle{definition}
\newtheorem{ex}[thm]{Example}

\newtheorem{defn}[thm]{Definition}

\newtheorem{ques}[thm]{Question}
\newtheorem{rem}[thm]{Remark}

\numberwithin{equation}{section}

\font\sc=rsfs10
\newcommand{\cC}{{\sc\mbox{C}\hspace{1.0pt}}}

\newcommand{\cI}{\sc\mbox{I}\hspace{1.0pt}}

\newcommand{\cS}{\sc\mbox{S}\hspace{1.0pt}}

\newcommand{\cD}{{\sc\mbox{D}\hspace{1.0pt}}}

\font\scc=rsfs7
\newcommand{\ccC}{\scc\mbox{C}\hspace{1.0pt}}
\newcommand{\ccD}{\scc\mbox{D}\hspace{1.0pt}}

\newcommand{\mi}{\mathtt{i}}
\newcommand{\mj}{\mathtt{j}}
\newcommand{\mk}{\mathtt{k}}
\newcommand{\ml}{\mathtt{l}}
\newcommand{\mt}{\mathtt{t}}
\newcommand{\rF}{\mathrm{F}}
\newcommand{\rG}{\mathrm{G}}
\newcommand{\rK}{\mathrm{K}}
\newcommand{\rH}{\mathrm{H}}

\newcommand{\rv}{\mathrm{v}}
\newcommand{\rI}{\mathrm{I}}
\newcommand{\rA}{\mathrm{A}}
\newcommand{\rC}{\mathrm{C}}
\newcommand{\rD}{\mathrm{D}}
\newcommand{\rM}{\mathrm{M}}
\newcommand{\rN}{\mathrm{N}}
\newcommand{\bM}{\mathbf{M}}
\newcommand{\bP}{\mathbf{P}}
\newcommand{\bN}{\mathbf{N}}

\newcommand{\id}{\operatorname{id}}

\newcommand{\op}{^{\operatorname{op}}}
\newcommand{\Hom}{\text{\rm Hom}}

\newcommand{\End}{\operatorname{End}}

\newcommand{\Rep}{\operatorname{Rep}}

\newcommand{\add}{\operatorname{add}}
\newcommand{\ob}{\operatorname{Ob}}

\newcommand{\dash}{\text{-}}
\newcommand{\comod}{\operatorname{comod}}

\newcommand{\ca}{{\cD}_A}

\newcommand{\cj}{\cC_{\mathcal J}}

\newcommand{\ihom}{\underline{\operatorname{Hom}}}

\newcommand{\inj}{\hookrightarrow}

\newcommand{\Id}{\operatorname{Id}}

\begin{document}

\title{Adjunction in the absence of identity}
\author{Hankyung Ko, Volodymyr Mazorchuk and Xiaoting Zhang}

\begin{abstract}
We develop a bicategorical setup in which one can speak about
adjoint $1$-morphisms even in the absence of genuine identity
$1$-morphisms. We also investigate which part of
$2$-representation theory of $2$-categories extends to this
new setup.
\end{abstract}
\maketitle

\section{Introduction}\label{s1}

Let $\mathcal{A}$ and $\mathcal{B}$ be two categories and $F:\mathcal{A}\to \mathcal{B}$
and $G:\mathcal{B}\to \mathcal{A}$ two functors. One says that $(F,G)$ is an {\em adjoint
pair of functors} if there are bijections $\mathcal{B}(F\, X,Y)\cong \mathcal{A}(X,G\, Y)$ natural in $X\in \mathcal{A}$ and $Y\in \mathcal{B}$.
This latter condition is equivalent to the existence of so called {\em adjunction morphisms}
$\varepsilon:FG\to \mathrm{Id}_{\mathcal{B}}$ and $\eta:\mathrm{Id}_{\mathcal{A}}\to GF$
satisfying $\varepsilon_{F}\circ F(\eta)=\mathrm{id}_F $ and
$G(\varepsilon)\circ \eta_{G}=\mathrm{id}_G$. This reformulation allows one to define
adjoint pairs of $1$-morphisms in the general setup of bicategories, which, in particular,
covers $2$-categories, monoidal categories and strict monoidal categories.
There is a catch, though: all this requires existence of some kind of genuine identity 1-morphisms
(or unit objects), and, sometimes, an introduction of such an identity into the category 
might look very artificial.

Consider the following example. Let $A$ be a finite dimensional algebra over a
field $\Bbbk$. The category $A\text{-proj-}A$ of projective $A$-$A$-bimodules is closed
under tensoring over $A$. However, this category is not a monoidal category in general,
as the regular $A$-$A$-bimodule ${}_AA_A$ is not a projective $A$-$A$-bimodule unless
$A$ is semi-simple. So, if one wants to define a monoidal category, one needs to
artificially add ${}_AA_A$ to $A\text{-proj-}A$, cf. \cite[Subsection~7.3]{MM1}. If
$A$ is self-injective, it turns out that, in this extended monoidal category,
every object has both a left and a right adjoint, cf. \cite[Lemma~45]{MM1}
and \cite[Section~7]{MM6}. And the presence of ${}_AA_A$
seems crucial to be able to establish this fundamental property.
The same phenomenon, i.e., the necessity of adding an identity,
motivated the notion of $\mathcal{J}$-simple
finitary $2$-category, see \cite[Subsection~6.2]{MM2} and the construction
of a $2$-category associated to a two-sided cell of a fiat
$2$-category, see \cite[Section~5]{MM1}.

The main motivation for the present paper is to formulate a setup in which
one could work with dual objects (or adjoint $1$-morphisms) of monoidal (resp., bi-) categories
in the absence of a genuine unit object (resp., identity 1-morphisms).
The main idea behind our answer
is to require certain lax and oplax units to exist
(this idea appeared during a discussion with Marco Mackaay when working on \cite{MMMTZ}).
At first glance, this might sound as cheating as we do still require
some kinds of units. However, we do think that the idea is nontrivial.
The main reason for this claim is the fact that, in module categories (resp., $2$-representations), our
lax and oplax units are usually {\em not} represented by the identity functors.
In the setup of fiat $2$-categories as described in \cite{MM1,MM3}, our lax units correspond to the so-called {\em Duflo involutions}
as defined in \cite[Subsection~4.5]{MM1}.

Another inspiration for the present paper stems from the comparison of the
theory of finite groups or monoids versus the theory of finite semigroups.
One might think that the difference between semigroups and monoids is very small, since
given any semigroup which is not yet a monoid, one can always add an external
identity element to turn it into a monoid.  However, the study of {\em simple}
groups, monoids, and semigroups reveals a striking difference between the theories. A simple
group is a group which does not have any non-trivial quotients. Similarly,
define a simple monoid (or semigroup)  as the one which does not have any
non-trivial quotients (in semigroup theory the term {\em congruence free}
is usually used instead of the term simple
in this context). It turns out that finite simple monoids
are, essentially, finite simple groups (up to possible addition of an external
zero element). At the same time, the theory of finite simple semigroups is
significantly richer, see e.g. \cite[Theorem~3.7.1]{Ho}. A major issue
in the case of monoids is the fact that all non-invertible elements of
a finite monoid form an ideal: if this latter ideal is not a singleton,
one can factor it out to get a proper quotient of the original monoid.
So, again, we see that presence of a ``global unit'' causes ``problems'' which
can be ``resolved'' by withdrawing the requirement on its existence.

We emphasize that our new setup for the study of adjoint pairs
which we define in this paper is not completely independent from the
classical one. The notion of an adjunction we define is equivalent to the
classical one (compared, e.g., to a {\em different} notion of adjunction of
semifunctors studied in \cite{Ha}) and hence admits reformulation in terms
of proper identities. This is explained by the property that the (op)lax
identities we work with, although being themselves not necessarily represented by the
identity functor in a $2$-representation, are always represented by a functor
which has a natural transformation to (or from) the identity functor, thus
giving a possibility to relate our adjunction to the classical one.

Let us now briefly describe the structure of the paper. Section~\ref{s2} is
devoted to the introduction and preliminary analysis of our setup.
In Subsection~\ref{s2.1}, we recall the definitions of a $2$-semicategory
and various versions of (op)lax units in $2$-semicategories and formulate
our main definition: Definition~\ref{defxxx} of a bilax-unital $2$-category.
To simplify arguments and proofs throughout the paper, we often use
diagrammatic calculus for computations with $2$-morphisms. The main
ingredients of this diagrammatic calculus are also set up in Subsection~\ref{s2.1}.
In Subsection~\ref{sec:abelianization} we recall diagrammatic abelianizations
of additive categories and its analogues for bi- and $2$-categories in our setup.
Subsection~\ref{s2.3} describes the notion of a $2$-representation
and the $2$-category which these $2$-representations form, adapted to
our setup. In Subsection~\ref{s2.4} we adapt the notions of algebra and coalgebra
object to the setup of bilax-unital $2$-categories and further,
in Subsection~\ref{s2.5} we adapt the corresponding notions of (co)modules.
Subsection~\ref{s2.6} contains our main definitions: there we define the notion of
adjoint $1$-morphisms for  bilax-unital $2$-categories and establish their basic
properties: uniqueness of the adjoint, adjoint of composition
and how adjunction gives rise to certain algebra and coalgebra $1$-morphisms.

Section~\ref{s3} is devoted to a detailed study of our
protagonist in this paper: fiax categories. These are defined in
Subsection~\ref{ss:fiax}. In Subsection~\ref{s3.2}, with each
fiax category $\cC$ we associate a fiat $2$-category $\hat{\cC}$ with weak units
(i.e., weak identity $1$-morphisms). In Subsection~\ref{sec:lifted-2-rep}
we describe how certain $2$-representations of $\cC$  can be
lifted to $2$-representations of $\hat{\cC}$. In Subsection~\ref{s3.4}
we describe how (co)algebra $1$-morphisms can be lifted from $\cC$ to $\hat{\cC}$.
The rest of the subsection is devoted to the study of the connection between
$2$-representations and (co)algebra $1$-morphisms, following \cite{EGNO}
and \cite{MMMT}. The main technical problem in our case seems to be the check of
unitality axioms, as the main difference between our and the classical setups is
exactly the absence of genuine identity $1$-morphisms. We discuss
internal homs in Subsection~\ref{s3.6}, apply internal homs to connect
$2$-representations of $\hat{\cC}$ with (co)modules over
(co)algebra $1$-morphisms in Subsection~\ref{sec:coalgcomod}, and develop
an analogue of Morita-Takeuchi theory in Subsection~\ref{s3.8}.
In the latter, an important role is played by (co)tensor products,
which are discussed in Subsection~\ref{s3.5}. In
Subsection~\ref{s3.new} we describe correspondence between coalgebras
and 2-representations. Finally, in
Subsection~\ref{s3.9} we recall the basic combinatorial notions
for additive bicategories and adjust them to our setup.

In Section~\ref{s4} we provide a number of examples. Our motivating example
is the category $\ca$ of projective bimodules for a
finite dimensional self-injective algebra $A$ which is
given in detail in Subsection~\ref{sec:example1}. The corresponding
$2$-category $\cC_A$ was defined in \cite[Subsection~7.3]{MM1}. A major
difference between our $\ca$ and the $2$-category $\cC_A$ from
\cite[Subsection~7.3]{MM1} is the absence, in $\ca$, of the regular
bimodule (which is the unit object in the monoidal category of bimodules).
Instead, we have lax and oplax units, given by Duflo 1-morphisms in the
corresponding left cells, which are projective bimodules
(the regular bimodule is, usually, not projective).
This emphasizes our main point: compared to our $\ca$, the $2$-category
$\cC_A$ from \cite[Subsection~7.3]{MM1} does look artificial.

Some further examples, in particular, fiax categories associated to
two-sided cells of fiat $2$-categories and the category of projective $G$-modules for a finite group $G$,
are briefly discussed in the
remaining subsection of Section~\ref{s4}.

\subsection*{Acknowledgments}
This research was partially supported by
the Swedish Research Council and G{\"o}ran Gustafsson Stiftelse.
We especially thank Marco Mackaay in conversation with whom the
main idea behind this paper crystalized.
We thank Gustavo Jasso for information about adjunction of semifunctors.

\section{Bilax 2-representation theory}\label{s2}

Let $\Bbbk$ be an algebraically closed field. We assume that all
categories (and functors, etc) are additive and $\Bbbk$-linear.
If not mentioned otherwise or clear from the context,
all categories we work with are assumed to be small.

\subsection{Bilax-unital 2-categories}\label{s2.1}
We first recall the definition of a $2$-semicategory.

\begin{defn}\label{def}
A {\em 2-semicategory} $\cC$ consists of
\begin{itemize}
    \item a class $\ob\cC$ of objects;
    \item for each $\mi,\mj\in \ob\cC$, a category $\cC(\mi,\mj)$, whose objects are called {\em 1-morphisms}, morphisms are called {\em 2-morphisms}, and the composition is called the {\em vertical composition}, denoted by $\circ_{\rv}$;
    \item for each $\mi,\mj,\mk\in \ob\cC$, a functor (called the {\em horizontal composition})
    \[h_{\mi,\mj,\mk}:\cC(\mj,\mk)\times \cC(\mi,\mj)\to \cC(\mi,\mk);\]
    which is (strictly) associative, that is, we have
    \[ h_{\mi,\mk,\ml}\circ
    (\Id_{\ccC(\mk,\ml)}\times h_{\mi,\mj,\mk})=h_{\mi,\mj,\ml}\circ (h_{\mj,\mk,\ml}\times \Id_{\ccC (\mi,\mj)}).\]
\end{itemize}
\end{defn}
For simplicity, we write $\rG\rF:=h(\rG,\rF)$, for any $1$-morphisms $\rF\in\cC(\mi,\mj),\rG\in\cC(\mj,\mk)$, and $\beta\alpha:=h(\beta, \alpha)$, for any $2$-morphisms $\alpha:\rF\to \rH$ in $\cC(\mi,\mj)$ and $\beta:\rG\to \rK$ in $\cC(\mj,\mk)$. Here we omit the subscript of the horizontal functor $h$ when there is no confusion.
Due to the strict associativity, we could write $\rF_1\rF_2\ldots \rF_n$ (whenever it makes sense) for the composition of $1$-morphisms $\rF_1,\rF_2,\ldots,\rF_n$, where $n\geq 3$.

We depict a 2-morphism $\alpha:\rF\to \rH$ in $\cC(\mi,\mj)$ diagrammatically as follows:
\begin{equation*}
\begin{tikzpicture}[anchorbase]
      \node (j) at (-1,0) {$\mj$};
      \node (i) at (1,0) {$\mi$};
      \node at (0,0) [rectangle,draw, very thick] (d0) {$\alpha$};
      \draw[very thick,blue] (d0) to (0,.7)node[above] {$\rH$};
      \draw[very thick] (d0) to (0,-.7)node[below] {$\rF$};
    \end{tikzpicture}
\end{equation*}
In partcular, the identity 2-morphism $\id_{\rF}$ on $\rF$ is simply drawn as:
\begin{equation*}
    \begin{tikzpicture}[anchorbase]
      \node (j) at (-1,0.5) {$\mj$};
      \node (i) at (1,0.5) {$\mi$};
      \draw[very thick] (0,1) node[above] {$\rF$} to (0,0) node[below] {$\rF$};
    \end{tikzpicture}
\end{equation*}
The vertical composition can be depicted by the vertical concatenation of diagrams. The strict associativity justifies drawing horizontal composition as horizontal concatenation. Functoriality of the horizontal composition implies the interchange law, that is,
for any $1$-morphisms $\rF,\rH\in\cC(\mi,\mj), \rG,\rK\in\cC(\mj,\mk)$ and $2$-morphisms $\alpha:\rF\to \rH$ in
$\cC(\mi,\mj)$ and $\beta:\rG\to \rK$ in $\cC(\mj,\mk)$, we have:
\begin{equation}\label{diag:sliding}
    \begin{tikzpicture}[anchorbase]
      \node (j) at (-.5,0) {$\mj$};
      \node (i) at (.5,0) {$\mi$};
      \node (k) at (-1.5,0) {$\mk$};
      \node at (0,.3) [rectangle,draw,very thick] (d0) {$\alpha$};
      \node at (-1,-0.3) [rectangle,draw,very thick] (d1) {$\beta$};
      \draw[blue,very thick] (d0) to (0,1)node[above] {$\rH$};
      \draw[very thick] (d0) to (0,-1)node[below] {$\rF$};
      \draw[red,very thick] (d1) to (-1,1)node[above] {$\rK$};
      \draw[teal,very thick](d1) to (-1,-1)node[below] {$\rG$};
    \end{tikzpicture}
    \quad=\quad
        \begin{tikzpicture}[anchorbase]
      \node (j) at (-.5,0) {$\mj$};
      \node (i) at (.5,0) {$\mi$};
      \node (k) at (-1.5,0) {$\mk$};
      \node at (0,-0.3) [rectangle,draw,very thick] (d0) {$\alpha$};
      \node at (-1,.3) [rectangle,draw,very thick] (d1) {$\beta$};
      \draw[blue,very thick] (d0) to (0,1)node[above] {$\rH$};
      \draw[very thick] (d0) to (0,-1)node[below] {$\rF$};
      \draw[red,very thick] (d1) to (-1,1)node[above] {$\rK$};
      \draw[teal,very thick] (d1) to (-1,-1)node[below] {$\rG$};
    \end{tikzpicture}
\end{equation}
We drop the labels (for objects and 1-morphisms) unless they are necessary.

\begin{defn}\label{def:bilax-unit}
Let $\cC$ be a 2-semicategory and $\mi$ an object in $\cC$.
\begin{enumerate}\label{def:lax-unit}
\item\label{def:lax-unit-1} A {\em lax unit} in $\cC(\mi,\mi)$ is a triple
$(\rI_{\mi},\,l_{\mi},\,r_{\mi})$ consisting of a 1-morphism $\rI_{\mi}\in\cC(\mi,\mi)$
and collections of natural transformations $l_{\mi}=(l_{\mj,\mi})_{\mj\in\ccC}$
and $r_{\mi}=(r_{\mi,\mk})_{\mk\in\ccC}$,
called the {\em left} and {\em right lax unitors}, where
\[l_{\mj,\mi,{}_-}:h(\rI_{\mi},{}_-)\to \Id_{\ccC(\mj,\mi)}\quad
\text{and}\quad r_{\mi,\mk,{}_-}:h({}_-,\rI_{\mi})\to \Id_{\ccC(\mi,\mk)},\]
depicted as
\begin{equation*}
    \begin{tikzpicture}[anchorbase]
      \draw[very thick] (4, 8)node[above] {$\rF$} to (4, 7) node[below] {$\rF$};
      \draw[dashed, very thick ] (3.5, 7)[out=90, in=-180]node[below] {$\rI_{\mi}$} to (4, 7.5) ;
    \end{tikzpicture}
    \quad\text{and}\quad
    \begin{tikzpicture}[anchorbase]
      \draw[very thick] (4, 8)node[above] {$\rG$} to (4, 7)node[below] {$\rG$};
      \draw[dashed ,very thick] (4.5, 7)[out=90, in=0]node[below] {$\ \rI_{\mi}$} to (4, 7.5);
    \end{tikzpicture}
\end{equation*} for any $1$-morphisms
    $\rF\in \cC(\mj,\mi)$ and $\rG\in \cC(\mi,\mk)$, respectively, such that
\begin{enumerate}[(a)]
    \item\label{def:lax-unit-1-a} $\id_{\rG}l_{\mj,\mi,\rF}=r_{\mi,\mk,\rG}\id_{\rF}$, for any $1$-morphisms
    $\rF\in \cC(\mj,\mi)$ and $\rG\in \cC(\mi,\mk)$, that is:
    \begin{equation}\label{diag:lrcompatible}
    \begin{tikzpicture}[anchorbase]
      \draw[blue,very thick] (3,8)node[above] {$\rG$} to (3,7) node[below] {$\rG$};
      \draw[very thick] (4, 8)node[above] {$\rF$} to (4, 7) node[below] {$\rF$};
      \draw[dashed , very thick] (3.5, 7)[out=90, in=-180]node[below] {$\rI_{\mi}$} to (4, 7.5);
    \end{tikzpicture}
    \quad=\quad
    \begin{tikzpicture}[anchorbase]
      \draw[very thick] (5,8)node[above] {$\rF$} to (5,7) node[below] {$\rF$};
      \draw[blue,very thick] (4, 8)node[above] {$\rG$} to (4,7)node[below] {$\rG$};
     \draw[dashed ,very thick] (4.5, 7)[out=90, in=0]node[below] {$\ \rI_{\mi}$} to (4, 7.5);
    \end{tikzpicture}\
\end{equation}
    \item\label{def:lax-unit-1-b} $l_{\ml,\mi,\rF\rH}=l_{\mj,\mi,\rF}\id_{\rH}$, for any $1$-morphisms $\rF\in\cC(\mj,\mi)$ and $\rH\in\cC(\ml,\mj)$; moreover, similarly, $r_{\mi,\mt,\rK\rG}=\id_{\rK} r_{\mi,\mk,\rG}$, for any $1$-morphisms $\rG\in\cC(\mi,\mk)$ and $\rK\in\cC(\mk,\mt)$, which implies that there is no ambiguity in the following diagrams:
\begin{equation}\label{diag:lr2}
        \begin{tikzpicture}[anchorbase]
      \draw[very thick] (4, 8)node[above] {$\rF$} to (4, 7) node[below] {$\rF$};
      \draw[blue,very thick] (4.5, 8) node[above] {$\rH$}to (4.5, 7)node[below] {$\rH$};
      \draw[dashed , very thick] (3.5, 7)[out=90, in=-180]node[below] {$\rI_{\mi}$} to (4, 7.5);
    \end{tikzpicture}
    \quad\text{and}\quad
    \begin{tikzpicture}[anchorbase]
      \draw[blue,very thick] (4, 8) node[above] {$\rG$} to (4, 7)node[below] {$\rG$};
      \draw[very thick] (3.5, 8) node[above] {$\rK$}to (3.5, 7)node[below] {$\rK$};
       \draw[dashed ,very thick] (4.5, 7)[out=90, in=0]node[below] {$\ \rI_{\mi}$} to (4, 7.5);
    \end{tikzpicture}\
\end{equation}
\end{enumerate}

\item A lax unit $\rI_{\mi}=(\rI_{\mi},\,l_{\mi},\,r_{\mi})$ in $\cC(\mi,\mi)$ is said to
be {\em split}
if both $l_{\mj,\mi,\rF}$ and $r_{\mi,\mk,\rG}$ are split epic, for every 1-morphism $\rF\in\cC(\mj,\mi)$ and
every 1-morphism $\rG\in\cC(\mi,\mk)$, respectively, i.e., there exsit 2-morphisms $\alpha_{\rF},\beta_{\rG}$
(not necessarily natural in $\rF$ and $\rG$) such that $l_{\mj,\mi,\rF}\circ_{\rv}\alpha_{\rF}= \id_{\rF}$
and $r_{\mi,\mk,\rG}\circ_{\rv}\beta_{\rG}=\id_{\rG}$, respectively.

\item An {\em oplax unit} in $\cC(\mi,\mi)$ is a triple $(\rI'_{\mi},\,l'_{\mi},\,
r'_{\mi})$ consisting of a 1-morphism $\rI'_{\mi}\in\cC(\mi,\mi)$ and collections of
natural transformations, $l'_{\mi}=(l'_{\mj,\mi})_{\mj\in\ccC}$
and $r'_{\mi}=(r'_{\mi,\mk})_{\mk\in\ccC}$, called the {\em left} and {\em right oplax unitors}, where
\[l'_{\mj,\mi,{}_-}:\Id_{\ccC(\mj,\mi)}\to h(\rI'_{\mi},{}_-)\quad\text{and}\quad
r'_{\mi,\mk,{}_-}:\Id_{\ccC(\mi,\mk)}\to h({}_-,\rI'_{\mi}),\] depicted as
\begin{equation*}
    \begin{tikzpicture}[anchorbase]
      \draw[very thick] (4, 8.09)node[above]{$\rF$} (4,8)  to (4, 7) node[below] {$\rF$};
      \draw[dashed , very thick] (3.5, 8)[out=-90,in=-180] node[above] {$\rI'_{\mi}$} to (4, 7.5);
    \end{tikzpicture}
    \quad\text{and}\quad
    \begin{tikzpicture}[anchorbase]
      \draw[very thick] (4, 8) to (4, 7)node[below] {$\rG$} (4,8.09)node[above] {$\rG$};
      \draw[dashed ,very thick] (4.5, 8)[out=-90,in=0]node[above] {$\ \rI'_{\mi}$}
      to (4, 7.5);
    \end{tikzpicture}
\end{equation*}  for any $1$-morphisms
    $\rF\in \cC(\mj,\mi)$ and $\rG\in \cC(\mi,\mk)$, respectively, which satisfy the duals of the respective coherence axioms for lax unitors, that is,
     \begin{equation}\label{diag:oplaxlrcompatible}
    \begin{tikzpicture}[anchorbase]
      \draw[blue,very thick] (3,8.09)node[above] {$\rG$} to (3,7) node[below] {$\rG$};
      \draw[very thick] (4, 8.09)node[above] {$\rF$} to (4, 7) node[below] {$\rF$};
       \draw[dashed , very thick] (3.5, 8)[out=-90,in=-180] node[above] {$\rI'_{\mi}$} to (4, 7.5);
    \end{tikzpicture}
    \quad=\quad
    \begin{tikzpicture}[anchorbase]
      \draw[very thick] (5,8.09)node[above] {$\rF$} to (5,7) node[below] {$\rF$};
      \draw[blue,very thick] (4, 8.09)node[above] {$\rG$} to (4,7)node[below] {$\rG$};
      \draw[dashed ,very thick] (4.5, 8)[out=-90,in=0]node[above] {$\ \rI'_{\mi}$}
      to (4, 7.5);
    \end{tikzpicture}
\end{equation}
where $\rF\in\cC(\mj,\mi)$ and $\rG\in\cC(\mi,\mk)$, and there is no ambiguity in the diagrams
\begin{equation}\label{diag:oplaxlr2}
        \begin{tikzpicture}[anchorbase]
      \draw[very thick] (4, 8.09)node[above] {$\rF$} to (4, 7) node[below] {$\rF$};
      \draw[blue,very thick] (4.5, 8.09) node[above] {$\rH$}to (4.5, 7)node[below] {$\rH$};
      \draw[dashed , very thick] (3.5, 8)[out=-90,in=-180] node[above] {$\rI'_{\mi}$} to (4, 7.5);
    \end{tikzpicture}
    \quad\text{and}\quad
    \begin{tikzpicture}[anchorbase]
      \draw[blue,very thick] (4, 8.09) node[above] {$\rG$} to (4, 7)node[below] {$\rG$};
      \draw[very thick] (3.5, 8.09) node[above] {$\rK$}to (3.5, 7)node[below] {$\rK$};
      \draw[dashed ,very thick] (4.5, 8)[out=-90,in=0]node[above] {$\ \rI'_{\mi}$}
      to (4, 7.5);
    \end{tikzpicture}
\end{equation}
where $\rF\in\cC(\mj,\mi),\rH\in\cC(\ml,\mj),\rG\in\cC(\mi,\mk)$ and $\rH\in\cC(\mk,\mt)$.
\item An oplax unit $\rI'_{\mi}=(\rI'_{\mi},\,l'_{\mi},\,r'_{\mi})$ in $\cC(\mi,\mi)$ is said to be {\em split} if both $l'_{\mj,\mi,\rF}$ and $r'_{\mi,\mk,\rG}$ are split monic for every 1-morphism $\rF\in\cC(\mj,\mi)$ and $\rG\in\cC(\mi,\mk)$, respectively.
\end{enumerate}
\end{defn}

The naturality of the left and right lax unitors $l_{\mi},r_{\mi}$ can be depicted by the followig diagrams:
\begin{equation}\label{diag:unitornatural}
    \begin{tikzpicture}[anchorbase]
      \node at (0,0) [rectangle,draw,very thick] (d0) {$\alpha$};
      \draw[blue,very thick] (d0) to (0,1) node[above] {$\rH$};
      \draw[very thick] (d0) to (0,-1) node[below] {$\rF$};
      \draw[dashed , very thick] (-.9, -1)[out=90, in=-180]node[below] {$\ \rI_{\mi}$} to (0, .55) ;
    \end{tikzpicture}
    \quad =\quad
    \begin{tikzpicture}[anchorbase]
      \node at (0,0) [rectangle,draw,very thick] (d0) {$\alpha$};
      \draw[blue,very thick] (d0) to (0,1)node[above] {$\rH$};
      \draw[very thick] (d0) to (0,-1)node[below] {$\rF$};
      \draw[dashed , very thick] (-.45, -1)[out=90, in=-180]node[below] {$\ \rI_{\mi}$} to (0, -.5) ;
    \end{tikzpicture}
    \quad\text{and}\quad
        \begin{tikzpicture}[anchorbase]
      \node at (0,0) [rectangle,draw,very thick] (d0) {$\beta$};
      \draw[blue,very thick] (d0) to (0,1)node[above] {$\rK$};
      \draw[very thick] (d0) to (0,-1)node[below] {$\rG$};
      \draw[dashed ,very thick] (.9, -1)[out=90, in=0]node[below] {$\ \rI_{\mi}$} to (0, .55);
    \end{tikzpicture}
    \quad=\quad
    \begin{tikzpicture}[anchorbase]
      \node at (0,0) [rectangle,draw,very thick] (d0) {$\beta$};
      \draw[blue,very thick] (d0) to (0,1)node[above] {$\rK$};
      \draw[very thick] (d0) to (0,-1)node[below] {$\rG$};
      \draw[dashed ,very thick] (.45, -1)[out=90, in=0]node[below] {$\ \rI_{\mi}$} to (0,-.5);
    \end{tikzpicture}
\end{equation}
for any $2$-morphisms $\alpha:\rF\to \rH$ and $\beta:\rG\to \rK$, respectively, where $\rF,\rH\in\cC(\mj,\mi)$
and $\rG,\rK\in\cC(\mi,\mk)$. For instance, if we let $\alpha$ be the map $r_{\mj,\mi,\rF}: \rF\rI_{\mj}\to\rF$, then the left diagram in \eqref{diag:unitornatural}
implies that
\begin{equation}\label{diag:unitornatural-1}\begin{tikzpicture}[anchorbase]
\draw[very thick] (4, 8)node[above] {$\rF$} to (4, 6)node[below] {$\rF$};
\draw[dashed , very thick] (3.3, 6)[out=90, in=-180]node[below] {$\ \rI_{\mi}$} to (4, 7.5);
      \draw[red,dashed , very thick] (4.5, 6)[out=90, in=0]node[below] {$\ \rI_{\mj}$} to (4, 6.9);
\end{tikzpicture}
\quad=\quad
\begin{tikzpicture}[anchorbase]
\draw[very thick] (4, 8)node[above] {$\rF$} to (4, 6)node[below] {$\rF$};
\draw[dashed , very thick] (3.5, 6)[out=90, in=-180]node[below] {$\ \rI_{\mi}$} to (4, 6.9);
      \draw[red, dashed , very thick] (4.7, 6)[out=90, in=0]node[below] {$\ \rI_{\mj}$} to (4, 7.5);
\end{tikzpicture}
\end{equation}

Similarly, by the naturality of the left and right oplax unitors $l'_{\mi}, r'_{\mi}$,
we have
\begin{equation}\label{diag:oplaxunitornatural}
    \begin{tikzpicture}[anchorbase]
      \node at (0,0) [rectangle,draw,very thick] (d0) {$\alpha$};
      \draw[blue,very thick] (d0) to (0,1); \draw[blue]  (0,1.09) node[above] {$\rH$};
      \draw[very thick] (d0) to (0,-1) node[below] {$\rF$};
      \draw[dashed , very thick] (-.9, 1)[out=-90,in=-180] node[above] {$\rI'_{\mi}$} to (0, -.55);
    \end{tikzpicture}
    \quad =\quad
    \begin{tikzpicture}[anchorbase]
      \node at (0,0) [rectangle,draw,very thick] (d0) {$\alpha$};
      \draw[blue,very thick] (d0) to (0,1); \draw[blue] (0,1.09)node[above] {$\rH$};
      \draw[very thick] (d0) to (0,-1)node[below] {$\rF$};
      \draw[dashed , very thick] (-.5, 1)[out=-90,in=-180] node[above] {$\rI'_{\mi}$} to (0, .55);
    \end{tikzpicture}
    \quad\text{and}\quad
        \begin{tikzpicture}[anchorbase]
      \node at (0,0) [rectangle,draw,very thick] (d0) {$\beta$};
      \draw[blue,very thick] (d0) to (0,1); \draw[blue]  (0,1.09)node[above] {$\rK$};
      \draw[very thick] (d0) to (0,-1)node[below] {$\rG$};
      \draw[dashed ,very thick] (.9, 1)[out=-90,in=0]node[above] {$\ \rI'_{\mi}$}
      to (0, -.55);
    \end{tikzpicture}
    \quad=\quad
    \begin{tikzpicture}[anchorbase]
      \node at (0,0) [rectangle,draw,very thick] (d0) {$\beta$};
      \draw[blue,very thick] (d0) to (0,1); \draw[blue]  (0,1.09)node[above] {$\rK$};
      \draw[very thick] (d0) to (0,-1)node[below] {$\rG$};
        \draw[dashed ,very thick] (.45, 1)[out=-90,in=0]node[above] {$\ \rI'_{\mi}$}
      to (0, .5);
    \end{tikzpicture}
\end{equation}
for any $2$-morphisms $\alpha:\rF\to \rH$ and $\beta:\rG\to \rK$, respectively, where $\rF,\rH\in\cC(\mj,\mi)$ and $\rG,\rK\in\cC(\mi,\mk)$.
In particular, we have the diagram dual to \eqref{diag:unitornatural-1} with dashed strands labeled by the corresponding oplax units.

\begin{rem}\label{rem:coherence}
The evaluations of the left unitor and the right unitor at $\rI_{\mi}$,
namely, $l_{\mi,\mi,\rI_{\mi}}:\rI_{\mi}\rI_{\mi}\to \rI_{\mi}$ and $r_{\mi,\mi,\rI_{\mi}}:\rI_{\mi}\rI_{\mi}\to \rI_{\mi}$,
differ in general, see Subsection~\ref{sec:example1} for example.
At the same time, via left (resp. right) horizontal composition with $\id_{\rF}$,
for some $1$-morphism $\rF$, the results are equalized
by the evaluation of the left (resp. right) unitor at $\rF$, see~\eqref{eq:002}.
\end{rem}

\begin{lem} If
 $\rI_{\mi}=(\rI_{\mi},\,l_{\mi},\,r_{\mi})$ and $\tilde{\rI}_{\mi}=(\tilde{\rI}_{\mi},\,\tilde{l}_{\mi},\,\tilde{r}_{\mi})$ are two lax units in $\cC(\mi,\mi)$,
 then $\rI_{\mi}\tilde{\rI}_{\mi}$ is also a lax unit in $\cC(\mi,\mi)$. The dual statement holds for oplax units.
\end{lem}

\begin{proof}
It is easy to check that the composition $\rI_{\mi}\tilde{\rI}_{\mi}$ has a natural structure of lax unit with the left lax unitor  given by the composition
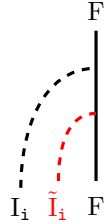
\[\xymatrix{\rI_{\mi}\tilde{\rI}_{\mi}\rF\ar[rr]^{\id_{\rI_{\mi}}\tilde{l}_{\mj,\mi,\rF}}&&
\rI_{\mi}\rF\ar[rr]^{l_{\mj,\mi,\rF}}&&\rF,}\] diagrammatically,
\[ \begin{tikzpicture}[anchorbase]
      \draw[very thick] (4, 8)node[above] {$\rF$} to (4, 6); \draw (4,5.91) node[below] {$\rF$};
      \draw[red, dashed,  very thick] (3.5, 6)[out=90, in=-180]node[below] {$\tilde{\rI}_{\mi}$} to (4, 6.9);
        \draw[dashed,very thick ] (3, 6)[out=90, in=-180](3,5.91)node[below] {$\rI_{\mi}$} to (4, 7.5) ;
    \end{tikzpicture}\]
for any $1$-morphism $\rF\in\cC(\mj,\mi)$. For the right unitor,
the structure is given by  the composition
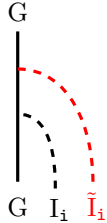
\[\xymatrix{\rG\rI_{\mi}\tilde{\rI}_{\mi}\ar[rr]^{r_{\mi,\mk,\rG}\id_{\tilde{\rI}_{\mi}}}&&
\rG\tilde{\rI}_{\mi}\ar[rr]^{\tilde{r}_{\mi,\mk,\rG}}&&\rG,}\]
diagrammatically,
\[\begin{tikzpicture}[anchorbase]
      \draw[very thick] (4, 8)node[above] {$\rG$} to (4, 6); \draw (4, 5.91)node[below] {$\rG$};
      \draw[dashed, very thick] (4.5, 6)[out=90, in=0] (4.5,5.91)node[below] {$\ \rI_{\mi}$}to (4, 6.9);
      \draw[red, dashed, very thick] (5, 6)[out=90, in=0]node[below] {$\ \tilde{\rI}_{\mi}$} to (4, 7.5);
    \end{tikzpicture}\]
for any $1$-morphism $\rG\in\cC(\mi,\mk)$.
\end{proof}
In particular, for any positive integer $n$, the composition $(\rI_{\mi})^n:=\underbrace{\rI_{\mi} \rI_{\mi}\cdots\rI_{\mi}}_{n\text{ factors}}$ is a lax unit
and $(\rI'_{\mi})^n:=\underbrace{\rI'_{\mi} \rI'_{\mi}\cdots\rI'_{\mi}}_{n\text{ factors}}$ is an oplax unit.

\begin{defn}\label{defxxx}
A {\em lax-unital} 2-category $(\cC,\,\rI=\{\rI_{\mi}|\,\mi\in\ob\cC\})$ is a 2-semicategory $\cC$ with a choice of triples $(\rI_{\mi},\,l_{\mi},\,r_{\mi})_{\mi\in\ob\ccC}$, each of which is either a lax unit or an oplax unit.

A {\em bilax-unital} 2-category $(\cC,\,\rI=\{\rI_{\mi}|\,\mi\in\ob\cC\},\,\rI'=\{\rI'_{\mi}|\,\mi\in\ob\cC\})$ is a 2-semicategory $\cC$ with a choice of lax units $(\rI_{\mi},\,l_{\mi},\,r_{\mi})_{\mi\in\ob\ccC}$ and
oplax units $(\rI'_{\mi},\,l'_{\mi},\,r'_{\mi})_{\mi\in\ob\ccC}$.
\end{defn}

For a bilax-unital 2-category $(\cC,\,\rI=\{\rI_{\mi}|\,\mi\in\ob\cC\},\,\rI'=\{\rI'_{\mi}|\,\mi\in\ob\cC\})$, a natural question to ask  is the following
\begin{ques}\label{ques-compatible}
Does the equality
\begin{equation}\label{laxoplaxcompatible}
r_{\mi,\mi,\rI'_{\mi}}\circ_{\rv} l'_{\mi,\mi,\rI_{\mi}}=l_{\mi,\mi,\rI'_{\mi}}\circ_{\rv} r'_{\mi,\mi,\rI_{\mi}},
\end{equation}
or, diagrammatically,
\begin{equation}\label{diag:laxoplaxcompatible}
\begin{tikzpicture}[anchorbase]
   \draw[dashed, very thick] (4,7.96)  to (4, 7) node[below] {$\ \rI_{\mi}$};
      \draw[dashed , very thick, red] (3.5, 7.95)[out=-90,in=-180] to (4, 7.5);
   \draw[dashed, very thick, red] (3.5, 9)node[above] {$\ \rI'_{\mi}$} to (3.5, 8.03);
      \draw[dashed ,very thick] (4, 8.05)[out=90, in=0] to (3.5, 8.5);
\end{tikzpicture}
\quad=\quad\begin{tikzpicture}[anchorbase]
 \draw[dashed, very thick] (4, 7.96) to (4, 7)node[below] {$\ \rI_{\mi}$};
      \draw[dashed, very thick, red] (4.5, 7.95)[out=-90,in=0]
      to (4, 7.5);
     \draw[dashed, very thick, red] (4.5, 8.03) to (4.5, 9)node[above] {$\ \rI'_{\mi}$};
      \draw[dashed , very thick] (4, 8.05)[out=90, in=-180] to (4.5, 8.5) ;
\end{tikzpicture}
\end{equation}
hold in $\cC$ ?
\end{ques}

A 2-category $\cC$ can be viewed as a bilax-unital 2-category $(\cC,\,\rI=\{\rI_{\mi}|\,\mi\in\ob\cC\},\,\rI'=\{\rI'_{\mi}|\,\mi\in\ob\cC\})$ by setting $\rI_{\mi}=\mathbbm{1}_{\mi}=\rI'_{\mi}$, for each $\mi$, where the $1$-morphism $\mathbbm{1}_{\mi}\in\cC(\mi,\mi)$ is the strict unit associated to $\mi$, and choosing all unitors to be the identities. We denote the associated bilax-unital 2-category by $(\cC,\,\rI=\rI'=\{\mathbbm{1}_{\mi}|\,\mi\in\ob\cC\})$. In this case, the identity \eqref{laxoplaxcompatible} trivializes since $\mathbbm{1}_{\mi}\mathbbm{1}_{\mi}=\mathbbm{1}_{\mi}$.
However, $\cC$ could also be viewed as a bilax-unital 2-category with a different choice of (op)lax units.
Indeed, any $2$-morphism $\gamma_{\mi}: \rI_{\mi}\to \mathbbm{1}_{\mi}$, where $\rI_{\mi}$ is some $1$-morphism in $\cC(\mi,\mi)$, makes $\rI_{\mi}$ into
a lax unit associated to the object $\mi$ by defining the left and right lax unitors as the left and right horizontal
composition with $\gamma_{\mi}$ respectively, that is,
\[l_{\mj,\mi,\rF}:=\gamma_{\mi} \id_{\rF}\quad\text{and}\quad r_{\mi,\mk,\rG}:=\id_{\rG} \gamma_{\mi},\]
for any $1$-morphisms $\rF\in\cC(\mj,\mi)$ and $\rG\in\cC(\mi,\mk)$. Similarly, any 2-morphism
$\gamma'_{\mi}:\mathbbm{1}_{\mi}\to \rI'_{\mi}$, where $\rI'_{\mi}$ is some $1$-morphism in $\cC(\mi,\mi)$, makes $\rI'_{\mi}$ into
an oplax unit associated to the object $\mi$ by defining the left and right oplax unitors as the left and right
horizontal composition with $\gamma'_{\mi}$, respectively.
If such $2$-morphisms $\gamma_{\mi},\gamma'_{\mi}$ are chosen for every object $\mi$, then
we denote by $\cC [\gamma,\gamma']$ the obtained bilax-unital 2-category. By the interchange law, i.e. \eqref{diag:sliding}, we have
$\gamma_{\mi}\gamma'_{\mi}=\gamma'_{\mi}\gamma_{\mi}=\gamma'_{\mi}\circ_{\rv}\gamma_{\mi}$, for any object $\mi$, and hence the identity \eqref{laxoplaxcompatible} holds in $\cC [\gamma,\gamma']$.

We will answer Question~\ref{ques-compatible} positively in Proposition~\ref{propblue} for certain bilax-unital 2-categories (called fiax; see Definition~\ref{deffiax}).
We suspect the answer to Question~\ref{ques-compatible} is negative in general, so it would be
interesting to find a counterexample.

\subsection{Abelianizations}\label{sec:abelianization}
Following \cite[Section~3]{MMMT}, we define the injective
(resp. projective) abelianization of 2-semicategories.
First recall the injective abelianization of an additive
category from \cite[Subsection~3.2]{MMMT}.

\begin{defn}
The {\em injective abelianization} $\underline{\mathcal A}$ of an
additive category $\mathcal A$ is defined as a category whose
\begin{itemize}
    \item objects are tuples $(\{f_i:X\to Y_i\}_{i=1}^\infty,\,n)$, where $n\in\mathbb N$, $X, Y_i$ are objects in $\mathcal A$ and $f_i:X\to Y_i$ is a morphism in $\mathcal A$ such that $Y_i=0$, for $i>n$;
    \item morphisms from $(\{f_i:X\to Y_i\}_{i=1}^\infty,n)$ to $(\{f'_i:X'\to Y'_i\}_{i=1}^\infty,n')$ are the homotopy classes of tuples $(g:X\to X',h_{i,j}:Y_i\to Y_j')_{i,j=1}^\infty$ such that $f'_i\circ g=\sum_j h_{j,i}\circ f_j$, where $\circ$ denotes the composition in $\mathcal A$ and the homotopy relation is spanned by the tuples $(g,h_{i,j})$ for which there exist $q_i:Y_i\to X'$ such that $\sum_i q_i\circ f_i=g$;
    \item the identity morphism for the object $(\{f_i:X\to Y_i\}_{i=1}^\infty,n)$ is given by $(\id_X , h_{i,j}: Y_i\to Y_j)_{i,j=1}^\infty$, where $h_{i,j}=\delta_{i,j}\id_{Y_i}$
    \item composition is given by $(g,h_{i,j})_{i,j=1}^\infty\circ (g',h'_{i,j})_{i,j=1}^\infty=(g\circ g',\sum_k h'_{k,j} \circ h_{i,k})_{i,j=1}^\infty$.
\end{itemize}
\end{defn}

Denote by $\Rep\mathcal A$ the category of all ($\Bbbk$-linear) additive functors from $\mathcal A$ to the category
$\operatorname{Vect}_{\Bbbk}$ of $\Bbbk$-vector spaces.
The category $\Rep\mathcal A$ is an abelian category (as $\operatorname{Vect}_{\Bbbk}$
is an abelian category).
The Yoneda embedding $X\mapsto \Hom_\mathcal{A}(X,{}_-)$ realizes $\mathcal A$ as a full subcategory of $(\Rep \mathcal A)\op$.
Let $L(\mathcal A)$ be the full subcategory of $(\Rep\mathcal A)\op$ which
consists of all objects $Z$ admitting a copresentation of the form
\[0\to Z\to \Hom_{\mathcal A}(X,{}_-)\to \Hom_{\mathcal A}(Y,{}_-) \]
with $X,Y\in\mathcal A$. By definition, the category $L(\mathcal A)$ coincides with the category $A(\mathcal A\op)\op$ from
\cite[Definition 5.1.3]{Nee} (applied to an additive category).

\begin{prop}\label{toneeman}
The categories $\underline{\mathcal A}$ and $L(\mathcal A)$ are equivalent.
\end{prop}
\begin{proof}
Define a functor from $\underline{\mathcal A}$ to $L(\mathcal A)$ by sending any object $(\{f_i:X\to Y_i\}_{i=1}^\infty,n)$ in $\underline{\mathcal A}$ to
$$\ker (\oplus_i \Hom(f_i,{}_-):\Hom(X,{}_-)\to \oplus_i \Hom(Y_i,{}_-))$$
in $L(\mathcal A)$ and with the obvious assignment on morphisms. It is an equivalence by \cite[Subsection~3.2]{MMMT} and the proof of \cite[Proposition 5.1.14]{Nee} using $L(\mathcal A)=A(\mathcal A\op)\op$.
\end{proof}

By Proposition~\ref{toneeman} and $L(\mathcal A)=A(\mathcal A\op)\op$, we can apply the proof of \cite[Lemma 5.1.6]{Nee} and obtain the following:
\begin{cor}\label{kernel}
The category $\underline{\mathcal A}$ has kernels.
\end{cor}

In general, the category $\underline{\mathcal A}$ does not have to be abelian. But if $\mathcal A$ is finitary, e.g. see Subsection~\ref{ss:fiax},
then $\underline{\mathcal A}$ is abelian. By using the definition of $\underline{\mathcal A}$ instead of that of $L(\mathcal A)$, we can define the injective abelianization of a $2$-semicategory $\cC$ which has a strictly associative horizontal composition.

\begin{defn}\label{injective-abelianization}
The {\em injective abelianization} $\underline{\cC}$ of a 2-semicategory $\cC$ is a 2-semicategory with
\begin{itemize}
    \item $\ob\underline{\cC}=\ob\cC$;
    \item $\underline{\cC}(\mi,\mj):=\underline{\cC (\mi,\mj)}$, for any $\mi,\mj\in\ob\cC$;
    \item composition of 1-morphisms given by
    \[(\{\alpha_i:\rF\to \rG_i\}_{i=1}^\infty, n)\circ (\{\alpha'_i:\rF'\to \rG'_i\}_{i=1}^\infty,n')=(\{\beta_i:\rF\rF'\to \rH_i\},n+n'),\]
    where
    \[\rH_i=\begin{cases}
        \rF\rG'_i, & \text{for } i=1,\cdots,n'\\
        \rG_{i-n'}\rF', & \text{for } i=n'+1,\cdots,n'+n\\
        0, & \text{else},
        \end{cases}\]
    and
    \[\beta_i=\begin{cases}
        \id_{\rF} \alpha'_i, & \text{for } i=1,\cdots, n'\\
        \alpha_{i-n'} \id_{\rF'} &\text{for } i=n+1,\cdots, n'+n\\
        0, &\textrm{else};
        \end{cases};\]
    \item componentwise horizontal composition of 2-morphisms.
\end{itemize}
\end{defn}
Note that the 2-semicategory $\cC$ can be embedded into its injective abelianization $\underline{\cC}$ by sending every $1$-morphism
$\rF\in\cC(\mi,\mj)$ to $(\{0:\rF\to 0\},0)$ and with the obvious assignment on $2$-morphisms.
If $\cC$ is a bilax-unital 2-category, i.e., $(\cC,\,\rI=\{\rI_{\mi}|\,\mi\in\ob\cC\},\,\rI'=\{\rI'_{\mi}|\,\mi\in\ob\cC\})$,
then its injective abelianization $\underline{\cC}$ is also a bilax-unital $2$-category with the images of the lax units $\rI_{\mi}$ under the above embedding
being the lax units and the images of the oplax units $\rI'_{\mi}$ under the above embedding being the oplax units, with the corresponding unitors.
However, the induced (op)lax units in $\underline{\cC}$ does not need to be split even if the original (op)lax units are split in $\cC$.

Dualizing the construction above, one can similarly define the projective abelianization $\overline{\cC}$ of $\cC$.

\subsection{Representations of bilax-unital 2-categories}\label{s2.3}
Consider any two bilax-unital 2-categories
\[\cC=(\cC,\,\rI^{\ccC}=\{\rI_{\mi}^{\ccC}|\,\mi\in\ob\cC\},\,(\rI')^{\ccC}=\{(\rI'_{\mi})^{\ccC}|\,\mi\in\ob\cC\}))\]
and
\[\cD=(\cD,\,\rI^{\ccD}=\{\rI_{\mi}^{\ccD}|\,\mi\in\ob\cD\},\,(\rI')^{\ccD}=\{(\rI'_{\mi})^{\ccD}|\,\mi\in\ob\cD\}).\]
Below we adapt to our setup the notion of a $2$-representation from
\cite[Subsection~2.3]{MM3}.

\begin{defn}\label{bilax-unital 2-functor}
A {\em bilax-unital 2-functor} $\bM$ from $\cC$ to $\cD$ consists of
\begin{itemize}
    \item a function $\bM:\ob \cC\to\ob\cD$;
    \item a functor $\bM_{\mi,\mj}:\cC(\mi,\mj)\to \cD(\bM(\mi),\bM(\mj))$, for each pair of objects $\mi,\mj\in \ob\cC$;
    \item two $2$-morphisms
    \[u_{\mi}:\bM_{\mi,\mi} (\rI_{\mi}^{\ccC})\to \rI_{\mi}^{\ccD}\quad\text{and}\quad u'_{\mi}:(\rI'_{\mi})^{\ccD}\to \bM_{\mi,\mi}\big((\rI'_{\mi})^{\ccC}\big),\]
    for each object $\mi\in\ob\cC$,
\end{itemize}
such that
\begin{enumerate}
    \item\label{bilax-unital 2-functor-1} for any two $1$-morphisms $\rF\in\cC(\mi,\mj),\rG\in\cC(\mj,\mk)$ in $\cC$, we have $\bM_{\mi,\mk}(\rG\rF)=\bM_{\mj,\mk}(\rG)\bM_{\mi,\mj}(\rF)$;
    \item\label{bilax-unital 2-functor-2} for any object $\mj\in\ob\cC$, we have
    \[l_{\mi,\mj,\bM_{\mi,\mj}(\rF)}^{\ccD}\circ_{\rv} (u_{\mj} \id_{\bM_{\mi,\mj}(\rF)})=\bM_{\mi,\mj}(l_{\mi,\mj,\rF}^{\ccC})
    \quad\text{and}\quad r_{\mj,\mk,\bM_{\mj,\mk}(\rG)}^{\ccD}\circ_{\rv}(\id_{\bM_{\mj,\mk}(\rG)} u_{\mj})=\bM_{\mj,\mk}(r_{\mj,\mk,\rG}^{\ccC}),\]
    for any $1$-morphisms $\rF\in\cC(\mi,\mj), \rG\in\cC(\mj,\mk)$;
    \item\label{bilax-unital 2-functor-3} for any object $\mj\in\ob\cC$, we have
    \[(u'_{\mj}\id_{\bM_{\mi,\mj}(\rF)})\circ_{\rv} (l'_{\mi,\mj,\bM_{\mi,\mj}(\rF)})^{\ccD}=\bM_{\mi,\mj}\big((l'_{\mi,\mj,\rF})^{\ccC}\big)\quad
    \text{and}\quad (\id_{\bM_{\mj,\mk}(\rG)}u'_{\mj})\circ_{\rv} (r'_{\mj,\mk,\bM_{\mj,\mk}(\rG)})^{\ccD}=\bM_{\mj,\mk}\big((r'_{\mj,\mk,\rG})^{\ccC}\big),\]
    for any $1$-morphisms $\rF\in\cC(\mi,\mj),\rG\in\cC(\mj,\mk)$.
\end{enumerate}
\end{defn}

Note that the embedding from $\cC$ to $\underline{\cC}$ is functorial. Hence, a bilax-unital $2$-functor from $\cC$ to $\cD$ induces a bilax-unital $2$-functor from $\underline{\cC}$ to $\underline{\cD}$.

Recall that any $2$-category can be viewed as a bilax-unital 2-categories with the lax and oplax units being the identity 1-morphisms, with the identity unitors. Denote by
$\mathbf{Cat}$ the 2-category of small categories, which is also a bilax-unital 2-category. Now we give the definition of
a representation of a bilax-unital 2-category $\cC=(\cC,\,\rI=\{\rI_{\mi}|\,\mi\in\ob\cC\},\,\rI'=\{\rI'_{\mi}|\,\mi\in\ob\cC\})$

\begin{defn}\label{def:bilax-2-rep}
A {\em bilax $2$-representation} $\bM$ of $\cC$ is a bilax-unital 2-functor from $\cC$ to $\mathbf{Cat}$, i,e, it consists of
\begin{itemize}
    \item a function $\bM: \ob \cC\to \ob \mathbf{Cat}$;
     \item a functor $\bM_{\mi,\mj}:\cC(\mi,\mj)\to \mathbf{Cat}(\bM(\mi),\bM(\mj))$, for each pair of objects $\mi,\mj\in \ob\cC$;
    \item two natural transformations
    \[u_{\mi}:\bM_{\mi,\mi} (\rI_{\mi})\to \id_{\bM(\mi)}\quad\text{and}\quad u'_{\mi}:\id_{\bM(\mi)}\to \bM_{\mi,\mi}(\rI'_{\mi}),\]
    for each object $\mi\in\ob\cC$,
\end{itemize}
satisfying the conditions of Definition~\ref{bilax-unital 2-functor}.
\end{defn}
The identities in \eqref{bilax-unital 2-functor-2} and \eqref{bilax-unital 2-functor-3} of Definition~\ref{bilax-unital 2-functor} can be written as
\begin{equation}\label{eq:bilax-2-rep}
\begin{split}
u_{\mj}\id_{\bM_{\mi,\mj}(\rF)}=\bM_{\mi,\mj}(l_{\mi,\mj,\rF}),\qquad\id_{\bM_{\mj,\mk}(\rG)}u_{\mj}=\bM_{\mj,\mk}(r_{\mj,\mk,\rG}),\\
u_{\mj}'\id_{\bM_{\mi,\mj}(\rF)}=\bM_{\mi,\mj}(l'_{\mi,\mj,\rF}),\qquad\id_{\bM_{\mj,\mk}(\rG)}u'_{\mj}=\bM_{\mj,\mk}(r'_{\mj,\mk,\rG}).
\end{split}
\end{equation}

\begin{defn}Let $\bM$ and $\bN$ be two bilax $2$-representations of $\cC$. A {\em bilax $2$-natural transformation} $\Phi:\bM\to\bN$ is given by
\begin{itemize}
\item a functor $\Phi_{\mi}:\bM(\mi)\to\bN(\mi)$, for each object $\mi\in\ob\cC$;
\item a natural isomorphism $\varphi_{\mi,\mj}({}_-):\bN_{\mi,\mj}({}_-)\Phi_{\mi}\to\Phi_{\mj}\bM_{\mi,\mj}({}_-)$, whose source and target are functors from $\cC(\mi,\mj)$ to $\mathbf{Cat}(\bM(\mi),\bN(\mj))$, for each pair of objects $\mi,\mj\in\ob\cC$;
\end{itemize}
such that the diagrams
\begin{gather}
\xymatrix@C=2.8pc{\bN_{\mj,\mk}(\rG)\bN_{\mi,\mj}(\rF)\Phi_{\mi}\ar@{=}[d]\ar[rr]^{\id_{\bN_{\mj,\mk}(\rG)}\varphi_{\mi,\mj}(\rF)}
&&\bN_{\mj,\mk}(\rG)\Phi_{\mj}\bM_{\mi,\mj}(\rF)\ar[rr]^{\varphi_{\mj,\mk}(\rG)\id_{\bM_{\mi,\mj}(\rF)}}&&\Phi_{\mk}\bM_{\mj,\mk}(\rG)\bM_{\mi,\mj}(\rF)\ar@{=}[d]\\
\bN_{\mi,\mk}(\rG\rF)\Phi_{\mi}\ar[rrrr]^{\varphi_{\mi,\mk}(\rG\rF)}&&&&\Phi_{\mk}\bM_{\mi,\mk}(\rG\rF)}\label{eq:0001}\\
\xymatrix@C=2.8pc{\bN_{\mi,\mi}(\rI_{\mi})\Phi_{\mi}\ar[rr]^{\varphi_{\mi,\mi}(\rI_{\mi})}\ar[d]_{u_{\mi}^{\bN}\id_{\Phi_{\mi}}}
&&\Phi_{\mi}\bM_{\mi,\mi}(\rI_{\mi})\ar[d]^{\id_{\Phi_{\mi}}u_{\mi}^{\bM}}\\
\id_{\bN(\mi)}\Phi_{\mi}\ar@{=}[r]&\Phi_{\mi}\ar@{=}[r]&\Phi_{\mi}\id_{\bM(\mi)}}\label{eq:0002}\\
\xymatrix@C=2.8pc{\id_{\bN(\mi)}\Phi_{\mi}\ar[d]_{(u'_{\mi})^{\bN}\id_{\Phi_{\mi}}}\ar@{=}[r]&
\Phi_{\mi}\ar@{=}[r]&\Phi_{\mi}\id_{\bM(\mi)}\ar[d]^{\id_{\Phi_{\mi}}(u'_{\mi})^{\bM}}\\
\bN_{\mi,\mi}(\rI'_{\mi})\Phi_{\mi}\ar[rr]^{\varphi_{\mi,\mi}(\rI'_{\mi})}
&&\Phi_{\mi}\bM_{\mi,\mi}(\rI'_{\mi})}\label{eq:0003}
\end{gather}
commute, for any objects $\mi,\mj,\mk\in\ob\cC$ and $1$-morphisms $\rF\in\cC(\mi,\mj),\rG\in\cC(\mj,\mk)$.

Moreover, if each $\Phi_{\mi}:\bM(\mi)\to\bN(\mi)$ is an equivalence of categories, for every object $\mi\in\ob\cC$,
then $\bM$ is said to be {\em equivalent} to $\bN$ as bilax $2$-representations of $\cC$.
\end{defn}

\begin{defn} Let $\Phi,\Psi:\bM\to\bN$ be two bilax $2$-natural transformations. A {\em bilax modification} $\chi:\Phi\to\Psi$ is given by
\begin{itemize}
\item a natural transformation $\chi_{\mi}:\Phi_{\mi}\to\Psi_{\mi}$ for each object $\mi\in\cC$,
\end{itemize}
such that the diagram
\[\xymatrix@C=2.8pc{\bN_{\mi,\mj}(\rF)\Phi_{\mi}\ar[d]_{\varphi_{\mi,\mj}(\rF)}\ar[rr]^{\id_{\bN_{\mi,\mj}(\rF)}\chi_{\mi}}&&
\bN_{\mi,\mj}(\rF)\Psi_{\mi}\ar[d]^{\psi_{\mi,\mj}(\rF)}\\
\Phi_{\mj}\bM_{\mi,\mj}(\rF)\ar[rr]^{\chi_{\mj}\id_{\bM_{\mi,\mj}(\rF)}}
&&\Psi_{\mj}\bM_{\mi,\mj}(\rF)}\]
commutes, for any objects $\mi,\mj\in\ob\cC$ and $1$-morphism $\rF\in\cC(\mi,\mj)$.
\end{defn}
Note that all bilax $2$-representations of a bilax-unital $2$-category, together with bilax $2$-natural transformations and bilax modifications forms a $2$-category.

For simplicity, we will omit all subscripts of functors $\bM_{\mi,\mj}$ associated to a bilax 2-representation $\bM$.

\begin{ex}
For any object $\mi$ in a bilax-unital 2-category $\cC$, the Yoneda functor $\cC(\mi,{}_-)$
is a bilax-unital 2-functor from $\cC$ to $\mathbf{Cat}$. Chosing $u_{\mi}$ and $u'_{\mi}$
to be the left (op)lax unitors $l$ and $l'$, respectively,  defines a bilax
$2$-representation of $\cC$. We call $\bP_{\mi}:=\cC(\mi,{}_-)$ the $\mi$-th principal
bilax $2$-representation. Dually, each functor $\cC({}_-,\mi)$ also defines a bilax $2$-representation of $\cC^{\,\mathrm{op},{}_-}$.
\end{ex}

For a bilax $2$-representation $\bM$ of a bilax-unital 2-category $\cC$,
we can define its injective abelianization $\underline{\bM}$ (resp. its
projective abelianization $\overline{\bM}$)
by letting $\underline{\bM}(\mi)=\underline{\bM(\mi)}$
(resp. $\overline{\bM}(\mi)=\overline{\bM(\mi)}$),
see \cite[Subsection~3.4]{MMMT} for more details. Then $\underline {\bM}$
(resp. $\overline {\bM}$) is a bilax $2$-representation of
$\underline{\cC}$ (resp. $\overline{\cC}$).

\subsection{Algebra and coalgebra 1-morphisms}\label{s2.4}
From now on, if not explcitly stated otherwise, we
let $$\cC=(\cC,\,\rI=\{\rI_{\mi}|\,\mi\in\ob\cC\},\,\rI'=\{\rI'_{\mi}|\,\mi\in\ob\cC\})$$
be a bilax-unital 2-category.

\begin{defn}\label{alg-coalg}
A 1-morphism $\rF\in\cC(\mi,\mi)$ is called an {\em algebra} if there exists a multiplication $2$-morphism $\mu_{\rF}:\rF\rF\to\rF$ and a unit $2$-morphism
$\eta_{\rF}:\rI_{\mi}\to\rF$, depicted as
\begin{equation*}
   \begin{tikzpicture}[anchorbase]
      \draw[very thick] (4, 8)node[above] {$\rF$} to (4, 7.5);
      \draw[very thick] (3.5, 6.8)[out=90, in=-180]node[below] {$\rF$} to (4, 7.5);
      \draw[very thick] (4.5, 6.8)[out=90, in=0]node[below] {$\rF$} to (4, 7.5);
    \end{tikzpicture}
    \quad\text{and}\quad
    \begin{tikzpicture}[anchorbase]
      \draw[very thick] (4, 8)node[above] {$\rF$} to (4, 7.5);
      \draw[dashed ,very thick] (4, 6.8)node[below]{$\ \rI_{\mi}$} to (4, 7.5);
    \end{tikzpicture}
\end{equation*}
such that $\mu_{\rF}$ is associative, i.e., $\mu_{\rF}\circ_{\rv}(\mu_{\rF}\id_{\rF})=\mu_{\rF}\circ_{\rv}(\id_{\rF}\mu_{\rF})$,
and we have the unitality $l_{\mi,\mi,\rF}=\mu_{\rF}\circ_{\rv} (\eta_{\rF} \id_{\rF})$ and  $r_{\mi,\mi,\rF}= \mu_{\rF}\circ_{\rv} (\id_{\rF}\eta_{\rF})$.
Diagrammatically, the associativity can be drawn as
\begin{equation}\label{diag:algassoc}
  \begin{tikzpicture}[anchorbase]
   \draw[very thick] (4.5, 8.5)node[above] {$\rF$} to (4.5, 8);
   \draw[very thick] (4,7.5)[out=90, in=-180] to (4.5, 8); \draw (4,7.7) node[above]{$\rF$};
      \draw[very thick] (3.6, 6.8)[out=90, in=-180]node[below] {$\rF$} to (4, 7.5);
      \draw[very thick] (4.4, 6.8)[out=90, in=0]node[below] {$\rF$} to (4, 7.5);
      \draw[very thick] (5, 6.8)[out=90, in=0] node[below] {$\rF$} to (4.5, 8);
    \end{tikzpicture}
\quad=\quad
  \begin{tikzpicture}[anchorbase]
   \draw[very thick] (3.5, 8.5)node[above] {$\rF$} to (3.5, 8);
   \draw[very thick] (4,7.5)[out=90, in=0] to (3.5, 8); \draw (4.1,7.7)node[above] {$\rF$};
      \draw[very thick] (3.6, 6.8)[out=90, in=-180]node[below] {$\rF$} to (4, 7.5);
      \draw[very thick] (4.4, 6.8)[out=90, in=0]node[below] {$\rF$} to (4, 7.5);
      \draw[very thick] (3, 6.8)[out=90, in=-180] node[below] {$\rF$} to (3.5, 8);
    \end{tikzpicture}
\end{equation}
and the unitality can be drawn as follows:
\begin{equation}\label{diag:algunit}
\begin{tikzpicture}[anchorbase]
\draw[very thick] (4, 8.1)node[above] {$\rF$} to (4, 6.4)node[below] {$\rF$};
\draw[dashed ,very thick] (3.3, 6.4)[out=90, in=-180]node[below] {$\ \rI_{\mi}$} to (4, 7.4);
\end{tikzpicture}
\quad=\quad
    \begin{tikzpicture}[anchorbase]
     \draw[very thick] (3.5, 8.1)node[above] {$\rF$} to (3.5, 7.5);
   \draw[very thick] (4,7) [out=90, in=0] to (3.5, 7.5); \draw (3,7.2)node[above] {$\rF$};
    \draw[very thick] (3, 7)[out=90, in=-180]  to (3.5, 7.5);
    \draw[very thick] (4, 6.4)node[below]{$\rF$} to (4, 7);
      \draw[dashed ,very thick] (3, 6.4)node[below]{$\ \rI_{\mi}$} to (3, 7);
    \end{tikzpicture}
        \quad\text{and}\quad
\begin{tikzpicture}[anchorbase]
\draw[very thick] (4, 8.1)node[above] {$\rF$} to (4, 6.4)node[below] {$\rF$};
      \draw[dashed ,very thick] (4.7, 6.4)[out=90, in=0]node[below] {$\ \rI_{\mi}$} to (4, 7.4);
    \end{tikzpicture}
    \quad=\quad
    \begin{tikzpicture}[anchorbase]
     \draw[very thick] (3.5, 8.1)node[above] {$\rF$} to (3.5, 7.5);
   \draw[very thick] (4,7)[out=90, in=0] to (3.5, 7.5); \draw (4.1,7.2)node[above] {$\rF$};
      \draw[very thick] (3, 7)[out=90, in=-180]  to (3.5, 7.5);
       \draw[very thick] (3, 6.4)node[below]{$\rF$} to (3, 7);
      \draw[dashed ,very thick] (4, 6.4)node[below]{$\ \rI_{\mi}$} to (4, 7);
    \end{tikzpicture}
\end{equation}
For any two algebra $1$-morphisms $\rF:=(\rF,\,\mu_{\rF},\,\eta_{\rF})$ and $\rG:=(\rG,\,\mu_{\rG},\,\eta_{\rG})$ in $\cC(\mi,\mi)$, an {\em algebra homomorphism} from $\rF$ to $\rG$ is a $2$-morphism $\alpha:\rF\to\rG$ such that
$\mu_{\rG}\circ_{\rv}(\alpha\alpha)=\alpha\circ_{\rv}\mu_{\rF}$ and $\eta_{\rG}=\alpha\circ_{\rv}\eta_{\rF}$, that is, we have the following diagrams:
\begin{equation}\label{diag:alghomo}
 \begin{tikzpicture}[anchorbase]
   \draw[blue,very thick] (-.5, 1.5)node[above] {$\rG$} to (-.5, 1);
      \draw[blue,very thick] (-1, .5)[out=90, in=-180] to (-.5, 1);
      \draw[blue,very thick] (0, .5)[out=90, in=0] to (-.5, 1);
      \node at (0,0) [rectangle,draw,very thick] (d0) {$\alpha$};
      \draw[blue,very thick] (d0) to (0,.5) (0.2,.5)node[above] {$\rG$};
      \draw[very thick] (d0) to (0,-1) node[below] {$\rF$};
       \node at (-1,0) [rectangle,draw,very thick] (d0) {$\alpha$};
      \draw[blue,very thick] (d0) to (-1,.5)(-1.2,.5)node[above] {$\rG$};
      \draw[very thick] (d0) to (-1,-1) node[below] {$\rF$};
    \end{tikzpicture}\quad=\quad
 \begin{tikzpicture}[anchorbase]
 \tikzstyle{box} = [minimum height=0.4cm, draw,rectangle]
   \draw[very thick] (-.5, 1.6)node[above, black, box]{$\alpha$} to (-.5, .9);
      \draw[very thick] (-1, .1)node[below]{$\rF$}[out=90, in=-180] to (-.5, .9); \draw (-.3, 1)node[above]{$\rF$};
      \draw[very thick] (0, .1)node[below]{$\rF$} [out=90, in=0] to (-.5, .9);
      \draw[blue,very thick] (-.5,2) to (-.5, 2.6)node[above] {$\rG$};
    \end{tikzpicture}
    \quad\text{and}\quad
\begin{tikzpicture}[anchorbase]
 \tikzstyle{box} = [minimum height=0.4cm, draw,rectangle]
\draw[very thick] (4, 8.1)node[above, black, box]{$\alpha$} to (4, 7.5); \draw (4.2,7.5)node[above]{$\rF$};
      \draw[dashed ,very thick] (4, 6.8)node[below]{$\ \rI_{\mi}$} to (4, 7.5);
      \draw[blue, very thick] (4,8.5) to (4, 9.3)node[above] {$\rG$};
    \end{tikzpicture}\quad=\quad
 \begin{tikzpicture}[anchorbase]
 \draw[blue, very thick] (4, 8.5)node[above] {$\rG$} to (4, 7.2);
      \draw[dashed ,very thick] (4, 6)node[below]{$\ \rI_{\mi}$} to (4, 7.2);
    \end{tikzpicture}
\end{equation}

A 1-morphism $\rF\in\cC(\mi,\mi)$ is called a {\em coalgebra} if there exists a comultiplication $2$-morphism $\Delta_{\rF}:\rF\to\rF\rF$ and
a counit $2$-morphism $\epsilon_{\rF}:\rF\to\rI'_{\mi}$, depicted as
\begin{equation*}
   \begin{tikzpicture}[anchorbase]
   \draw[very thick] (4, 7.5) to (4, 7) node[below] {$\rF$};
      \draw[very thick] (3.5, 8.2)[out=-90,in=-180] node[above] {$\rF$} to (4, 7.5);
       \draw[very thick] (4.5,8.2)[out=-90,in=0]node[above] {$\rF$} to (4,7.5);
    \end{tikzpicture}
    \quad\text{and}\quad
    \begin{tikzpicture}[anchorbase]
      \draw[dashed , very thick] (4, 8.2)node[above] {$\ \rI'_{\mi}$} to (4, 7.5);
      \draw[very thick] (4, 7)node[below]{$\rF$} to (4, 7.5);
    \end{tikzpicture}
\end{equation*}
satisfying both coassociativity  $(\Delta_{\rF}\id_{\rF})\circ_{\rv}\Delta_{\rF}=(\id_{\rF}\Delta_{\rF})\circ_{\rv}\Delta_{\rF}$
and counitality $$l'_{\mi,\mi,\rF}=(\epsilon_{\rF} \id_{\rF})\circ_{\rv}\Delta_{\rF}
\quad\text{ and }\quad r'_{\mi,\mi,\rF}=(\id_{\rF}\epsilon_{\rF}) \circ_{\rv}\Delta_{\rF}.$$
The diagrams for coassociativity and counitality are obtained by flipping
Diagrams~\eqref{diag:algassoc} and ~\eqref{diag:algunit} with the labelling of
the dashed  strand replaced by $\rI'_{\mi}$ from $\rI_{\mi}$.
Homomorphisms of coalgebras $1$-morphisms is defined in the obvious way.
\end{defn}

For a $2$-category $\cC$, the usual definition of algebra and coalgebra $1$-morphisms
is obtained from Definition~\ref{alg-coalg} by letting $\rI_{\mi}=\mathbbm{1}_{\mi}=\rI'_{\mi}$ and all unitors to be the identities. Recall that, for a $2$-category $\cC$ with any choice of $\{\gamma_{\mi}: \rI_{\mi}\to \mathbbm{1}_{\mi}|\,\mi\in\ob\cC\}$ and $\{\gamma'_{\mi}:\mathbbm{1}_{\mi}\to\rI'_{\mi}|\,\mi\in\ob\cC\}$, we have the bilax-unital 2-category $\cC[\gamma,\gamma']$.

\begin{lem}\label{algebraisalgebra}
Let $\cC$ be a 2-category and $\rF\in\cC(\mi,\mi)$ be an algebra (resp. coalgebra) $1$-morphism. Then the 1-morphism $\rF$ is also an algebra (resp. coalgebra) in $\cC[\gamma,\gamma']$.
\end{lem}

\begin{proof}
We only prove the statement for algebras. For coalgebras the proof is similar.
Assume that the structure maps for the algebra $1$-morphism $\rF\in\cC(\mi,\mi)$ are $\mu_{\rF}:\rF\rF\to \rF$ and $\eta_{\rF}:\mathbbm{1}_{\mi}\to \rF$.
Then we claim that $\rF$ is also an algebra in $\cC[\gamma,\gamma']$ with the multiplication $\mu_{\rF}$ and
the unit $\eta_{\rF}\circ_{\rv} \gamma_{\mi}:\rI_{\mi}\to\rF$. The associativity of $\mu_{\rF}$ is clear.
The unitality of $\rF$ in $\cC[\gamma,\gamma']$ follows from the unitality of $\rF$ in
$\cC$ and the fact that the left and right unitors in $\cC[\gamma,\gamma']$ are given by the left and right horizontal composition with the $\gamma_{\mi}$'s.
\end{proof}

\subsection{(Co)modules over (co)algebra $1$-morphisms} \label{s2.5}
Let $\rF=(\rF,\,\mu_{\rF},\,\eta_{\rF})\in\cC(\mi,\mi)$ be an algebra $1$-mor\-phism.

\begin{defn}
A {\em left module} over $\rF$ in $\cC$ is a $1$-morphism $\rM\in\cC(\mj,\mi)$,
for some $\mj$, together with a 2-morphism $\upsilon_{\rM}:\rF\rM\to\rM$, depicted as:
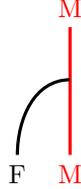
\[\begin{tikzpicture}[anchorbase]
\draw[red, very thick] (4, 8.1)node[above] {$\rM$} to (4, 6.4)node[below] {$\rM$};
\draw[very thick] (3.3, 6.4)[out=90, in=-180]node[below] {$\rF$} to (4, 7.4);
\end{tikzpicture}
\]
satisfying the following associativity and unitality axioms:
\begin{equation}\label{diag:moduleassoc-unit}
\begin{tikzpicture}[anchorbase]
\draw[red, very thick] (4, 8.1)node[above] {$\rM$} to (4, 6)node[below] {$\rM$};
\draw[very thick] (3, 6)[out=90, in=-180]node[below] {$\rF$} to (4, 7.5);
\draw[very thick] (3.5, 6)[out=90, in=-180]node[below] {$\rF$} to (4, 7);
\end{tikzpicture}
\quad=\quad
\begin{tikzpicture}[anchorbase]
      \draw[very thick] (3, 6)[out=90, in=-180]node[below] {$\rF$} to (3.3, 6.75);
      \draw[very thick] (3.6, 6)[out=90, in=0]node[below] {$\rF$} to (3.3, 6.75);
\draw[red, very thick] (4, 8.1)node[above] {$\rM$} to (4, 6)node[below] {$\rM$};
\draw[very thick] (3.3, 6.75)[out=90, in=-180] to (4, 7.5)(3.5,7.3) node[above]{$\rF$};
\end{tikzpicture}\quad\text{and}\quad
\begin{tikzpicture}[anchorbase]
 \draw[red, very thick] (4, 8.1)node[above] {$\rM$} to (4, 6)node[below] {$\rM$};
\draw[dashed , very thick] (3.4, 6)[out=90, in=-180]node[below] {$\ \rI_{\mi}$} to (4, 7.3);
\end{tikzpicture}
\quad=\quad
\begin{tikzpicture}[anchorbase]
 \draw[red, very thick] (4, 8.1)node[above] {$\rM$} to (4, 6)node[below] {$\rM$};
\draw[very thick] (3.4, 6.7)[out=90, in=-180] to (4, 7.4); \draw (3.5,7.2) node[above]{$\rF$};
\draw[dashed  ,very thick] (3.4, 6)node[below]{$\ \rI_{\mi}$} to (3.4, 6.8);
\end{tikzpicture}
\end{equation}
For any two left $\rF$-modules $\rM=(\rM,\upsilon_{\rM})$ and $\rN=(\rN,\upsilon_{\rN})$, an
{\em $\rF$-module homomorphism} from $\rM$ to $\rN$ is a
$2$-morphism $\alpha:\rM\to\rN$ such that:
\begin{equation}\label{diag:mod-homo}
\begin{tikzpicture}[anchorbase]
 \tikzstyle{box} = [minimum height=0.4cm, draw,rectangle]
      \draw[blue,very thick] (0,.3) to (0,1) node[above] {$\rN$};
      \draw[red, very thick] (0,-.1)node[above, black, box]{$\alpha$} to (0,-1)node[below] {$\rM$};
      \draw[very thick] (-.6, -1)[out=90, in=-180]node[below] {$\rF$} to (0, -.44) ;
    \end{tikzpicture}
    \quad =\quad
     \begin{tikzpicture}[anchorbase]
 \tikzstyle{box} = [minimum height=0.4cm, draw,rectangle]
      \draw[blue,very thick] (0,.2) to (0,1) node[above] {$\rN$};
      \draw[red, very thick] (0,-.2)node[above, black, box]{$\alpha$} to (0,-1)node[below] {$\rM$};
      \draw[very thick] (-.7, -1)[out=90, in=-180]node[below] {$\rF$} to (0, .49) ;
    \end{tikzpicture}
\end{equation}
\end{defn}
One can define {\em right $\rF$-modules} and the corresponding {\em $\rF$-module homomorphisms}
similarly using the mirror images of  Diagrams~\eqref{diag:moduleassoc-unit} and \eqref{diag:mod-homo}
with respect to a vertical mirror.

Let $\rG=(\rG,\,\mu_{\rG},\,\eta_{\rG})\in\cC(\mj,\mj)$ be another algebra $1$-morphism.
\begin{defn}
An  {\em $\rF$-$\rG$-bimodule} in $\cC$ is a $1$-morphism $\rM\in\cC(\mj,\mi)$ with two 2-morphisms
\[\upsilon_{\rM}:\rF\rM\to\rM\quad\text{and}\quad\tau_{\rM}:\rM\rG\to\rM,\] such that $(\rM,\upsilon_{\rM})$ is a left $\rF$-module and $(\rM,\tau_{\rM})$ is a right $\rG$-module which is compatible in the following sense:
\begin{equation}\label{diag:bimod}\begin{tikzpicture}[anchorbase]
\draw[red, very thick] (4, 8.1)node[above] {$\rM$} to (4, 6.4)node[below] {$\rM$};
\draw[very thick] (3.5, 6.4)[out=90, in=-180]node[below] {$\rF$} to (4, 7.4);
      \draw[teal,very thick] (4.5, 6.4)[out=90, in=0]node[below] {$\rG$} to (4, 7);
\end{tikzpicture}
\quad=\quad
\begin{tikzpicture}[anchorbase]
\draw[red, very thick] (4, 8.1)node[above] {$\rM$} to (4, 6.4)node[below] {$\rM$};
\draw[very thick] (3.5, 6.4)[out=90, in=-180]node[below] {$\rF$} to (4, 7);
      \draw[teal,very thick] (4.5, 6.4)[out=90, in=0]node[below] {$\rG$} to (4, 7.4);
\end{tikzpicture}
\end{equation}
We can simply denote the two above diagrams as follows:
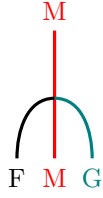
\[\begin{tikzpicture}[anchorbase]
\draw[red, very thick] (4, 8.1)node[above] {$\rM$} to (4, 6.4)node[below] {$\rM$};
\draw[very thick] (3.5, 6.4)[out=90, in=-180]node[below] {$\rF$} to (4, 7.2);
      \draw[teal, very thick] (4.5, 6.4)[out=90, in=0]node[below] {$\rG$} to (4, 7.2);
\end{tikzpicture}\]

For any two $\rF$-$\rG$-bimodules $\rM=(\rM,\upsilon_{\rM},\tau_{\rM})$ and
$\rN=(\rN,\upsilon_{\rN},\tau_{\rN})$, a {\em $\rF$-$\rG$-bimodule homomorphism} from $\rM$ to $\rN$ is a
$2$-morphism $\alpha:\rM\to\rN$ such that we have:
\begin{equation}\label{diag:bimod-homo}
\begin{tikzpicture}[anchorbase]
 \tikzstyle{box} = [minimum height=0.4cm, draw,rectangle]
      \draw[blue,very thick] (0,.3) to (0,1) node[above] {$\rN$};
      \draw[red, very thick] (0,-.1)node[above, black, box]{$\alpha$} to (0,-1)node[below] {$\rM$};
      \draw[very thick] (-.6, -1)[out=90, in=-180]node[below] {$\rF$} to (0, -.44) ;
         \draw[teal,very thick] (.6, -1)[out=90, in=0]node[below] {$\rG$} to (0, -.44);
    \end{tikzpicture}
    \quad =\quad
     \begin{tikzpicture}[anchorbase]
 \tikzstyle{box} = [minimum height=0.4cm, draw,rectangle]
      \draw[blue,very thick] (0,.2) to (0,1) node[above] {$\rN$};
      \draw[red, very thick] (0,-.2)node[above, black, box]{$\alpha$} to (0,-1)node[below] {$\rM$};
      \draw[very thick] (-.6, -1)[out=90, in=-180]node[below] {$\rF$} to (0, .49) ;
      \draw[teal,very thick] (.6, -1)[out=90, in=0]node[below] {$\rG$} to (0, .49);
    \end{tikzpicture}
\end{equation}

Dually, one can define a left (right or bi-) comodule over a coalgebra and homomorphisms of left (right or bi-) comodules by reversing the arrows of morphisms and fliping all diagrams above with the dashed strand labeled by the oplax unit $\rI'_{\mi}$ instead of the lax unit $\rI_{\mi}$.
\end{defn}

For any algebra $1$-morphism $\rF\in\cC(\mi,\mi)$ and every object $\mj\in\ob\cC$, we
denote by $\rF\text{-}\mathrm{mod}_{\ccC}(\mj)$ (resp. $\big(\mathrm{mod}_{\ccC}\text{-}\rF\big)(\mj)$),
the (additive and $\Bbbk$-linear) category of left (resp. right) $\rF$-modules in $\cC(\mj,\mi)$
(resp. $\cC(\mi,\mj)$) and with the corresponding $\rF$-modules homomorphisms as morphisms.
With such notation, $\rF\text{-}\mathrm{mod}_{\ccC}:\cC\to\mathbf{Cat}$ (resp. $\mathrm{mod}_{\ccC}\text{-}\rF:\cC\to\mathbf{Cat}$) defines a bilax $2$-representations of $\cC$ where the action of $1$-morphisms is given by right (resp. left) composition with the corresponding $1$-morphisms and the action of $2$-morphisms is given by right (resp. left) horizontal composition with the corresponding $2$-morphisms.
Indeed, for the bilax $2$-representation $\rF\text{-}\mathrm{mod}_{\ccC}$ (resp. $\mathrm{mod}_{\ccC}\text{-}\rF$),
we have $u_{\mi}=r_{\mi}$ (resp. $u_{\mi}=l_{\mi}$) and $u'_{\mi}=r'_{\mi}$ (resp. $u'_{\mi}=l'_{\mi}$).

Similarly, one can define $\rF\text{-}\comod_{\ccC}$ and
$\comod_{\ccC}\text{-}\rF$, for a coalgebra $1$-morphism $\rF$, and, for two (co)algebra $1$-morphisms $\rF\in\cC(\mi,\mi)$ and $\rG\in\cC(\mj,\mj)$, the category of $\rF$-$\rG$-bi(co)modules in $\cC(\mj,\mi)$.
As above, $\rF\text{-}\comod_{\ccC}$ and
$\comod_{\ccC}\text{-}\rF$ give bilax $2$-representations of $\cC$.

\subsection{Adjoint pairs of $1$-morphisms}\label{s2.6}

Recall that $\cC=(\cC,\,\rI=\{\rI_{\mi}|\,\mi\in\ob\cC\},\,\rI'=\{\rI'_{\mi}|\,\mi\in\ob\cC\})$
is a bilax-unital 2-category.

\begin{defn}\label{def:ad-pair}
For any $\rF\in\cC(\mi,\mj)$ and $\rG\in\cC(\mj,\mi)$, we say that $(\rF,\rG)$ is an {\em adjoint pair}
of $1$-morphisms (or, alternatively, that $\rF$ is a {\em left adjoint} of $\rG$ and
$\rG$ is a {\em right adjoint} of $\rF$), if there exist two $2$-morphisms
\begin{equation*}\label{alphabet}
    \alpha:\rI'_{\mi}\to \rG\rF\quad\text{and}\quad
    \beta:\rF\rG \to \rI_{\mj},
\end{equation*}
such that the composite
\begin{equation*}\label{eq:defadjoint1}
    \xymatrix@C=1.2pc{\rF\ar[rr]^{r'_{\mi,\mj,\rF}}&& \rF\rI'_{\mi}\ar[rr]^{\id_{\rF}\alpha}&&\rF\rG\rF\ar[rr]^{\beta\id_{\rF}}&& \rI_{\mj}\rF\ar[rr]^{l_{\mi,\mj,\rF}}&& \rF}
\end{equation*}
equals $\id_{\rF}$ and the composite
\begin{equation*}\label{eq:defadjoint2}
\xymatrix@C=1.2pc{\rG\ar[rr]^{l'_{\mj,\mi,\rG}}&& \rI'_{\mi}\rG\ar[rr]^{\alpha\id_{\rG}}&&\rG\rF\rG\ar[rr]^{\id_{\rG}\beta}&& \rG\rI_{\mj}\ar[rr]^{r_{\mj,\mi,\rG}}&& \rG}
\end{equation*}
equals $\id_{\rG}$.
\end{defn}

Let us depict $\alpha$ and $\beta$ by the following diagrams:
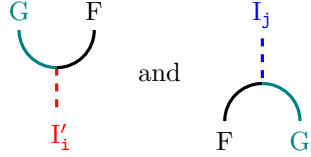
\begin{equation*}
\begin{tikzpicture}[anchorbase]
   \draw[dashed , red, very thick] (4, 7.5) to (4, 6.9) node[below] {$\ \rI'_{\mi}$};
      \draw[teal, very thick] (3.5, 8)[out=-90,in=-180] node[above] {$\rG$} to (4, 7.5);
       \draw[very thick] (4.5,8)[out=-90,in=0]node[above] {$\rF$} to (4,7.5);
\end{tikzpicture}
\quad\text{and}\quad
  \begin{tikzpicture}[anchorbase]
      \draw[dashed , blue, very thick] (4, 8.1) node[above] {$\rI_{\mj}$} to (4, 7.5);
      \draw[very thick] (3.5, 7)[out=90, in=-180]node[below] {$\rF$} to (4, 7.5);
      \draw[teal, very thick] (4.5, 7)[out=90, in=0]node[below] {$\rG$} to (4, 7.5);
    \end{tikzpicture}
\end{equation*}
Then the diagrams for the two conditions in Definition~\ref{def:ad-pair} are as follows:
\begin{equation}\label{diag:adjunction conditions}
\begin{tikzpicture}[anchorbase,scale=.7]
      \draw[very thick] (3, 7.5) to (3, 5.3)node[below] {$\rF$};
      \draw[dashed  ,red, very thick] (4.5, 7)[out=-90,in=0]
      to (3, 6) ;
      \draw (4.2,6.3) node[below] {$\ \rI'_{\mi}$};
          \draw[very thick] (3, 7.5)[out=90, in=-180] to (3.5, 8);
      \draw[teal, very thick] (4, 7.5)[out=90, in=0] to (3.5, 8);
      \draw[teal, very thick] (4, 7.5)[out=-90,in=-180]  to (4.5, 7); \draw (4.2,7.5)node[above] {$\rG$};
             \draw[very thick] (5,7.5)[out=-90,in=0] to (4.5,7);
                \draw[very thick] (5,7.5) to (5,9.7) node[above] {$\rF$};
                \draw[dashed  , blue, very thick] (3.5, 8)[out=90, in=-180] to (5, 9);
                 \draw (4.2,8.9)node[above] {$\rI_{\mj}$};
\end{tikzpicture}
\quad=\quad
\begin{tikzpicture}[anchorbase,scale=.7]
      \draw[very thick] (3, 9.7)node[above] {$\rF$} to (3, 5.3)node[below] {$\rF$};
    \end{tikzpicture}\quad\text{and}\quad
\begin{tikzpicture}[anchorbase,scale=.7]
      \draw[teal,very thick] (5, 7.5) to (5, 5.3)node[below] {$\rG$};
      \draw[dashed  ,red, very thick] (3.5, 7)[out=-90,in=-180]
      to (5, 6);\draw (3.9,6.2)node[below] {$\ \rI'_{\mi}$};
          \draw[very thick] (4, 7.5)[out=90, in=-180] to (4.5, 8);
      \draw[teal, very thick] (5, 7.5)[out=90, in=0] to (4.5, 8);
      \draw[teal, very thick] (3, 7.5)[out=-90,in=-180]  to (3.5, 7);
             \draw[very thick] (4,7.5)[out=-90,in=0] to (3.5,7); \draw (3.9,7.5)node[above] {$\rF$};
                \draw[teal, very thick] (3,7.5) to (3,9.7) node[above] {$\rG$};
               \draw[dashed  , blue, very thick] (4.5, 8)[out=90, in=0] to (3, 9); \draw (3.9,8.8)node[above] {$\rI_{\mj}$};
    \end{tikzpicture}
\quad=\quad
\begin{tikzpicture}[anchorbase,scale=.7]
      \draw[teal,very thick] (3, 9.7)node[above] {$\rG$} to (3, 5.3)node[below] {$\rG$};
    \end{tikzpicture}
\end{equation}

Let us now establish some basic properties of adjuctions.

\begin{prop}\label{adjunction and horizontal}
Let $\rF\in\cC(\mi,\mj)$, $\rG\in\cC(\mj,\mi)$, $\rF'\in\cC(\mk,\mi)$ and $\rG'\in\cC(\mi,\mk)$.
If $(\rF,\rG)$ and $(\rF',\rG')$ are two adjoint pairs in $\cC$, then $(\rF\rF',\rG'\rG)$ is also an adjoint pair in $\cC$.
\end{prop}
\begin{proof}
We define two $2$-morphisms:
   \[ \tilde{\alpha}:\rI'_{\mk}\to \rG'\rG\rF\rF'\quad\text{and}\quad
    \tilde{\beta}:\rF\rF'\rG'\rG \to \rI_{\mj},\]
    respectively, as follows:
\begin{equation}\label{eq:01}
\tilde{\alpha}:\quad
\begin{tikzpicture}[anchorbase]
   \draw[dashed  , brown!70!red, very thick] (4, 7.7) to (4, 7) node[below] {$\ \rI'_{\mk}$};
      \draw[blue!45, very thick] (2.8, 9)[out=-90,in=-180] node[above] {$\rG'$} to (4, 7.7);
       \draw[green!50!orange, very thick] (5.2,9)[out=-90,in=0]node[above] {$\rF'$} to (4,7.7);
       \draw[dashed  , red, very thick] (4, 8.5)[out=-90,in=0] to (3.3, 7.9); \draw (4.2, 7.74)node[above,red] {$\ \rI'_{\mi}$};
      \draw[teal, very thick] (3.5, 9)[out=-90,in=-180] node[above]{$\rG$} to (4, 8.5);
       \draw[very thick] (4.5, 9)[out=-90,in=0]node[above] {$\rF$} to (4,8.5);
\end{tikzpicture}
\stackrel{\eqref{diag:oplaxlrcompatible}}{=\joinrel=\joinrel=}
\begin{tikzpicture}[anchorbase]
   \draw[dashed  , brown!70!red, very thick] (4, 7.7) to (4, 7) node[below] {$\ \rI'_{\mk}$};
      \draw[blue!45, very thick] (2.8, 9)[out=-90,in=-180] node[above] {$\rG'$} to (4, 7.7);
       \draw[green!50!orange, very thick] (5.2,9)[out=-90,in=0]node[above] {$\rF'$} to (4,7.7);
       \draw[dashed  , red, very thick] (4, 8.5)[out=-90,in=-180] to (4.7, 7.9); \draw  (3.8, 7.74) node[above,red] {$\ \rI'_{\mi}$};
      \draw[teal, very thick] (3.5, 9)[out=-90,in=-180] node[above] {$\rG$} to (4, 8.5);
       \draw[very thick] (4.5, 9)[out=-90,in=0]node[above] {$\rF$} to (4,8.5);
\end{tikzpicture}
\end{equation}
and
\begin{equation}\label{eq:02}
\tilde{\beta}:\quad
\begin{tikzpicture}[anchorbase]
   \draw[dashed  , blue, very thick] (4, 8.3) to (4, 9) node[above] {$\ \rI_{\mj}$};
      \draw[very thick] (2.8, 7)[out=90, in=-180]node[below] {$\rF$} to (4, 8.3);
       \draw[teal, very thick] (5.2, 7)[out=90, in=0]node[below] {$\rG$} to (4,8.3);
       \draw[dashed  , orange!80, very thick] (4, 7.5)[out=90, in=0] to (3.3, 8.1) ; \draw (4.3, 8.2)node[below,orange] {$\rI_{\mi}$};
      \draw[green!50!orange,very thick] (3.5, 7)[out=90, in=-180]node[below] {$\rF'$} to (4, 7.5);
      \draw[blue!45, very thick] (4.5, 7)[out=90, in=0]node[below] {$\rG'$} to (4, 7.5);
\end{tikzpicture}
\stackrel{\eqref{diag:lrcompatible}}{=\joinrel=\joinrel=}
\begin{tikzpicture}[anchorbase]
   \draw[dashed  , blue, very thick] (4, 8.3) to (4, 9) node[above] {$\ \rI_{\mj}$};
      \draw[very thick] (2.8, 7)[out=90, in=-180]node[below] {$\rF$} to (4, 8.3);
       \draw[teal, very thick] (5.2, 7)[out=90, in=0]node[below] {$\rG$} to (4,8.3);
       \draw[dashed  , orange!80, very thick] (4, 7.5)[out=90, in=-180] to (4.7, 8.1) ; \draw (3.8, 8.2)node[below,orange] {$\rI_{\mi}$};
      \draw[green!50!orange,very thick] (3.5, 7)[out=90, in=-180]node[below] {$\rF'$} to (4, 7.5);
      \draw[blue!45, very thick] (4.5, 7)[out=90, in=0]node[below] {$\rG'$} to (4, 7.5);
\end{tikzpicture}
\end{equation}
Now we claim that $\tilde{\alpha}$ and $\tilde{\beta}$ make $(\rF\rF',\rG'\rG)$ into an adjoint pair.
Taking \eqref{eq:01} and \eqref{eq:02} into account, the left diagram in \eqref{diag:adjunction conditions}
is checked diagrammatically, see Figure~\ref{fig2}. The right diagram
in \eqref{diag:adjunction conditions} is checked similarly.
\end{proof}

\begin{figure}
\begin{equation*}
\begin{split}
\begin{tikzpicture}[anchorbase,scale=.75]
   \draw[dashed  , blue, very thick] (4, 8.3)[out=90, in=-180] to (6.2, 9.7) ; \draw (4.3,9.2)node[above] {$\rI_{\mj}$};
      \draw[very thick] (2.8, 7)[out=90, in=-180] to (4, 8.3);
       \draw[teal, very thick] (5.2, 7)[out=90, in=0] to (4,8.3);
     \draw[dashed  , orange!80, very thick] (4, 7.5)[out=90, in=0] to (3.4, 8.1); \draw  (4.3, 8.2)node[below] {$\rI_{\mi}$};
      \draw[green!50!orange,very thick] (3.5, 7)[out=90, in=-180] to (4, 7.5);
       \draw[green!50!orange,very thick] (3.5,7) to (3.5,3.3)node[below] {$\rF'$};
       \draw[very thick] (2.8,7) to (2.8,3.3)node[below] {$\rF$};
      \draw[blue!45, very thick] (4.5, 7)[out=90, in=0] to (4, 7.5);
      \draw[dashed  , brown!70!red, very thick] (5.7, 5.5)[out=-90,in=0]
      to (3.5, 4.1) ; \draw (4.9, 3.5)node[above] {$\ \rI'_{\mk}$};
      \draw[blue!45, very thick] (4.5, 6.8)[out=-90,in=-180] to (5.7, 5.5); \draw (4.3, 5.7)node[above] {$\rG'$};
       \draw[green!50!orange, very thick] (6.9,6.8)[out=-90,in=0] to (5.7,5.5);
       \draw[dashed  , red, very thick] (5.7, 6.3)[out=-90,in=-180] to (6.3, 5.7) ; \draw (5.4, 5.4) node[above] {$\ \rI'_{\mi}$};
      \draw[teal, very thick] (5.2, 6.8)[out=-90,in=-180] to (5.7, 6.3); \draw (5.3, 7.5)node[above] {$\rG$};
      \draw[teal, very thick] (5.2, 7)to (5.2, 6.8);
      \draw[blue!45, very thick] (4.5, 7)to (4.5, 6.8);
       \draw[very thick] (6.2, 6.8)[out=-90,in=0] to (5.7,6.3);
           \draw[green!50!orange, very thick] (6.9,10.5)node[above] {$\rF'$} to (6.9,6.8);
        \draw[very thick] (6.2, 10.5)node[above] {$\rF$} to (6.2,6.8);
\end{tikzpicture}&
\stackrel{\eqref{diag:sliding}}{=\joinrel=\joinrel=}
\begin{tikzpicture}[anchorbase,scale=.8]
   \draw[dashed  , blue, very thick] (4, 8.8)[out=80, in=-180] to (6.2, 9.7) ; \draw (4.4,9.5)node[above] {$\rI_{\mj}$};
      \draw[very thick] (2.8, 7.9)[out=90, in=-180] to (4, 8.8);
       \draw[teal, very thick] (5.2, 7.9)[out=90, in=0] to (4,8.8);
 \draw[dashed  , orange!80, very thick] (4, 6.2)[out=90, in=0] to (2.8, 6.6); \draw  (3.7, 7.3)node[below] {$\rI_{\mi}$};
      \draw[green!50!orange,very thick] (3.5, 5.7)[out=90, in=-180] to (4, 6.2);
       \draw[green!50!orange,very thick] (3.5,5.7) to (3.5,3.3)node[below] {$\rF'$};
       \draw[very thick] (2.8,7.9) to (2.8,3.3)node[below] {$\rF$};
      \draw[blue!45, very thick] (4.5, 5.7)[out=90, in=0] to (4, 6.2);
      \draw[dashed  , brown!70!red, very thick] (5.7, 5)[out=-100,in=0]
      to (3.5, 4.1) ; \draw (5, 3.4)node[above] {$\ \rI'_{\mk}$};
      \draw[blue!45, very thick] (4.5, 5.8)[out=-90,in=-180] to (5.7, 5); \draw (4.7, 4.5)node[above] {$\rG'$};
       \draw[green!50!orange, very thick] (6.9,6)[out=-90,in=0] to (5.7,5);
       \draw[dashed  , red, very thick] (5.7, 7.5)[out=-90,in=-180] to (6.9, 7.1) ; \draw (6, 6.3) node[above] {$\ \rI'_{\mi}$};
      \draw[teal, very thick] (5.2, 7.9)[out=-90,in=-180] to (5.7, 7.5); \draw (5.4, 8.1)node[above] {$\rG$};
       \draw[very thick] (6.2, 7.9)[out=-90,in=0] to (5.7,7.5);
           \draw[green!50!orange, very thick] (6.9,10.5)node[above] {$\rF'$} to (6.9,6);
        \draw[very thick] (6.2, 10.5)node[above] {$\rF$} to (6.2,7.9);
\end{tikzpicture}
\stackrel{\eqref{diag:oplaxlrcompatible}}{=\joinrel=\joinrel=}
\begin{tikzpicture}[anchorbase,scale=.8]
   \draw[dashed  , blue, very thick] (4, 8.8)[out=80, in=-180] to (6.2, 9.7) ; \draw (4.4,9.5)node[above] {$\rI_{\mj}$};
      \draw[very thick] (2.8, 7.9)[out=90, in=-180] to (4, 8.8);
       \draw[teal, very thick] (5.2, 7.9)[out=90, in=0] to (4,8.8);
 \draw[dashed  , orange!80, very thick] (4, 6.2)[out=90, in=0] to (2.8, 6.6) ; \draw (3.33, 6.5)node[below] {$\rI_{\mi}$};
      \draw[green!50!orange,very thick] (3.5, 5.7)[out=90, in=-180] to (4, 6.2);
       \draw[green!50!orange,very thick] (3.5,5.7) to (3.5,3.3)node[below] {$\rF'$};
       \draw[very thick] (2.8,7.9) to (2.8,3.3)node[below] {$\rF$};
      \draw[blue!45, very thick] (4.5, 5.7)[out=90, in=0] to (4, 6.2);
      \draw[dashed  , brown!70!red, very thick] (5.7, 5)[out=-100,in=0]
      to (3.5, 4.1) ; \draw (5, 3.4)node[above] {$\ \rI'_{\mk}$};
      \draw[blue!45, very thick] (4.5, 5.8)[out=-90,in=-180] to (5.7, 5); \draw (4.7, 4.5)node[above] {$\rG'$};
       \draw[green!50!orange, very thick] (6.9,6)[out=-90,in=0] to (5.7,5);
       \draw[dashed  , red, very thick] (5.7, 7.5)[out=-120,in=0] to (2.8, 7) ; \draw (6, 6.3) node[above] {$\ \rI'_{\mi}$};
      \draw[teal, very thick] (5.2, 7.9)[out=-90,in=-180] to (5.7, 7.5); \draw (5.4, 8.1)node[above] {$\rG$};
       \draw[very thick] (6.2, 7.9)[out=-90,in=0] to (5.7,7.5);
           \draw[green!50!orange, very thick] (6.9,10.5)node[above] {$\rF'$} to (6.9,6);
        \draw[very thick] (6.2, 10.5)node[above] {$\rF$} to (6.2,7.9);
\end{tikzpicture}\\
&\stackrel{\eqref{diag:lrcompatible}}{=\joinrel=\joinrel=}
\begin{tikzpicture}[anchorbase,scale=.8]
   \draw[dashed  , blue, very thick] (4, 8.8)[out=80, in=-180] to (6.2, 9.7); \draw (4.4,9.5)node[above] {$\rI_{\mj}$};
      \draw[very thick] (2.8, 7.9)[out=90, in=-180] to (4, 8.8);
       \draw[teal, very thick] (5.2, 7.9)[out=90, in=0] to (4,8.8);
 \draw[dashed  , orange!80, very thick] (4, 5.5)[out=50, in=-180] to (6.9, 6.3) ; \draw (6.3, 7)node[below] {$\rI_{\mi}$};
      \draw[green!50!orange,very thick] (3.5, 5)[out=90, in=-180] to (4, 5.5);
       \draw[green!50!orange,very thick] (3.5,5) to (3.5,2.5)node[below] {$\rF'$};
       \draw[very thick] (2.8,7.9) to (2.8,2.5)node[below] {$\rF$};
      \draw[blue!45, very thick] (4.5, 5)[out=90, in=0] to (4, 5.5);
      \draw[dashed  , brown!70!red, very thick] (5.7, 4.3)[out=-100,in=0]
      to (3.5, 3.4) ; \draw (4.8, 2.65)node[above] {$\ \rI'_{\mk}$};
      \draw[blue!45, very thick] (4.5, 5)[out=-90,in=-180] to (5.7, 4.3); \draw (4.9, 4.5)node[above] {$\rG'$};
       \draw[green!50!orange, very thick] (6.9,5)[out=-90,in=0] to (5.7,4.3);
       \draw[dashed  , red, very thick] (5.7, 7.5)[out=-130,in=0] to (2.8, 6.7) ; \draw (3.5, 5.9) node[above] {$\ \rI'_{\mi}$};
      \draw[teal, very thick] (5.2, 7.9)[out=-90,in=-180] to (5.7, 7.5); \draw (5.4, 8.1)node[above] {$\rG$};
       \draw[very thick] (6.2, 7.9)[out=-90,in=0] to (5.7,7.5);
           \draw[green!50!orange, very thick] (6.9,10.5)node[above] {$\rF'$} to (6.9,5);
        \draw[very thick] (6.2, 10.5)node[above] {$\rF$} to (6.2,7.9);
\end{tikzpicture}
\stackrel{\eqref{diag:sliding},\eqref{diag:adjunction conditions}}{=\joinrel=\joinrel=\joinrel=\joinrel=\joinrel=}
\begin{tikzpicture}[anchorbase,scale=.8]
       \draw[very thick] (3.5,10.5)node[above] {$\rF$} to (3.5,2.5)node[below] {$\rF$};
           \draw[green!50!orange, very thick] (4.2,10.5)node[above] {$\rF'$} to (4.2,2.5)node[below]{$\rF'$};
\end{tikzpicture}\quad.
\end{split}
\end{equation*}
\caption{The digram in the proof of Proposition~\ref{adjunction and horizontal}}\label{fig2}
\end{figure}

\begin{prop}\label{propadj}
Let $\bM$ be a bilax 2-representation of $\cC$.
Assume that $\rF\in\cC(\mi,\mj)$ is a left adjoint of $\rG\in\cC(\mj,\mi)$.
Then there are natural isomorphisms
\[\Hom_{\bM(\mj)}(\bM(\rF)(X),Y)\cong \Hom_{\bM(\mi)}(X,\bM(\rG)(Y)),\]
for any objects $X\in \bM(\mi),Y\in\bM(\mj)$.
\end{prop}

\begin{proof}
Suppose that $ \alpha:\rI'_{\mi}\to \rG\rF$ and $\beta:\rF\rG \to \rI_{\mj}$ are the adjunction morphisms
associated to the adjoint pair $(\rF,\rG)$.
Define a map $\Phi: \Hom_{\bM(\mj)}(\bM(\rF)(X),Y)\to \Hom_{\bM(\mi)}(X,\bM(\rG)(Y))$ by sending
any morphism $f$ from $\bM(\rF)(X)$ to $Y$ in $\bM (\mj)$ to the composite
\[\xymatrix@C=1.5pc{X\ar[rr]^{(u'_{\mi})_X\qquad}&&\bM(\rI'_{\mi})(X)\ar[rr]^{\bM(\alpha)_X\ }&& \bM(\rG\rF)(X)\ar@{=}[r]&\bM(\rG)\bM(\rF)(X)
\ar[rr]^{\quad \id_{\bM(\rG)}f}&&\bM(\rG)Y}\]
in $\bM(\mi)$. Define a map $\Psi: \Hom_{\bM(\mi)}(X,\bM(\rG)(Y)) \to \Hom_{\bM(\mj)}(\bM(\rF)(X),Y)$
by sending any morphism $g$ from $X$ to $\bM(\rG)Y$ in $\bM (\mi)$ to the composite
\[  \xymatrix@C=1.5pc{\bM(\rF)(X)\ar[rr]^{\id_{\bM(\rF)}g\quad}&& \bM(\rF)\bM(\rG)(Y)\ar@{=}[r]&\bM(\rF\rG)(Y)\ar[rr]^{
\quad\bM(\beta)_Y}&&\bM(\rI_{\mj})(Y)\ar[rr]^{\quad(u_{\mj})_Y}&&Y.}\]
We claim that $\Psi\Phi=\id_{\Hom_{\bM(\mj)}(\bM(\rF)(X),Y)}$ and $\Phi\Psi=\id_{\Hom_{\bM(\mi)}(X,\bM(\rG)(Y))}$.
Now we only prove the former equality since the latter can be proved simiarly.
Consider the diagram in Figure~\ref{fig3}.
\begin{figure}
\begin{equation*}\label{eq:03}\xymatrix@C=1.5pc{\bM(\rF)(X)\ar@/_5.0pc/@{=}[ddddrrrr]_{\bM(\id_{\rF})_X}\ar[rr]^{\id_{\bM(\rF)}(u'_{\mi})_X\qquad}\ar[drr]_{\bM(r'_{\mi,\mj,\rF})_X}&&
\bM(\rF)\bM(\rI'_{\mi})(X)\ar[rr]^{\id_{\bM(\rF)}\bM(\alpha)_X\ }\ar@{=}[d]&& \bM(\rF)\bM(\rG\rF)(X)\ar@{=}[dd]\ar@{=}[r]&\bM(\rF)\bM(\rG)\bM(\rF)(X)
\ar[d]^{\id_{\bM(\rF)\bM(\rG)}f}\\
&&\bM(\rF\rI'_{\mi})(X)\ar[dr]_{\bM(\id_{\rF}\alpha)_X\ }&&&\bM(\rF)\bM(\rG)(Y)\ar@{=}[d]\\
&&&\bM(\rF\rG\rF)(X)\ar[d]_{\bM(\beta\id_{\rF})}\ar@{=}[r]&
\bM(\rF\rG)\bM(\rF)(X)\ar[r]^{\id_{\bM(\rF\rG)}f}\ar[d]^{\bM(\beta)_{\bM(\rF)(X)}}&\bM(\rF\rG)(Y)\ar[d]^{\bM(\beta)_Y}\\
&&&\bM(\rI_{\mj}\rF)(X)\ar@{=}[r]\ar[dr]_{\bM(l_{\mi,\mj,\rF})_X}&\bM(\rI_{\mj})\bM(\rF)(X)\ar[r]^{\bM(\rI_{\mj})(f)}\ar[d]^{(u_{\mj})_{\bM(\rF)(X)}}&
\bM(\rI_{\mj})(Y)\ar[d]^{(u_{\mj})_Y}\\
&&&&\bM(\rF)(X)\ar[r]^{f}&Y.}
\end{equation*}
\caption{The diagram in the proof of Proposition~\ref{propadj}}\label{fig3}
\end{figure}
In this digram, the top and bottom left triangles commute by \eqref{eq:bilax-2-rep}. The left-most pentagon commutes by the fact that $(\rF,\rG)$ is an adjoint pair.
The bottom right two squares commutes due to the naturality of $\bM(\beta)$ and $u_{\mj}$. The remaining
subploygons of the diagram in Figure~\ref{fig3} commute by definition.
Hence the two pathes from $\bM(\rF)(X)$ to $Y$ along the boundary of the diagram in Figure~\ref{fig3} coincide with each other.
This implies the statement of the proposition.
\end{proof}

Proposition~\ref{propadj} says that, in any bilax $2$-representation, the pair $(\rF,\rG)$
of adjoint  $1$-morphisms is repersented by adjoint functors (in the usual sense).
In particular, left and right adjoints are unique up to a unique isomorphism of functors
(in the underlying category of the representation).
The following statement asserts that such an isomorphism exists already in $\cC$.

\begin{prop}\label{prop:adjointunique}
For any $1$-morphisms $\rF,\rF'\in\cC(\mi,\mj)$ and $\rG,\rG'\in\cC(\mj,\mi)$, we have:

\begin{enumerate}[$($i$)$]
\item\label{prop:adjointunique-1} if $(\rF,\rG)$ and $(\rF',\rG)$ are two adjoint pairs in $\cC$, then $\rF\cong \rF'$ in $\cC$;
\item\label{prop:adjointunique-2} if $(\rF,\rG)$ and $(\rF,\rG')$ are two adjoint pairs in $\cC$, then $\rG\cong\rG'$ in $\cC$.
\end{enumerate}
\end{prop}
\begin{proof}
We only prove claim~\ref{prop:adjointunique-1}, claim~\ref{prop:adjointunique-2} is proved similarly.
Suppose that $ \alpha:\rI'_{\mi}\to \rG\rF$ and $\beta:\rF\rG \to \rI_{\mj}$ are the adjunction morphisms
associated to the adjoint pair $(\rF,\rG)$ and $\alpha': \rI'_{\mi}\to \rG\rF'$ and $\beta':\rF'\rG \to \rI_{\mj}$ are the adjunction morphisms associated to the adjoint pair $(\rF',\rG)$. Define two 2-morphisms:
\[\begin{split}\phi:\xymatrix@C=3pc{\rF\ar[r]^{r'_{\mi,\mj,\rF}}&\rF\rI'_{\mi}\ar[r]^{\id_{\rF}\alpha'}&\rF\rG\rF'\ar[r]^{\beta\id_{\rF'}}& \rI_{\mj}\rF'\ar[r]^{l_{\mi,\mj,\rF'}}&\rF'}\\
\psi:\xymatrix@C=3pc{\rF'\ar[r]^{r'_{\mi,\mj,\rF'}}&\rF'\rI'_{\mi}\ar[r]^{\id_{\rF'}\alpha}&\rF'\rG\rF\ar[r]^{\beta'\id_{\rF}}& \rI_{\mj}\rF\ar[r]^{l_{\mi,\mj,\rF}}&\rF}.
\end{split}\]
These can be depicted as follows:
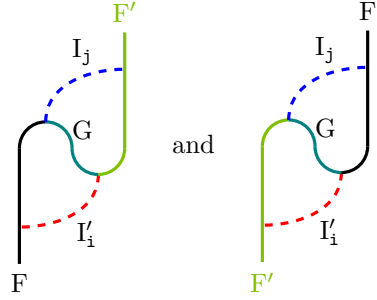
\[\begin{tikzpicture}[anchorbase,scale=.7]
      \draw[very thick] (3, 7.5) to (3, 5.3)node[below] {$\rF$};
      \draw[dashed  ,red, very thick] (4.5, 7)[out=-90,in=0]
      to (3, 6) ; \draw (4.2,6.3)node[below] {$\ \rI'_{\mi}$};
          \draw[very thick] (3, 7.5)[out=90, in=-180] to (3.5, 8);
      \draw[teal, very thick] (4, 7.5)[out=90, in=0] to (3.5, 8);
      \draw[teal, very thick] (4, 7.5)[out=-90,in=-180]  to (4.5, 7) ; \draw (4.2,7.5)node[above] {$\rG$};
             \draw[green!50!orange, very thick] (5,7.5)[out=-90,in=0] to (4.5,7);
                \draw[green!50!orange, very thick] (5,7.5) to (5,9.7) node[above] {$\rF'$};
                \draw[dashed  , blue, very thick] (3.5, 8)[out=90, in=-180] to (5, 9); \draw  (4.2,8.9)node[above] {$\rI_{\mj}$};
    \end{tikzpicture}
    \quad\text{and}\quad
    \begin{tikzpicture}[anchorbase,scale=.7]
      \draw[green!50!orange, very thick] (3, 7.5) to (3, 5.3)node[below] {$\rF'$};
      \draw[dashed  ,red, very thick] (4.5, 7)[out=-90,in=0]
      to (3, 6) ; \draw (4.2,6.3)node[below] {$\ \rI'_{\mi}$};
          \draw[green!50!orange, very thick] (3, 7.5)[out=90, in=-180] to (3.5, 8);
      \draw[teal, very thick] (4, 7.5)[out=90, in=0] to (3.5, 8);
      \draw[teal, very thick] (4, 7.5)[out=-90,in=-180]  to (4.5, 7); \draw (4.2,7.5)node[above] {$\rG$};
             \draw[very thick] (5,7.5)[out=-90,in=0] to (4.5,7);
                \draw[very thick] (5,7.5) to (5,9.7) node[above] {$\rF$};
                \draw[dashed  , blue, very thick] (3.5, 8)[out=90, in=-180] to (5, 9) ; \draw (4.2,8.9)node[above] {$\rI_{\mj}$};
    \end{tikzpicture}\]
We only prove the equality $\psi\phi=\id_{\rF}$ since the equality $\phi\psi=\id_{\rF'}$ follows by symmetry.
The equality $\psi\phi=\id_{\rF}$ is proved diagrammatically in Figure~\ref{fig1}. There the last equality in the first row uses both the naturality of the right oplax unitor $r'_{\mi,\mj,{}_-}$ and that of the left lax unitor $l_{\mi,\mj,{}_-}$ at the same time and so does the first one in the second row.
\end{proof}

\begin{figure}
\begin{equation*}
\begin{split}
\begin{tikzpicture}[anchorbase,scale=.59]
      \draw[very thick] (3, 7.5) to (3, 5.3)node[below] {$\rF$};
      \draw[dashed  ,red, very thick] (4.5, 7)[out=-90,in=0]
      to (3, 6); \draw  (4.15,6.27)node[below] {$\ \rI'_{\mi}$};
          \draw[very thick] (3, 7.5)[out=90, in=-180] to (3.5, 8);
      \draw[teal, very thick] (4, 7.5)[out=90, in=0] to (3.5, 8); \draw (6.23, 10.7)node[above] {$\rG$};
      \draw[teal, very thick] (4, 7.5)[out=-90,in=-180]  to (4.5, 7); \draw (4.2,7.5)node[above] {$\rG$};
             \draw[green!50!orange, very thick] (5,7.5)[out=-90,in=0] to (4.5,7);
                \draw[green!50!orange, very thick] (5,7.5) to (5,9.7);
                \draw[dashed  , blue, very thick] (3.5, 8)[out=90, in=-180] to (5, 8.9) ; \draw (4.2,8.9)node[above] {$\rI_{\mj}$};
                \draw[green!50!orange, very thick] (5, 10.7) to (5, 9.7); \draw (5.45, 7.3)node[above] {$\rF'$};
      \draw[dashed  ,red, very thick] (6.5, 10.2)[out=-90,in=0]
      to (5, 9.3) ; \draw (6,9.3)node[below] {$\ \rI'_{\mi}$};
          \draw[green!50!orange, very thick] (5, 10.7)[out=90, in=-180] to (5.5, 11.2); \draw (5.45, 10)node[above] {$\rF'$};
      \draw[teal, very thick] (6, 10.7)[out=90, in=0] to (5.5, 11.2);
      \draw[teal, very thick] (6, 10.7)[out=-90,in=-180]  to (6.5, 10.2); \draw (4.2,7.5)node[above] {$\rG$};
             \draw[very thick] (7,10.7)[out=-90,in=0] to (6.5,10.2);
                \draw[very thick] (7,10.7) to (7,12.9) node[above] {$\rF$};
                \draw[dashed  , blue, very thick] (5.5, 11.2)[out=90, in=-180] to (7, 12.2) ; \draw (5.9,12)node[above] {$\rI_{\mj}$};
    \end{tikzpicture}
   & \stackrel{\eqref{diag:unitornatural-1}}{=\joinrel=\joinrel=}
    \begin{tikzpicture}[anchorbase,scale=.59]
      \draw[very thick] (3, 7.5) to (3, 5.3)node[below] {$\rF$};
      \draw[dashed  ,red, very thick] (4.5, 7)[out=-90,in=0]
      to (3, 6) ; \draw (4.15,6.27)node[below] {$\ \rI'_{\mi}$};
          \draw[very thick] (3, 7.5)[out=90, in=-180] to (3.5, 8);
      \draw[teal, very thick] (4, 7.5)[out=90, in=0] to (3.5, 8); \draw (6.23, 10.7)node[above] {$\rG$};
      \draw[teal, very thick] (4, 7.5)[out=-90,in=-180]  to (4.5, 7); \draw (4.2,7.5)node[above] {$\rG$};
             \draw[green!50!orange, very thick] (5,7.5)[out=-90,in=0] to (4.5,7);
                \draw[green!50!orange, very thick] (5,7.5) to (5,9.7);
                \draw[dashed  , blue, very thick] (3.5, 8)[out=90, in=-180] to (5, 9.4) ; \draw (4.2,9.4)node[above] {$\rI_{\mj}$};
                \draw[green!50!orange, very thick] (5, 10.7) to (5, 9.7); \draw (5.45, 7.3)node[above] {$\rF'$};
      \draw[dashed  ,red, very thick] (6.5, 10.2)[out=-90,in=0]
      to (5, 8.8) ; \draw (6.2, 9.1)node[below] {$\ \rI'_{\mi}$};
          \draw[green!50!orange, very thick] (5, 10.7)[out=90, in=-180] to (5.5, 11.2); \draw (5.45, 10)node[above] {$\rF'$};
      \draw[teal, very thick] (6, 10.7)[out=90, in=0] to (5.5, 11.2);
      \draw[teal, very thick] (6, 10.7)[out=-90,in=-180]  to (6.5, 10.2); \draw (4.2,7.5)node[above] {$\rG$};
             \draw[very thick] (7,10.7)[out=-90,in=0] to (6.5,10.2);
                \draw[very thick] (7,10.7) to (7,12.9) node[above] {$\rF$};
                \draw[dashed  , blue, very thick] (5.5, 11.2)[out=90, in=-180] to (7, 12.2) ; \draw (5.9,12)node[above] {$\rI_{\mj}$};
    \end{tikzpicture} \stackrel{\eqref{diag:unitornatural},\eqref{diag:oplaxunitornatural}}{=\joinrel=\joinrel=\joinrel=\joinrel=\joinrel=}
     \begin{tikzpicture}[anchorbase,scale=.59]
      \draw[very thick] (3, 7.5) to (3, 5.3)node[below] {$\rF$};
      \draw[dashed  ,red, very thick] (4.5, 7)[out=-90,in=0]
      to (3, 6) ; \draw (4.15,6.27)node[below] {$\ \rI'_{\mi}$};
          \draw[very thick] (3, 7.5)[out=90, in=-180] to (3.5, 8);
      \draw[teal, very thick] (4, 7.5)[out=90, in=0] to (3.5, 8); \draw (6.23, 10.7)node[above] {$\rG$};
      \draw[teal, very thick] (4, 7.5)[out=-90,in=-180]  to (4.5, 7); \draw (4.2,7.5)node[above] {$\rG$};
             \draw[green!50!orange, very thick] (5,7.5)[out=-90,in=0] to (4.5,7);
                \draw[green!50!orange, very thick] (5,7.5) to (5,9.7);
                \draw[dashed  , blue, very thick] (3.5, 8)[out=90, in=-180] to (5.6,11.7) ; \draw (3.5,10)node[above] {$\rI_{\mj}$};
                \draw[green!50!orange, very thick] (5, 10.7) to (5, 9.7); \draw (5.45, 7.3)node[above] {$\rF'$};
      \draw[dashed  ,red, very thick] (6.5, 10.2)[out=-90,in=0]
      to (4.46, 6.45) ; \draw (6.8, 8.5)node[below] {$\ \rI'_{\mi}$};
          \draw[green!50!orange, very thick] (5, 10.7)[out=90, in=-180] to (5.5, 11.2); \draw (5.45, 10)node[above] {$\rF'$};
      \draw[teal, very thick] (6, 10.7)[out=90, in=0] to (5.5, 11.2);
      \draw[teal, very thick] (6, 10.7)[out=-90,in=-180]  to (6.5, 10.2); \draw (4.2,7.5)node[above] {$\rG$};
             \draw[very thick] (7,10.7)[out=-90,in=0] to (6.5,10.2);
                \draw[very thick] (7,10.7) to (7,12.9) node[above] {$\rF$};
                \draw[dashed  , blue, very thick] (5.5, 11.2)[out=90, in=-180] to (7, 12.2) ; \draw (5.9,12)node[above] {$\rI_{\mj}$};
    \end{tikzpicture}\\
    &\stackrel{\eqref{diag:unitornatural},\eqref{diag:oplaxunitornatural}}{=\joinrel=\joinrel=\joinrel=\joinrel=\joinrel=}
    \begin{tikzpicture}[anchorbase,scale=.59]
      \draw[very thick] (3, 7.5) to (3, 4.5)node[below] {$\rF$};
      \draw[dashed  ,red, very thick] (4.5, 7)[out=-90,in=0]
      to (3, 6); \draw  (4.5,6.7)node[below] {$\ \rI'_{\mi}$};
          \draw[very thick] (3, 7.5)[out=90, in=-180] to (3.5, 8);
      \draw[teal, very thick] (4, 7.5)[out=90, in=0] to (3.5, 8); \draw (6.23, 10.7)node[above] {$\rG$};
      \draw[teal, very thick] (4, 7.5)[out=-90,in=-180]  to (4.5, 7); \draw (4.2,7.5)node[above] {$\rG$};
             \draw[green!50!orange, very thick] (5,7.5)[out=-90,in=0] to (4.5,7);
                \draw[green!50!orange, very thick] (5,7.5) to (5,9.7);
                \draw[dashed  , blue, very thick] (3.5, 8)[out=90, in=-180] to (7,13.2) ; \draw (3.5,10.5)node[above] {$\rI_{\mj}$};
                \draw[green!50!orange, very thick] (5, 10.7) to (5, 9.7); \draw (5.45, 7.3)node[above] {$\rF'$};
      \draw[dashed  ,red, very thick] (6.5, 10.2)[out=-90,in=0]
      to (3, 5) ; \draw (6.65, 8.5)node[below] {$\ \rI'_{\mi}$};
          \draw[green!50!orange, very thick] (5, 10.7)[out=90, in=-180] to (5.5, 11.2); \draw (5.45, 10)node[above] {$\rF'$};
      \draw[teal, very thick] (6, 10.7)[out=90, in=0] to (5.5, 11.2);
      \draw[teal, very thick] (6, 10.7)[out=-90,in=-180]  to (6.5, 10.2); \draw (4.2,7.5)node[above] {$\rG$};
             \draw[very thick] (7,10.7)[out=-90,in=0] to (6.5,10.2);
                \draw[very thick] (7,10.7) to (7,13.7) node[above] {$\rF$};
                \draw[dashed  , blue, very thick] (5.5, 11.2)[out=90, in=-180] to (7, 12.2) ; \draw (5.3,11.2)node[above] {$\rI_{\mj}$};
    \end{tikzpicture}\stackrel{\eqref{diag:sliding}}{=\joinrel=\joinrel=}
    \begin{tikzpicture}[anchorbase,scale=.59]
      \draw[very thick] (3, 7.5) to (3, 3.3)node[below] {$\rF$};
      \draw[dashed  ,red, very thick] (4.5, 7)[out=-90,in=0]
      to (3, 6) ; \draw (4.1,6)node[below] {$\ \rI'_{\mi}$};
          \draw[very thick] (3, 9.5)[out=90, in=-180] to (3.5, 10);
      \draw[teal, very thick] (4, 9.5)[out=90, in=0] to (3.5, 10); \draw (6.35, 7)node[above] {$\rG$};
       \draw[very thick] (3, 7.5) to (3, 9.5);
        \draw[teal, very thick] (4, 7.5) to (4, 9.5);
      \draw[teal, very thick] (4, 7.5)[out=-90,in=-180]  to (4.5, 7); \draw (4.3,8.8)node[above] {$\rG$};
             \draw[green!50!orange, very thick] (5,7.5)[out=-90,in=0] to (4.5,7);
                \draw[green!50!orange, very thick] (5,7.5) to (5,8.5);
                \draw[dashed  , blue, very thick] (3.5, 10)[out=90, in=-180] to (7,12) ; \draw (4,11.3)node[above] {$\rI_{\mj}$};
      \draw[dashed  ,red, very thick] (6.5, 6)[out=-90,in=0]
      to (3, 4); \draw  (5.1, 4.3)node[below] {$\ \rI'_{\mi}$};
          \draw[green!50!orange, very thick] (5, 8.5)[out=90, in=-180] to (5.5, 9); \draw (5.45, 7.8)node[above] {$\rF'$};
      \draw[teal, very thick] (6, 8.5)[out=90, in=0] to (5.5, 9);
      \draw[teal, very thick] (6, 6.5)[out=-90,in=-180]  to (6.5, 6);
      \draw[teal, very thick] (6, 6.5) to (6, 8.5);
             \draw[very thick] (7,6.5)[out=-90,in=0] to (6.5,6);
                \draw[very thick] (7,6.5) to (7,12.5) node[above] {$\rF$};
                \draw[dashed  , blue, very thick] (5.5, 9)[out=90, in=-180] to (7, 10) ; \draw (5.3,9.8)node[above] {$\rI_{\mj}$};
    \end{tikzpicture}
\\
&\stackrel{\eqref{diag:lrcompatible},\eqref{diag:oplaxlrcompatible}}{=\joinrel=\joinrel=\joinrel=\joinrel=\joinrel=}
    \begin{tikzpicture}[anchorbase,scale=.56]
      \draw[very thick] (3, 7.5) to (3, 3.3)node[below] {$\rF$};
      \draw[dashed  ,red, very thick] (4.5, 7)[out=-90,in=-180]
      to (6, 6.38) ; \draw (4.9,6.3)node[below] {$\ \rI'_{\mi}$};
          \draw[very thick] (3, 9.5)[out=90, in=-180] to (3.5, 10);
      \draw[teal, very thick] (4, 9.5)[out=90, in=0] to (3.5, 10); \draw (6.35, 7)node[above] {$\rG$};
       \draw[very thick] (3, 7.5) to (3, 9.5);
        \draw[teal, very thick] (4, 7.5) to (4, 9.5);
      \draw[teal, very thick] (4, 7.5)[out=-90,in=-180]  to (4.5, 7); \draw (4.35,8.3)node[above] {$\rG$};
             \draw[green!50!orange, very thick] (5,7.5)[out=-90,in=0] to (4.5,7);
                \draw[green!50!orange, very thick] (5,7.5) to (5,8.5);
                \draw[dashed  , blue, very thick] (3.5, 10)[out=90, in=-180] to (7,11.8) ; \draw (4,11.2)node[above] {$\rI_{\mj}$};
      \draw[dashed  ,red, very thick] (6.5, 6)[out=-90,in=0]
      to (3, 4.2) ; \draw (5.1, 4.4)node[below] {$\ \rI'_{\mi}$};
          \draw[green!50!orange, very thick] (5, 8.5)[out=90, in=-180] to (5.5, 9); \draw (5.45, 7.8)node[above] {$\rF'$};
      \draw[teal, very thick] (6, 8.5)[out=90, in=0] to (5.5, 9);
      \draw[teal, very thick] (6, 6.5)[out=-90,in=-180]  to (6.5, 6);
      \draw[teal, very thick] (6, 6.5) to (6, 8.5);
             \draw[very thick] (7,6.5)[out=-90,in=0] to (6.5,6);
                \draw[very thick] (7,6.5) to (7,12.5) node[above] {$\rF$};
                \draw[dashed  , blue, very thick] (5.5, 9)[out=90, in=0] to (4, 9.59)
                ; \draw (5.1,9.6)node[above] {$\rI_{\mj}$};
    \end{tikzpicture}\stackrel{\eqref{diag:adjunction conditions}}{=\joinrel=\joinrel=}
    \begin{tikzpicture}[anchorbase,scale=.56]
      \draw[very thick] (3, 7.5) to (3, 3.3)node[below] {$\rF$};
          \draw[very thick] (3, 9.5)[out=90, in=-180] to (3.5, 10);
      \draw[teal, very thick] (4, 9.5)[out=90, in=0] to (3.5, 10); \draw (4.35, 8)node[above] {$\rG$};
       \draw[very thick] (3, 7.5) to (3, 9.5);
        \draw[teal, very thick] (4, 6.5) to (4, 9.5);
      \draw[teal, very thick] (4, 6.5)[out=-90,in=-180]  to (4.5, 6);
                \draw[dashed  , blue, very thick] (3.5, 10)[out=90, in=-180] to (5,11.6) ; \draw (3.7,11.1)node[above] {$\rI_{\mj}$};
      \draw[dashed  ,red, very thick] (4.5, 6)[out=-90,in=0]
      to (3, 4.4) ; \draw (4.2, 4.6)node[below] {$\ \rI'_{\mi}$};
             \draw[very thick] (5,6.5)[out=-90,in=0] to (4.5,6);
                \draw[very thick] (5,6.5) to (5,12.5) node[above] {$\rF$};
    \end{tikzpicture}
    \stackrel{\eqref{diag:adjunction conditions}}{=\joinrel=\joinrel=}
    \begin{tikzpicture}[anchorbase,scale=.56]
      \draw[very thick] (3, 12.5)node[above] {$\rF$} to (3, 3.3)node[below] {$\rF$};
    \end{tikzpicture}\quad.
\end{split}
\end{equation*}
\caption{The diagram in the proof of Proposition~\ref{prop:adjointunique}}\label{fig1}
\end{figure}
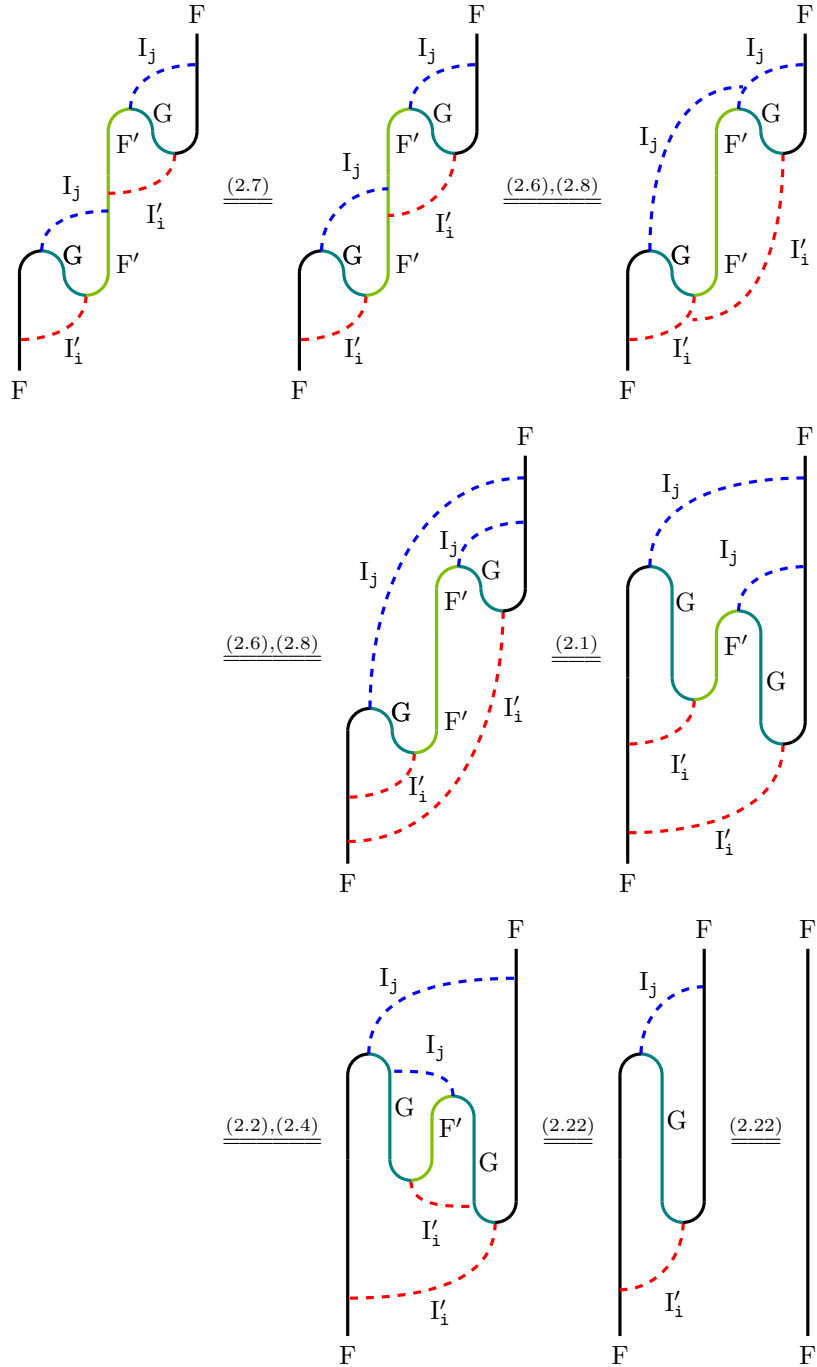

For any category $\mathcal A$, denote by $\End( \mathcal A)$ the category of
endofunctors on $\mathcal A$ and corresponding natural transformations.
Note that $\End( \mathcal A)$ can be viewed as a $2$-category with one object
(which can be identified with $\mathcal A$).

\begin{prop}\label{prop:(co)algebra}
Let $\rF\in\cC(\mi,\mj)$ be a left adjoint of $\rG\in\cC(\mj,\mi)$.
Then $h(\rG\rF,{}_-)=\rG\rF\circ{}_-\in\End\cC(\mi,\mi)$ is
an algebra in $\End (\cC(\mi,\mi))$ and $h(\rF\rG,{}_-)=\rF\rG\circ{}_-$ is a coalgebra in $\End (\cC(\mj,\mj))$.
\end{prop}

\begin{proof}
We only prove the first part and the rest can be proved similarly.
Suppose that $ \alpha:\rI'_{\mi}\to \rG\rF$ and $\beta:\rF\rG \to \rI_{\mj}$ are the adjunction morphisms
associated to the adjoint pair $(\rF,\rG)$.
It is easy to check that $\rG\rF\circ{}_-$ has the natural structure of a monad given by the multiplication map
\begin{equation}\label{eq:multiplication}
\xymatrix@C=2.5pc{\rG\rF\rG\rF\circ{}_-\ar[rr]^{\ \id_{\rG}\beta\id_{\rF}\circ{}_-}&&\rG\rI_{\mj}\rF\circ{}_-\ar[rr]^{r_{\mj,\mi,\rG}\id_{\rF}\circ{}_-}&&\rG\rF\circ{}_-}
\end{equation}
and the unit map
\begin{equation}
\xymatrix@C=1.5pc{\Id_{\ccC(\mi,\mi)}\ar[rr]^{\
l'_{\mi,\mi,{}_-}}&&\rI'_{\mi}\circ{}_-\ar[rr]^{\alpha\circ{}_-}&&\rG\rF\circ{}_-}.
\end{equation}
The associativity follows from the naturality of the lax right unitor $r_{\mi}$ and Definition~\ref{def:bilax-unit}~\eqref{def:lax-unit-1}~\ref{def:lax-unit-1-b}.
The left unitality follows directly from~\eqref{diag:adjunction conditions} and Definition~\ref{def:bilax-unit}~\eqref{def:lax-unit-1}~\ref{def:lax-unit-1-b}. Let us check the right unitality. For any $1$-morphism $\rH\in\cC(\mi,\mi)$, we have the diagram:
\[\xymatrix{
\rG\rF\rH\ar@{-->}[drr]_{\id_{\rG}r'_{\mi,\mj,\rF}\id_{\rH}}\ar[rr]^{\id_{\rG\rF}l'_{\mi,\mi,\rH}}&&\rG\rF\rI'_{\mi}\rH\ar@{==}[d]
\ar@{-->}[rr]^{\id_{\rG\rF}\alpha\id_{\rH}}&&\rG\rF\rG\rF\rH\ar@{-->}[rr]^{\id_{\rG}\beta\id_{\rF\rH}}&&
\rG\rI_{\mj}\rF\rH\ar[rr]^{r_{\mj,\mi,\rG}\id_{\rF\rH}}\ar@{-->}[drr]_{\id_{\rG}l_{\mi,\mj,\rF}\id_{\rH}}&&\rG\rF\rH\ar@{==}[d]\\
&&\rG\rF\rI'_{\mi}\rH&&&&&&\rG\rF\rH
}\]
where the two triangles commute due to \eqref{diag:lrcompatible} and \eqref{diag:oplaxlrcompatible}. By adjunction~\eqref{diag:adjunction conditions},
the dashed path from $\rG\rF\rH$ to $\rG\rF\rH$ equals $\id_{\rG\rF\rH}$. Therefore, by the commutativity of the above diagram, the path starting along all consecutive  arrows in the first row also equals $\id_{\rG\rF\rH}$.
This implies the desired right unitality and completes the proof.
\end{proof}

\section{Fiax categories}\label{s3}

\subsection{Fiax categories}\label{ss:fiax}

By a finitary category, we mean an additive, idempotent split,
$\Bbbk$-linear category which has finitely many isomorphism classes of indecomposable objects and finite dimensional morphism spaces (between any two indecomposable objects).
Denote by $\mathfrak A ^f_{\Bbbk}$ the 2-category of finitary categories with $1$-morphisms being $\Bbbk$-linear additive functors and $2$-morphisms being natural transformations.

\begin{defn}
A bilax-unital 2-category $\cC=(\cC,\,\rI=\{\rI_{\mi}|\,\mi\in\ob\cC\},\,\rI'=\{\rI'_{\mi}|\,\mi\in\ob\cC\})$ is called {\em finitary} if it satisfies that
\begin{enumerate}
    \item $\ob\cC$ is a finite set;
    \item for each $\mi,\mj\in \ob\cC$, the category $\cC(\mi,\mj)$ is finitary;
    \item all lax and oplax unit 1-morphisms are indecomposable.
\end{enumerate}
\end{defn}

\begin{defn}
A {\em finitary} bilax $2$-representation $\bM$ of a finitary bilax-unital 2-category $\cC$ is a bilax-unital 2-functor from $\cC$ to $\mathfrak A ^f_{\Bbbk}$ (cf. Definition~\ref{def:bilax-2-rep}).
\end{defn}

For any finitary bilax-unital $2$-category $\cC=(\cC,\,\rI=\{\rI_{\mi}|\,\mi\in\ob\cC\},\,\rI'=\{\rI'_{\mi}|\,\mi\in\ob\cC\})$ and each object $\mi\in\ob\cC$,
the $\mi$-th principal bilax $2$-representation $\bP_{\mi}=\cC(\mi,{}_-)$ is finitary.

\begin{defn}
Let $\bM$ be a finitary bilax $2$-representation of finitary bilax-unital $2$-category $\cC$. For any $X\in\bM(\mi)$, we say that the object $X$ {\em generates} $\bM$ if,
for any $\mj\in\ob\cC$ and $Y\in\bM(\mj)$, there exists a $1$-morphism $\rF\in\cC(\mi,\mj)$ such that $Y$ is isomorphic to a direct summand of $\bM(\rF)X$.

A finitary bilax $2$-representation $\bM$ of $\cC$ is called {\em transitive} if every nonzero $X\in\bM(\mi)$ generates $\bM$.
\end{defn}

\begin{defn}\label{deffiax}
A finitary bilax-unital 2-category $\cC=(\cC,\,\rI=\{\rI_{\mi}|\,\mi\in\ob\cC\},\,\rI'=\{\rI'_{\mi}|\,\mi\in\ob\cC\})$ is called {\em weakly fiax} if there exists an equivalence $({}_-)^\star:\cC\to \cC^{\operatorname{\,op\,,op}}$, where the latter $2$-category has the same object as $\cC$ with both $1$- and $2$-morphisms reversed, such that the following three
conditions are satisfied (we denote by $^\star({}_-)$ the inverse of $({}_-)^\star$):
\begin{enumerate}
    \item\label{deffiax-1} $(\rI_{\mi})^\star=\rI'_{\mi}$, $(l_{\mj,\mi,\rF})^\star=r'_{\mi,\mj,\rF^{\star}}$ and $(r_{\mi,\mk,\rG})^\star=l'_{\mk,\mi,\rG^{\star}}$, for any $1$-morphisms $\rF\in\cC(\mj,\mi)$ and $\rG\in\cC(\mi,\mk)$;
    \item\label{deffiax-2} for each $1$-morphism $\rF\in\cC(\mj,\mi)$, there exists adjunction morphisms
    \[\alpha_{\rF}:\rI'_{\mj}\to \rF^{\star}\rF\quad\text{and}\quad
    \beta_{\rF}:\rF\rF^{\star} \to \rI_{\mi},\]
    such that the pair $(\rF,\rF^{\star})$ is an adjoint pair;
    \item \label{deffiax-3} for each $2$-morphism $\gamma:\rF\to\rG$ with $\rF,\rG\in\cC(\mj,\mi)$, the $2$-morphism $\gamma^{\star}$ is given by the composite
\begin{equation}\label{composit}
\xymatrix{\rG^{\star}\ar[rr]^{l'_{\mi,\mj,\rG^{\star}}}&&\rI'_{\mj}\rG^{\star}\ar[rr]^{\alpha_{\rF}\id_{\rG^{\star}}}&&
\rF^{\star}\rF\rG^{\star}\ar[rr]^{\id_{\rF^{\star}}\gamma\id_{\rG^{\star}}}&&\rF^{\star}\rG\rG^{\star}\ar[rr]^{\id_{\rF^{\star}}\beta_{\rG}}&&
\rF^{\star}\rI_{\mi}\ar[rr]^{r_{\mi,\mj,\rF^{\star}}}&&\rF^{\star}.}
\end{equation}
\end{enumerate}
Furthermore, we say that $\cC$ is {\em fiax} if, additionally, the equivalence $({}_-)^\star$ above is a weak involution, that is, if we have
\begin{enumerate}[$($4$)$]
    \item $({}_-)^{\star\star}\cong \Id_{\ccC}$.
\end{enumerate}
\end{defn}

\begin{rem}
Compare Definition~\ref{deffiax} with that of a weakly fiat (or fiat) $2$-category in \cite{MM1,MM5}: A $2$-category $\cC$ is called {\em weakly fiat} if there exists an equivalence $({}_-)^\star:\cC\to\cC^{\operatorname{\,op,\,op}}$ such that, for each 1-morphism $\rF\in \cC$, $(\rF,\rF^\star)$ is an adjoint pair in the classical sense (which agrees with our definition of an adjoint pair if we view $\cC$ as a bilax-unital 2-category $(\cC, \rI=\rI'=\{\mathbbm{1}_{\mi}|\,\mi\in\ob\cC\})$).
A requirement \eqref{deffiax-3} in Definition~\ref{deffiax} is not present in this definition of a weakly fiat $2$-category, but one can always replace $({}_-)^\star$, as long as it exists, with one that satisfies \eqref{deffiax-3}.
One can for example choose adjunction morphisms $\alpha_\rF,\beta_\rF$ for each indecomposable 1-morphism $\rF$ in $\cC$, which yields a choice of adjunction morphisms for all 1-morphisms in $\cC$ at once, and then take \eqref{composit} as the definition of $\gamma^\star$ using the chosen adjunction morphisms.
Therefore, weakly fiax categories are bilax-unital analogues of weakly fiat $2$-categories from \cite{MM5}.

Also note that both the definitions for fiax and fiat $2$-categories only require $({}_-)^{\star\star}$ to be isomorphic to the identity functor without specifying an isomorphism $({}_-)^{\star\star}\cong \Id_{\ccC}$.
This is different from the standard definition of a pivotal (monoidal) category (see \cite[Section~4.7]{EGNO}) where such a natural isomorphism, called a {\em pivot}, is an additional structure on the monoidal category.
\end{rem}

For the rest of the paper, we only discuss fiax categories and not weakly fiax categories, while we do not use that $({}_-)^\star$ is an involution in our arguments (i.e., weakly fiax is enough).
This is comparable to the theory for fiat 2-categories, most of which is valid for weakly fiat 2-categories but stated for fiat 2-categories.\footnote{This also has to do with the terminology. There has been a suggestion that `weakly fiax' should be `fax' and `weakly fiat' should be `fat', which was dismissed by a majority of the authors of the current paper.}
All examples of weakly fiax or weakly fiat 2-categories appearing in the paper are fiax (resp., fiat).

\begin{lem}\label{prop:fiaxsplit}
For any fiax category $\cC=(\cC,\,\rI=\{\rI_{\mi}|\,\mi\in\ob\cC\},\,\rI'=\{\rI'_{\mi}|\,\mi\in\ob\cC\})$, each lax unit $\rI_{\mi}$ and each oplax unit $\rI'_{\mi}$ is split.
\end{lem}

\begin{proof}
The statement follows from the fact that $(\rF,\rF^{\star})$ is an adjoint pair for any $1$-morphism $\rF\in\cC(\mj,\mi)$ and Definition~\ref{def:ad-pair}.
\end{proof}

\begin{prop}\label{prop:Fmodabel}
Let $\cC=(\cC,\,\rI=\{\rI_{\mi}|\,\mi\in\ob\cC\},\,\rI'=\{\rI'_{\mi}|\,\mi\in\ob\cC\})$ be a fiax category and $\rA=(\rA,\,\mu_{\rA},\,\eta_{\rA})$ an algebra $1$-morphism in $\cC(\mi,\mi)$ for some object $\mi$

If the category $\cC (\mj,\mi)$ $($resp. $\cC (\mi,\mj))$ is abelian, then $\rA\text{-}{\mathrm{mod}}_{\ccC}(\mj)$ $($resp. $({\mathrm{mod}}_{\ccC}\text{-}\rA)(\mj))$ is also abelian.
\end{prop}

\begin{proof}
Since $\cC$ is fiax, (horizontal) composition with $\rA$ is exact and therefore preserves kernels and cokernels.
Thus, given a morphism $f:\rM\to \rN$ in $\rA\text{-}{\mathrm{mod}}_{\ccC}(\mj)$, the diagram
\begin{equation*}
    \begin{tikzcd}
        \rA\rM\arrow{r}{\id_\rA f}\arrow{d}{\upsilon_{\rM}} & \rA\rN\arrow{d}{\upsilon_{\rN}}\arrow{r}{\id_\rA g} &\rA\operatorname{coker}f \arrow{r}{}\arrow[dashed]{d}{\upsilon} & 0\\
        \rM\arrow{r}{f}& \rN \arrow{r}{g}& \operatorname{coker}f\arrow{r}{} &0
    \end{tikzcd}
\end{equation*}
in $\cC(\mj,\mi)$ is commutative and exact, where $g$ is the cokernel map in $\cC(\mj,\mi)$.
It is easy to check that the induced map $\upsilon:\rA\operatorname{coker}f\to \operatorname{coker} f$ defines a $\rA$-module structure on $\operatorname{coker} f$.
Similarly, the kernels exists in $\rA\text{-}{\mathrm{mod}}_{\ccC}(\mj)$. The proof for $({\mathrm{mod}}_{\ccC}\text{-}\rA)(\mj)$ is the same.
\end{proof}

\subsection{2-categories associated to fiax categories} \label{s3.2}

Let $\cC=(\cC,\,\rI=\{\rI_{\mi}|\,\mi\in\ob\cC\},\,\rI'=\{\rI'_{\mi}|\,\mi\in\ob\cC\})$ be a fiax  $2$-category.
For any object $\mj\in\ob\cC$, denote by $\mathcal D(\mj)$ be the full additive subcategory of $\End (\cC(\mj,\mi))$ generated by $\Id_{\ccC(\mj,\mi)}$ and $h_{\mj,\mi,\mi}(\rF,{}_-)=\rF\circ{}_-$, for any $1$-morphism $\rF\in \cC(\mi,\mi)$. Note that $h_{\mj,\mi,\mi}({}_-,{}_-)$ defines a faithful functor from $\cC(\mi,\mi)$ to $\mathcal D(\mj)$ by sending any $1$-morphism $\rF$ to $h_{\mj,\mi,\mi}(\rF,{}_-)$ and with the obvious assignment on $2$-mor\-phisms. Indeed, if there is any $2$-morphisms $\alpha:\rF\to\rG$ in $\cC(\mi,\mi)$ such that $h_{\mj,\mi,\mi}(\alpha,{}_-)=0$, then
we have $h_{\mj,\mi,\mi}(\alpha,\rI_{\mi})=\alpha\id_{\rI_{\mi}}=0$ which implies that $\alpha\circ_{\rv}r_{\mi,\mi,\rF}=r_{\mi,\mi,\rG}\circ_{\rv}(\alpha\id_{\rI_{\mi}})=0$ by the naturality of $r_{\mi,\mi,{}_-}$. Since $r_{\mi,\mi,\rF}$ is split epic, we have $\alpha=0$. Therefore for any $\mj\in\ob\cC$ we have the functor
\begin{equation}\label{embedtoD}
   \xymatrix{\cC(\mi,\mi)\ar[rr]^{h_{\mj,\mi,\mi}({}_-,{}_-)}&&\mathcal D(\mj)\, \ar@{^{(}->}[rr]&&\overline{\mathcal D(\mj)}}.
\end{equation}
Note that the functors $l_{\mi,\mi,\rI_{\mi}}\circ{}_-$ and $r_{\mi,\mi,\rI_{\mi}}\circ{}_-$ belong to $\mathcal D(\mj)$. Then we can consider
the coequalizer of the two functors $l_{\mi,\mi,\rI_{\mi}}\circ{}_-$ and $r_{\mi,\mi,\rI_{\mi}}\circ{}_-$ in the category $\overline{\mathcal D(\mj)}$ by (the dual of) Proposition~\ref{kernel} and describe it as follows.

\begin{lem}\label{lemcoeq}
Let $\cC=(\cC,\,\rI=\{\rI_{\mi}|\,\mi\in\ob\cC\},\,\rI'=\{\rI'_{\mi}|\,\mi\in\ob\cC\})$
be a finitary bilax-unital 2-category. Assume that each $2$-morphism $l_{\mj,\mi,\rF}:\rI_{\mi}\rF\to\rF$
splits, for any $1$-morphism $\rF\in\cC(\mj,\mi)$. Then the coequalizer of
\begin{equation}\label{coeq}
\xymatrix{\rI_{\mi}\rI_{\mi}\circ{}_-\ar@<.6ex>[rrr]^{l_{\mi,\mi,\rI_{\mi}}\circ{}_-}
\ar@<-.4ex>[rrr]_{r_{\mi,\mi,\rI_{\mi}}\circ{}_-}&&&\rI_{\mi}}\circ{}_-
\end{equation}
in $\overline{\mathcal D(\mj)}$
is isomorphic to the identity functor $\mathrm{Id}_{\ccC(\mj,\mi)}$ with the left lax unitor.
\end{lem}

\begin{proof}
Denote the coequalizer of \eqref{coeq} by $\mathrm{U}_{\mi}\in\overline{\mathcal D(\mj)}$ with the natural transformation
\[\xi_{\mi}:\rI_{\mi}\circ{}_-\to \mathrm{U}_{\mi}\]
in $\overline{\mathcal D(\mj)}$.
For any $1$-morphism $\rF\in\cC(\mj,\mi)$, we have:
\begin{equation}\label{eq:002}
\begin{tikzpicture}[anchorbase]
\draw[very thick] (4, 8)node[above] {$\rF$} to (4, 6)node[below] {$\rF$};
\draw[dashed  , very thick] (3.3, 6)[out=90, in=-180]node[below] {$\ \rI_{\mi}$} to (4, 7.5);
\draw[dashed  , very thick] (2.8, 6)[out=90, in=-180]node[below] {$\ \rI_{\mi}$} to (3.37, 6.9);
\end{tikzpicture}\stackrel{\eqref{diag:unitornatural}}{=\joinrel=\joinrel=}
\begin{tikzpicture}[anchorbase]
\draw[very thick] (4, 8)node[above] {$\rF$} to (4, 6)node[below] {$\rF$};
\draw[dashed  , very thick] (3.5, 6)[out=90, in=-180]node[below] {$\ \rI_{\mi}$} to (4, 6.9);
      \draw[dashed  , very thick] (3, 6)[out=90, in=-180]node[below] {$\ \rI_{\mi}$} to (4, 7.5);
\end{tikzpicture}
\stackrel{\eqref{diag:lrcompatible}}{=\joinrel=\joinrel=}
\begin{tikzpicture}[anchorbase]
\draw[very thick] (4, 8)node[above] {$\rF$} to (4, 6)node[below] {$\rF$};
\draw[dashed  , very thick] (3.45, 6)[out=90, in=0]node[below] {$\ \rI_{\mi}$} to (2.93, 6.65);
\draw[dashed  , very thick] (2.8, 6)[out=90, in=-180]node[below] {$\ \rI_{\mi}$} to (4, 7.4);
\end{tikzpicture}
\end{equation}
Here the left equality follows by applying the naturality of $l_{\mj,\mi,{}_-}$ to the $2$-morphism $l_{\mj,\mi,\rF}: \rI_{\mi}\rF\to\rF$.
Therefore $l_{\mj,\mi,\rF}$ coequalizes $l_{\mi,\mi,\rI_{\mi}}\mathrm{id}_{\rF}$ and $r_{\mi,\mi,\rI_{\mi}}\mathrm{id}_{\rF}$.
By the universal property of coequalizers, there exists a unique $2$-morphism $\theta_{\mi}(\rF):\mathrm{U}_{\mi}(\rF)\to\rF$ such that the triangle in the diagram
\begin{equation}\label{eq:001}
\xymatrix{\rI_{\mi}\rI_{\mi}\rF\ar@<.6ex>[rrr]^{l_{\mi,\mi,\rI_{\mi}}\mathrm{id}_{\rF}}
\ar@<-.4ex>[rrr]_{r_{\mi,\mi,\rI_{\mi}}\mathrm{id}_{\rF}}&&&\rI_{\mi}\rF\ar[rr]^{\xi_{\mi}(\rF)}\ar[drr]_{l_{\mj,\mi,\rF}}&&\mathrm{U}_{\mi}(\rF)
\ar@{-->}[d]^{\theta_{\mi}(\rF)}\\
&&&&&\rF}
\end{equation}
commutes, i.e., $\theta_{\mi}(\rF)\circ_{\rv}\xi_{\mi}(\rF)=l_{\mj,\mi,\rF}$.
The naturality of $l_{\mj,\mi,{}_-}$ and $\xi_{\mi}$, and the  fact that $\xi_{\mi}$ is epic,
implies that $\{\theta_{\mi}(\rF):\mathrm{U}_{\mi}(\rF)\to\rF|\, \rF\in\cC(\mi,\mi)\}$
defines a natural transformation $\theta_{\mi}:\mathrm{U}_{\mi}\to\mathrm{Id}_{\ccC(\mj,\mi)}$.

Note that each $l_{\mj,\mi,\rF}$ is split. Denote by $\alpha_{\rF}:\rF\to\rI_{\mi}\rF$ the $2$-morphism such that $l_{\mj,\mi,\rF}\circ_{\rv}\alpha_{\rF}=\mathrm{id}_{\rF}$. Then we have
\[\theta_{\mi}(\rF)\circ_{\rv}\xi_{\mi}(\rF)\circ_{\rv}\alpha_{\rF}=l_{\mj,\mi,\rF}\circ_{\rv}\alpha_{\rF}=\mathrm{id}_{\rF}.\]
Now we claim that
$\xi_{\mi}(\rF)\circ_{\rv}\alpha_{\rF}\circ_{\rv}\theta_{\mi}(\rF)=\mathrm{id}_{\mathrm{U}_{\mi}(\rF)}$. Using the universal property of coequalizers,
it suffices to show the equality
\begin{equation}\label{eq1}
\xi_{\mi}(\rF)\circ_{\rv}\alpha_{\rF}\circ_{\rv}\theta_{\mi}(\rF)\circ_{\rv}\xi_{\mi}(\rF)=\xi_{\mi}(\rF).
\end{equation}
Consider the following diagram
\[\xymatrix@R=0.9pc{&\rI_{\mi}\rF\ar[rr]^{\mathrm{id}_{\rI_{\mi}}\alpha_{\rF}}\ar@/^2.5pc/[rrrr]^{\mathrm{id}_{\rI_{\mi}\rF}}
\ar[ddl]_{\xi_{\mi}(\rF)}\ar[dddd]_{l_{\mj,\mi,\rF}}&&\rI_{\mi}\rI_{\mi}\rF\ar[dddd]^{l_{\mj,\mi,\rI_{\mi}\rF}}
\ar[rr]^{\mathrm{id}_{\rI_{\mi}}l_{\mj,\mi,\rF}}&&\rI_{\mi}\rF\ar[dddd]^{\xi_{\mi}(\rF)}\\
&&&&&\\
\mathrm{U}_{\mi}(\rF)\ar[ddr]_{\theta_{\mi}(\rF)}&&&&&\\
&&&&&\\
&\rF\ar[rr]^{\alpha_{\rF}}&&\rI_{\mi}\rF\ar[rr]^{\xi_{\mi}(\rF)}&&\mathrm{U}_{\mi}(\rF),}\]
where the top and the left triangles commute by definition; the left square commutes by the naturality of
$l_{\mj,\mi,{}_-}$;
and the right square commutes since $\xi_{\mi}$ gives the coequalizer of \eqref{coeq} and
the equality in Definition~\ref{def:bilax-unit}~\ref{def:lax-unit-1-b}
implies $\mathrm{id}_{\rI_{\mi}}l_{\mj,\mi,\rF}=r_{\mi,\mi,\rI_{\mi}}\id_{\rF}$.
Hence, the two pathes along the boundary of the above diagram from the north-west
vertex $\rI_{\mi}\rF$ to the south-east vertex $\mathrm{U}_{\mi}(\rF)$ coincide,
establishing \eqref{eq1}. The claim follows.
\end{proof}

Let $\cC=(\cC,\,\rI=\{\rI_{\mi}|\,\mi\in\ob\cC\},\,\rI'=\{\rI'_{\mi}|\,\mi\in\ob\cC\})$ be a fiax category. For any fixed object $\mi\in\ob\cC$, we have the embedding
\begin{equation*}
    \cC(\mi,\mi)\inj \overline{\cC(\mi,\mi)}.
\end{equation*}
By the dual of Proposition~\ref{kernel}, the category $\overline{\cC(\mi,\mi)}$ contains the coequalizer of $l_{\mi,\mi,\rI_{\mi}}$ and $r_{\mi,\mi,\rI_{\mi}}$.
Since the coequalizer is defined only up to isomorphism, we pick one representative in its isomorphism class and denote it by $\mathbbm{1}_{\mi}$.
By abuse of notation, we still denote by $\xi_{\mi}$ the epimorphism $\rI_{\mi}\to\mathbbm{1}_{\mi}$ associated to the coequalizer $\mathbbm{1}_{\mi}$.
By Lemmata~\ref{prop:fiaxsplit} and \ref{lemcoeq}, together with the diagram \eqref{embedtoD},
we have a natural map $\mathbbm{1}_{\mi}\rF\xrightarrow{\cong} \rF$, for any $1$-morphism $\rF\in\cC(\mj,\mi)$.
By arguments similar to those of Lemma~\ref{lemcoeq}, one can verify that ${}_-\circ\mathbbm{1}_{\mi}$, that is, the coequalizer of ${}_-\circ l_{\mi,\mi,\rI_{\mi}}$ and
${}_-\circ r_{\mi,\mi,\rI_{\mi}}$ as functors in $\End(\cC(\mi,\mj))$, is isomorphic to $\Id_{\ccC(\mi,\mj)}$, for each $\mj\in\ob\cC$.
Therefore we have a natural map $\rG\mathbbm{1}_{\mi}\xrightarrow{\cong} \rG$, for any $1$-morphism $\rG\in \cC(\mi,\mk)$.

\begin{rem}\label{rem:oplaxside}
By an argument, dual to that of Lemma~\ref{lemcoeq}, we can deduce that the equalizer of $l'_{\mi,\mi,\rI'_{\mi}}\circ{}_-$ and $r'_{\mi,\mi,\rI'_{\mi}}\circ{}_-$, which exists in the category $\underline{\cC(\mi,\mi)}$, is isomorphic to the functor $\Id_{\ccC(\mj,\mi)}$ in $\underline{\mathcal D(\mj)}$. One also has a similar statement when considering the corresponding functors given by right compositions. We pick one representative in the isomorphism class of the equalizer of $l'_{\mi,\mi,\rI'_{\mi}}$ and $r'_{\mi,\mi,\rI'_{\mi}}$, and denote it by
$\mathbbm{1}'_{\mi}$. Denote by $\xi'_{\mi}$ the monomorphism $\mathbbm{1}'_{\mi}\to\rI'_{\mi}$. Then we have natural maps $\mathbbm{1}'_{\mi}\rF\xrightarrow{\cong}\rF$ and $\rG\mathbbm{1}'_{\mi}\xrightarrow{\cong}\rG$, for any $1$-morphisms $\rF\in\cC(\mj,\mi)$ and $\rG\in\cC(\mi,\mk)$.
If $\cC$ is a fiax category, the weak involution $({}_-)^{\star}$ on $\cC$ extends to a biequivalence $\underline{\cC}\to \overline{\cC}^{\operatorname{\,op,\,op}}$. Alternatively, one can apply $({}_-)^{\star}$ to the proof of Lemma~\ref{lemcoeq} (which needs to be modified slightly) and obtain the dual statement.
\end{rem}

Note that, for a finitary bilax-unital $2$-category $\cC$, both $\underline{\cC}(\mi,\mj)$ and $\overline{\cC}(\mi,\mj)$ are abelian, for each pair $\mi,\mj\in\ob\cC$ of objects. But $\underline{\cC}(\mi,\mj)$ and $\overline{\cC}(\mi,\mj)$ are usually not finitary.

Now, we extend the natural maps $\mathbbm{1}_{\mi}\rF\xrightarrow{\cong} \rF$
and $\rG\mathbbm{1}_{\mi}\xrightarrow{\cong} \rG$
to any $1$-morphism $\rF\in\overline{\cC(\mj,\mi)}$, $\rG\in\overline{\cC(\mi,\mk)}$.

\begin{prop}\label{prop:1 unit in abeln} Given the same assumption as in Lemma~\ref{lemcoeq},
each 1-morphism $\mathbbm{1}_{\mi}$ is a weak unit (i.e., a lax unit where both unitors are isomorphisms) on $\overline{\cC}$.
\end{prop}

\begin{proof}
Note that a $1$-morphism $(\{f_i:\rG_i\to \rF\},n)\in \overline{\cC(\mj,\mi)}$, where $\mj\in\ob\cC$, is the
cokernel of the map $\sum_i f_i:\oplus_i \rG_i\to \rF$ in $\cC(\mj,\mi)$. Also, note the fact that
$\mathbbm{1}_{\mi}\circ{}_-$ is the cokernel of
a natural transformation between the exact functors $\rI_{\mi}\rI_{\mi}\circ{}_-$ and $\rI_{\mi}\circ{}_-$ on $\overline{\cC(\mj,\mi)}$. Hence
the action of $\mathbbm{1}_{\mi}$ is right exact on $\overline{\cC(\mj,\mi)}$.
Then the statement that $\mathbbm{1}_{\mi}\rF\xrightarrow{\cong} \rF$,
for any $\rF\in\overline{\cC(\mj,\mi)}$, follows from the naturality of those isomorphisms.
Similarly one shows that the right unitor for $\mathbbm{1}_{\mi}$ is an isomorphism.
The compatibility condition between the left and the right unitors of $\mathbbm{1}_{\mi}$
follows from the compatibility of $l_{\mi}$ and $r_{\mi}$ for $\rI_{\mi}$.
\end{proof}

The above proposition implies that each $\mathbbm{1}_{\mi}$ does not depend on the choice of split lax units, up to isomorphisms.
Indeed, if we have two split lax units $\rI_{\mi}$, $\tilde{\rI}_{\mi}$ and the corresponding coequalizer diagrams
\[
\xymatrix{\rI_{\mi}\rI_{\mi}\ar@<.6ex>[rr]^{l_{\mi,\mi,\rI_{\mi}}}
\ar@<-.4ex>[rr]_{r_{\mi,\mi,\rI_{\mi}}}&&\rI_{\mi}\ar[rr]&&\mathbbm{1}_{\mi}}\quad\text{and}\quad
\xymatrix{\tilde{\rI}_{\mi}\tilde{\rI}_{\mi}\ar@<.6ex>[rr]^{\tilde{l}_{\mi,\mi,\tilde{\rI}_{\mi}}}
\ar@<-.4ex>[rr]_{\tilde{r}_{\mi,\mi,\tilde{\rI}_{\mi}}}&&\tilde{\rI}_{\mi}\ar[rr]&&\tilde{\mathbbm{1}}_{\mi}},
\]
then, by Proposition~\ref{prop:1 unit in abeln}, we have
$\mathbbm{1}_{\mi}\cong\mathbbm{1}_{\mi}\tilde{\mathbbm{1}}_{\mi}\cong\tilde{\mathbbm{1}}_{\mi}$.

\begin{rem}\label{rem0}
Recall that $\overline{\cC}$ is a bilax-unital $2$-category with the induced lax/oplax units, which we still denoted by $\rI=\{\rI_{\mi}|\,\mi\in\ob\cC\}$ and
$\rI'=\{\rI'_{\mi}|\,\mi\in\ob\cC\}$ by abuse of notation. While, if $\cC$ is a finitary bilax-unital $2$-category with split lax units, by Proposition~\ref{prop:1 unit in abeln}, $\overline{\cC}$ is also a $2$-category with weak units being those $\mathbbm{1}_{\mi}$'s. To distinguish those two structures, we would write $(\overline{\cC}, \{\mathbbm{1}_{\mi}|\,\mi\in\ob\cC\})$ for being a $2$-category with weak units instead of $\overline{\cC}$, where the latter we refer to a bilax-unital $2$-category.
\end{rem}

Define $\widehat{\cC(\mi,\mi)}$ to be the additive closure of $\cC(\mi,\mi)$ and $\mathbbm{1}_{\mi}$ in the category $\overline{\cC(\mi,\mi)}$.
Then the category $\widehat{\cC(\mi,\mi)}$ has a monoidal structure with $\mathbbm{1}_{\mi}$ as the unit object with respect to the horizontal composition.
Note that the horizontal composition is strictly associative but $\mathbbm{1}_{\mi}$ is not a strict unit object.
By the paragraph before Remark~\ref{rem0}, the category $\widehat{\cC(\mi,\mi)}$, up to equivalence, does not depend (as a subcategory of $\overline{\cC(\mi,\mi)}$) on the choice of split lax units.
For later use, we record the coherence property of monoidal categories (see for example \cite{Kel64}) as the
following diagram equality:
\begin{equation}\label{r1=l1}
 \begin{tikzpicture}[anchorbase]
      \draw[dotted, very thick] (4, 8)node[above] {$\ \mathbbm {1}_{\mi}$} to (4, 6.8) node[below] {$\ \mathbbm {1}_{\mi}$};
      \draw[dotted, very thick] (3.5, 6.8)[out=90, in=-180]node[below] {$\ \mathbbm {1}_{\mi}$} to (4, 7.5) ;
    \end{tikzpicture}
    \quad=\quad
    \begin{tikzpicture}[anchorbase]
      \draw[dotted, very thick] (4, 8)node[above] {$\ \mathbbm {1}_{\mi}$} to (4, 6.8)node[below] {$\ \mathbbm {1}_{\mi}$};
      \draw[dotted,very thick] (4.5, 6.8)[out=90, in=0]node[below] {$\ \mathbbm {1}_{\mi}$} to (4, 7.5);
    \end{tikzpicture}
\end{equation}
This means that the evaluations of the left and the right unitors on $\mathbbm{1}_{\mi}$ coincide.
We note that the unit object in a monoidal category is usually presented by the empty diagram. But since $\mathbbm {1}_{\mi}$ is not a strict unit, we draw it as a dotted line.

Let $\cC=(\cC,\,\rI=\{\rI_{\mi}|\,\mi\in\ob\cC\},\,\rI'=\{\rI'_{\mi}|\,\mi\in\ob\cC\})$ be a fiax category. Define $\widehat{\cC}$ to be the (2-full) subcategory of $\overline{\cC}$ such that
\begin{itemize}
    \item $\ob\widehat{\cC}=\ob\cC$;
    \item $\widehat{\cC}(\mi,\mi)=\widehat{\cC(\mi,\mi)}$, for any object $\mi$, and $\widehat{\cC}(\mi,\mj)=\cC(\mi,\mj)$, for any objects $\mi\neq\mj$.
\end{itemize}

By definition and Proposition~\ref{prop:1 unit in abeln}, we have the following statement:

\begin{cor}\label{cor:01}
The category $\widehat\cC$ is a $2$-category with weak units, where $\mathbbm{1}_{\mi}$ is
a weak unit at $\mi$.
\end{cor}

\begin{prop}\label{prop:02}
If $\cC$ is a fiax category then the 2-category $\widehat{\cC}$ is fiat (with weak units).
\end{prop}

\begin{proof}
We first show that $\widehat{\cC}$ is finitary.
The only nontrivial condition to be checked is the indecomposability of each $\mathbbm{1}_{\mi}$.
Suppose that we have the decomposition $\mathbbm{1}_{\mi}=e_1\oplus\ldots \oplus e_r$ in $\widehat{\cC}$,
where $r>1$ is a positive integer and $e_1,\ldots,e_r$ are indecomposable $1$-morphisms in $\widehat{\cC}$.
Then we obtain $\rI_{\mi}\cong \rI_{\mi}\mathbbm{1}_{\mi}=\rI_{\mi}e_1\oplus\ldots\oplus\rI_{\mi}e_r$. Taking Lemma~\ref{prop:fiaxsplit} into account, we get a contradiction with the assumption on $\cC$ to be finitary, more precicely, with
the indecomposibility of $\rI_{\mi}$.

To show that $\widehat{\cC}$ is fiat, we now extend the weak involution $({}_-)^\star:\cC\to \cC^{\operatorname{\,op,\,op}}$ to a weak involution
from $\widehat{\cC}$ to $\widehat{\cC}^{\operatorname{\,op,\,op}}$ which, by abuse of notation, we still denote by $({}_-)^{\star}$.
We only need to define $\mathbbm{1}_{\mi}^\star=\mathbbm{1}_{\mi}$ and then prescribe
what the functor $({}_-)^{\star}$
does to the hom-spaces $\Hom_{\widehat{\ccC}}(\rF,\mathbbm{1}_{\mi})$ and $\Hom_{\widehat{\ccC}}(\mathbbm{1}_{\mi},\rG)$,
for any $1$-morphism $\rF,\rG\in\cC(\mi,\mi)$, and to the hom-space $\Hom_{\widehat{\ccC}}(\mathbbm{1}_{\mi},\mathbbm{1}_{\mi})$.
Restricting the $\mi$-th principal bilax $2$-representation $\widehat{\bP}_{\mi}:=\widehat{\cC}(\mi, {}_-)$ of $\widehat{\cC}$ to its fiax subcategory $\cC$,
we obtain a bilax $2$-representation of $\cC$. Therefore, for any $\rF\in \cC(\mi,\mi)$, we have
\begin{equation}\label{eq:04}
\Hom_{\widehat{\bP}_{\mi}(\mi)}(\rF,\mathbbm{1}_{\mi})\cong\Hom_{\widehat{\bP}_{\mi}(\mi)}(\rF\mathbbm{1}_{\mi},\mathbbm{1}_{\mi})\cong \Hom_{\widehat{\bP}_{\mi}(\mi)}(\mathbbm{1}_{\mi},\rF^{\star}\mathbbm{1}_{\mi})\cong \Hom_{\widehat{\bP}_{\mi}(\mi)}(\mathbbm{1}_{\mi},\rF^{\star}),
\end{equation}
where the middle isomorphism follows from Proposition~\ref{propadj} and the other two isomorphisms are defined by precomposing with the natural isomorphism $\rF\mathbbm{1}_{\mi}\xrightarrow{\cong}\rF$ and postcomposing with the natural isomorphism $\rF^{\star}\mathbbm{1}_{\mi}\xrightarrow{\cong}\rF^{\star}$, respectively.
Using the isomorphisms in \eqref{eq:04}, we can define the action of $({}_-)^{\star}$ on $\Hom_{\widehat{\ccC}}(\rF,\mathbbm{1}_{\mi})$ and $\Hom_{\widehat{\ccC}}(\mathbbm{1}_{\mi},\rG)$, respectively.

Thanks to the proof of Lemma~\ref{lemcoeq}, Definition~\ref{deffiax}~\eqref{deffiax-1} and $\mathbbm{1}_{\mi}^{\star}=\mathbbm{1}_{\mi}$, for any $1$-morphism $\rF\in\cC(\mj,\mi)$,
the pair $(\rF,\rF^{\star})$ is an adjoint pair in $\widehat{\cC}$ with the adjunction morphisms given by
\[\widehat{\alpha}_{\rF}:\xymatrix{\mathbbm{1}_{\mj}\ar[r]^{(\xi_{\mj})^{\star}}&\rI'_{\mj}\ar[r]^{\alpha_{\rF}\ }&\rF^{\star}\rF}\quad\text{and}\quad
    \widehat{\beta}_{\rF}:\xymatrix{\rF\rF^{\star}\ar[r]^{\ \beta_{\rF}}& \rI_{\mi}\ar[r]^{\xi_{\mi}}&\mathbbm{1}_{\mi}},\]
    where $\alpha_{\rF},\beta_{\rF}$ are the adjunction morphisms for the adjoint pair $(\rF,\rF^{\star})$ in $\cC$.

It remains to show that each $1$-morphism $\mathbbm{1}_{\mi}$ is self adjoint with the adjunction morphisms given by the left (or right) unitors evaluating at $\mathbbm{1}_{\mi}$ and its inverse, i.e.,
\begin{equation}\label{eq:05}
\widehat{l}_{\mi,\mi,\mathbbm{1}_{\mi}}=\widehat{r}_{\mi,\mi,\mathbbm{1}_{\mi}}:\mathbbm{1}_{\mi}\mathbbm{1}_{\mi}\xrightarrow{\cong}\mathbbm{1}_{\mi}
\quad\text{and}\quad (\widehat{l}_{\mi,\mi,\mathbbm{1}_{\mi}})^{\mone}=(\widehat{r}_{\mi,\mi,\mathbbm{1}_{\mi}})^{\mone}:
\mathbbm{1}_{\mi}\xrightarrow{\cong}\mathbbm{1}_{\mi}\mathbbm{1}_{\mi}.
\end{equation}
We need to check \eqref{diag:adjunction conditions}. The left  condition in \eqref{diag:adjunction conditions} is satisfied since
\begin{equation*}
\begin{tikzpicture}[anchorbase,scale=.7]
      \draw[dotted,very thick] (3, 7.35) to (3, 5.3)node[below] {$\mathbbm{1}_{\mi}$};
      \draw[dotted, very thick] (4.5, 7)[out=-90,in=0]
      to (3, 6); \draw (4.2,6.3)node[below] {$\mathbbm{1}_{\mi}$};
          \draw[dotted, very thick] (3, 7.5)[out=90, in=-180] to (3.5, 8);
      \draw[dotted, very thick] (4, 7.55)[out=90, in=0] to (3.5, 8);
      \draw[dotted, very thick] (4, 7.4)[out=-90,in=-180]  to (4.5, 7); \draw (4.2,7.5)node[above] {$\mathbbm{1}_{\mi}$};
             \draw[dotted, very thick] (5,7.55)[out=-90,in=0] to (4.5,7);
                \draw[dotted, very thick] (5,7.7) to (5,9.7) node[above] {$\mathbbm{1}_{\mi}$};
                \draw[dotted, very thick] (3.5, 8)[out=90, in=-180] to (5, 9); \draw (4.2,8.9)node[above] {$\mathbbm{1}_{\mi}$};
    \end{tikzpicture}
\stackrel{\eqref{eq:05}}{{=\joinrel=\joinrel=}}
    \begin{tikzpicture}[anchorbase,scale=.7]
      \draw[dotted,very thick] (3, 7.85) to (3, 5.3)node[below] {$\mathbbm{1}_{\mi}$};
      \draw[dotted, very thick] (5, 7)[out=-90,in=0]
      to (3, 6); \draw (4.2,6)node[below] {$\mathbbm{1}_{\mi}$};
      \draw[dotted, very thick] (4, 7.55)[out=90, in=0] to (3, 8.36);
      \draw[dotted, very thick] (4, 7.4)[out=-90,in=-180]  to (4.97, 6.66); \draw (4.35,7.5)node[above] {$\mathbbm{1}_{\mi}$};
                \draw[dotted, very thick] (5,7.15) to (5,9.7) node[above] {$\mathbbm{1}_{\mi}$};
                \draw[dotted, very thick] (3, 8)[out=90, in=-180] to (5, 9); \draw (4.2, 9)node[above] {$\mathbbm{1}_{\mi}$};
    \end{tikzpicture}
       \stackrel{\eqref{diag:lrcompatible}}{=\joinrel=\joinrel=}
    \begin{tikzpicture}[anchorbase,scale=.7]
      \draw[dotted,very thick] (3, 7.85) to (3, 5.3)node[below] {$\mathbbm{1}_{\mi}$};
      \draw[dotted, very thick] (5, 7)[out=-90,in=0]
      to (3, 6); \draw (4,6)node[below] {$\mathbbm{1}_{\mi}$};
      \draw[dotted, very thick] (4, 7.55)[out=90, in=-180] to (5, 8.36);
      \draw[dotted, very thick] (4, 7.4)[out=-90,in=-180]  to (4.97, 6.66); \draw (3.6,7.15)node[above] {$\mathbbm{1}_{\mi}$};
                \draw[dotted, very thick] (5,7.15) to (5,9.7) node[above] {$\mathbbm{1}_{\mi}$};
                \draw[dotted, very thick] (3, 8)[out=90, in=-180] to (5, 9); \draw (4, 9)node[above] {$\mathbbm{1}_{\mi}$};
    \end{tikzpicture}
    \stackrel{\eqref{eq:05}}{{=\joinrel=\joinrel=}}
    \begin{tikzpicture}[anchorbase,scale=.7]
      \draw[dotted,very thick] (3, 7.85) to (3, 5.3)node[below] {$\mathbbm{1}_{\mi}$};
      \draw[dotted, very thick] (5, 7)[out=-90,in=0]
      to (3, 6); \draw (4.2,6)node[below] {$\mathbbm{1}_{\mi}$};
       \draw[dotted, very thick] (5,7.15) to (5,9.7) node[above] {$\mathbbm{1}_{\mi}$};
                \draw[dotted, very thick] (3, 8)[out=90, in=-180] to (5, 9); \draw (4.2, 9)node[above] {$\mathbbm{1}_{\mi}$};
    \end{tikzpicture}   \stackrel{\eqref{eq:05}}{{=\joinrel=\joinrel=}}
     \begin{tikzpicture}[anchorbase,scale=.7]
      \draw[dotted,very thick] (3, 9.7)node[above] {$\mathbbm{1}_{\mi}$} to (3, 5.3)node[below] {$\mathbbm{1}_{\mi}$};
      \draw[dotted, very thick] (4.5, 7.45)[out=-90,in=0]
      to (3, 6); \draw (4.8,8.8)node[below] {$\mathbbm{1}_{\mi}$};
                \draw[dotted, very thick] (3, 9)[out=0, in=90] to (4.5, 7.55);
                \end{tikzpicture}\stackrel{\eqref{eq:05}}{{=\joinrel=\joinrel=}}
                \begin{tikzpicture}[anchorbase,scale=.7]
      \draw[dotted,very thick] (3, 9.7)node[above] {$\mathbbm{1}_{\mi}$} to (3, 5.3)node[below] {$\mathbbm{1}_{\mi}$};
                \end{tikzpicture}
               \quad .
\end{equation*}
The right condition in \eqref{diag:adjunction conditions} can be checked similarly.
\end{proof}

If we view $\widehat{\cC}$ as a fiax category $(\widehat{\cC},\, \widehat{\rI}=\widehat{\rI}'=\{\mathbbm{1}_{\mi}|\,\mi\in\ob\widehat{\cC}\})$, then the embedding $\cC\hookrightarrow\widehat{\cC}$ is a bilax-unital $2$-functor. But the embedding $\widehat{\cC}\hookrightarrow\overline{\cC}$
might not be a bilax-unital $2$-functor in general.
Indeed, if the latter were true, by Definition~\ref{bilax-unital 2-functor}~\eqref{bilax-unital 2-functor-2}, we would have that the isomorphism
$\widehat{l}_{\mi,\mi,\mathbbm{1}_{\mi}}:\mathbbm{1}_{\mi}\mathbbm{1}_{\mi}\xrightarrow{\cong}\mathbbm{1}_{\mi}$ factors through the $1$-morphism $\rI_{\mi}\mathbbm{1}_{\mi}$.
The latter $1$-morphism is isomorphic to $\rI_{\mi}$ which would imply that $\mathbbm{1}_{\mi}$ is a direct summand of $\rI_{\mi}$. Due to the indecomposability of $\rI_{\mi}$, we thus would have $\mathbbm{1}_{\mi}\cong\rI_{\mi}$, which would mean that $\cC$ is a fiat $2$-category (with weak or strict identity $1$-morphisms).
Conversely, if $\cC$ is a fiat $2$-category viewed as a fiax category $(\cC,\, \rI=\rI'=\{\mathbbm{1}_{\mi}|\,\mi\in\ob\cC\})$,
then the associated fiat $2$-category $\widehat{\cC}$ satisfies $\widehat{\cC}=\cC$.

\begin{rem}\label{rem:oplaxside1}
As in Remark~\ref{rem:oplaxside}, one can do the opposite constructions by using split oplax unitors in the set-up of the injective abelianization $\underline{\cC}$ and obtain the dual statements of Proposition~\ref{prop:1 unit in abeln}, Corollary~\ref{cor:01} and Proposition~\ref{prop:02}.
This gives rise to a fiat $2$-category $\widecheck{\cC}$ with weak units $\mathbbm{1}'_{\mi}$  which is a (2-full) subcategory of the 2-category $\underline{\cC}$ with weak units $\mathbbm{1}'_{\mi}$.
We draw the canonical map $\xi'_{\mi}: \mathbbm{1}'_\mi\to \rI'_\mi$ from our construction as
follows:
\begin{equation*}
\begin{tikzpicture}[anchorbase]
      \draw[very thick, red] (4, 8.4)node[above] {$\ \rI'_{\mi}$} to (4, 7.5);
      \draw[dotted ,very thick] (4, 6.6)node[below]{$\ \mathbbm{1}'_{\mi}$} to (4, 7.5);
    \end{tikzpicture}
\end{equation*}
Note that the $1$-morphisms $\rI_\mi,\rI'_\mi$, when viewed in the fiat 2-category $\widecheck{\cC}$ with weak units $\mathbbm{1}'_{\mi}$,
lose their status of lax/oplax units and thus are presented by solid lines rather than dashed lines. As a shorthand for a
more complicated diagram,  see \eqref{eq:04}, we let:
\begin{equation}\label{diag:shorthand dual d}
\begin{tikzpicture}[anchorbase]
      \draw[dotted, very thick] (4, 8.4)node[above] {$\ \mathbbm{1}'_{\mi}$} to (4, 7.5);
      \draw[blue ,very thick] (4, 6.6)node[below]{$\ \rI_{\mi}$} to (4, 7.5);
    \end{tikzpicture}
:=\
\left(
\begin{tikzpicture}[anchorbase]
      \draw[very thick, red] (4, 8.4)node[above] {$\ \rI'_{\mi}$} to (4, 7.5);
      \draw[dotted ,very thick] (4, 6.6)node[below]{$\ \mathbbm{1}'_{\mi}$} to (4, 7.5);
    \end{tikzpicture}
\right)^\star
\end{equation}

We denote the natural isomorphisms for $\mathbbm{1}'_\mi$ acting on $\underline{\cC}$ as follows:
\[\widecheck{l}'_{\mj,\mi,\rF}:\rF\cong\mathbbm{1}'_{\mi}\rF\quad\text{and}\quad\widecheck{r}'_{\mi,\mj,\rG}:\rG\cong\rG\mathbbm{1}'_{\mi},\]
for any $1$-morphisms $\rF\in\underline{\cC}(\mj,\mi),\ \rG\in\underline{\cC}(\mi,\mj)$ and any $\mj\in\ob\cC$.
The unitors for $\mathbbm{1}'_\mi$ are the restrictions of $\widecheck{l}'_{\mi}=(\widecheck{l}'_{\mj,\mi})_{\mj\in\ob\ccC}$ and
$\widecheck{r}'_{\mi}=(\widecheck{r}'_{\mi,\mj})_{\mj\in\ob\ccC}$ to $\widecheck{\cC}$.
For later use, we list the following commutative diagrams:
\begin{equation}\label{eq:001'}
\xymatrix{\mathbbm{1}'_{\mi}\rF\ar[rr]^{\xi'_{\mi}\id_{\rF}}&&\rI_{\mi}'\rF\\
\rF\ar[urr]_{l'_{\mj,\mi,\rF}}\ar[u]^{\cong}&&}\quad\text{and}\quad
\xymatrix{\rG\mathbbm{1}'_{\mi}\ar[rr]^{\id_{\rG}\xi'_{\mi}}&&\rG\rI_{\mi}',\\
\rG\ar[urr]_{r'_{\mi,\mj,\rG}}\ar[u]^{\cong}&&}
\end{equation}
for any $1$-morphism $\rF\in\underline{\cC}(\mj,\mi),\ \rG\in\underline{\cC}(\mi,\mj)$ and any $\mj\in\ob\cC$.
Diagrammatically, we have:
\begin{equation}\label{diag:unitor from d}
\begin{tikzpicture}[anchorbase]
 \draw[very thick] (4, 8.1) to (4, 6)node[below] {$\rF$} (4, 8.19)node[above] {$\rF$};
\draw[dotted, very thick] (3.4, 7.32)[out=-90,in=-180] to (4, 6.7) (3.4,6.9) node[below]{$\mathbbm{1}'_{\mi}$};
\draw[red,very thick] (3.4, 7.4) to (3.4, 8.1)node[above]{$\ \rI'_{\mi}$};
\end{tikzpicture}
=
\begin{tikzpicture}[anchorbase]
 \draw[very thick] (4, 8.1) to (4, 6)node[below] {$\rF$} (4, 8.19)node[above] {$\rF$};
\draw[red, very thick] (3.4, 7.4)[out=-90,in=-180] to (4, 6.7);
\draw[red,very thick] (3.4, 7.4) to (3.4, 8.1)node[above]{$\ \rI'_{\mi}$};
\end{tikzpicture}
  \quad\text{and}\quad
\begin{tikzpicture}[anchorbase]
 \draw[very thick] (4, 8.1) to (4, 6)node[below] {$\rG$} (4, 8.19)node[above] {$\rG$};
\draw[dotted, very thick] (4.6, 7.32)[out=-90, in=0] to (4, 6.7)(4.6,6.9) node[below]{$\mathbbm{1}'_{\mi}$};
\draw[red,very thick] (4.6, 7.4) to (4.6, 8.1)node[above]{$\ \rI'_{\mi}$};
\end{tikzpicture}
=
\begin{tikzpicture}[anchorbase]
 \draw[very thick] (4, 8.1) to (4, 6)node[below] {$\rG$} (4, 8.19)node[above] {$\rG$};
\draw[red, very thick] (4.6, 7.4)[out=-90, in=0] to (4, 6.7);
\draw[red,very thick] (4.6, 7.4) to (4.6, 8.1)node[above]{$\ \rI'_{\mi}$};
\end{tikzpicture}
\end{equation}

Note that, if $\rF\in\cC(\mj,\mi),\ \rG\in\cC(\mi,\mj)$, we also have the following commutative diagrams involving $\rI_\mi$:
\begin{equation}\label{eq:001'dual}
\xymatrix{
\rI_{\mi}\rF \ar[d]^{l_{\mj,\mi,\rF}}\ar[rr]^{(\xi_{\mi}')^\star\id_{\rF}}&&\mathbbm{1}'_{\mi}\rF\\
\rF\ar[urr]_{\cong}&&}\qquad
\text{and}\qquad
\xymatrix{\rG\rI_{\mi} \ar[d]^{r_{\mj,\mi,\rG}}\ar[rr]^{\id_{\rG}(\xi_{\mi}')^\star}&&\rG\mathbbm{1}'_{\mi}\\
\rG\ar[urr]_{\cong}&&}.
\end{equation}

For example, the second diagram in \eqref{eq:001'dual} follows from \eqref{eq:001'} as follows:
\begin{equation}\label{eq:000-1}
\begin{tikzpicture}[anchorbase]
 \draw[very thick] (4, 8.1)node[above] {$\rF$} to (4, 6)node[below] {$\rF$};
 \draw[blue, very thick] (4.6, 6.6)[out=90, in=0]to (4, 7.3);
 \draw[blue, very thick] (4.6, 6.6) to (4.6, 6)node[below] {$\ \rI_{\mi}$};
\end{tikzpicture}
=\
\left(
\begin{tikzpicture}[anchorbase]
 \draw[very thick] (4, 8.1) to (4, 6)node[below] {${}^\star\rF$} (4, 8.19)node[above] {${}^\star\rF$};
\draw[red, very thick] (3.4, 7.4)[out=-90,in=-180] to (4, 6.7);
\draw[red,very thick] (3.4, 7.4) to (3.4, 8.1)node[above]{$\ \rI'_{\mi}$};
\end{tikzpicture}
\right)^\star
\stackrel{\eqref{diag:unitor from d}}{{=\joinrel=\joinrel=}}
\left(
\begin{tikzpicture}[anchorbase]
 \draw[very thick] (4, 8.1) to (4, 6)node[below] {${}^\star\rF$} (4, 8.19)node[above] {${}^\star\rF$};
\draw[dotted, very thick] (3.4, 7.32)[out=-90,in=-180] to (4, 6.7) (3.4,6.9) node[below]{$\mathbbm{1}'_{\mi}$};
\draw[red,very thick] (3.4, 7.4) to (3.4, 8.1)node[above]{$\ \rI'_{\mi}$};
\end{tikzpicture}
\right)^\star
=
\begin{tikzpicture}[anchorbase]
 \draw[very thick] (4, 8.1)node[above] {$\rF$} to (4, 6)node[below] {$\rF$};
\draw[dotted, very thick] (4.6, 6.67)[out=90, in=0]to (4, 7.3) (4.6, 7.1)node[above] {$\mathbbm{1}'_{\mi}$};
 \draw[blue, very thick] (4.6, 6.6) to (4.6, 6)node[below] {$\ \rI_{\mi}$};
\end{tikzpicture}
\end{equation}
Here $^\star \rF$ is the left adjoint of $\rF$ (i.e., the image under the inverse of $({}_-)^\star$). The third equality is due to the fact that $({}_-)^\star:\cC\to \cC^{\operatorname{op,op}}$ is compatible with horizontal and vertical compositions, (the dual of) Proposition~\ref{prop:02}, and \eqref{diag:shorthand dual d}.
\end{rem}

\begin{prop}\label{propblue}
Let $(\cC,\rI,\rI')$ be a fiax category. Then both \eqref{diag:laxoplaxcompatible}
and its mirror image with respect to a vertical mirror hold.
\end{prop}

\begin{proof}
We prove \eqref{diag:laxoplaxcompatible} and the mirror image identity
follows by symmetry.
It is enough to prove $\eqref{diag:laxoplaxcompatible}$ in the fiat 2-category $(\widecheck{\cC},\{\mathbbm{1}'_{\mi}|\,\mi\in\ob\cC\})$.
Using Remark~\ref{rem:oplaxside1} we have:
\begin{equation*}
\begin{tikzpicture}[anchorbase]
   \draw[blue, very thick] (4.6, 8)  to (4.6, 6.7) node[below] {$\ \rI_{\mi}$};
      \draw[very thick, red] (4, 8)[out=-90,in=-180] to (4.6, 7.5);
   \draw[very thick, red] (4, 9.3)node[above] {$\ \rI'_{\mi}$} to (4, 8);
      \draw[blue ,very thick] (4.6, 8)[out=90, in=0] to (4, 8.5);
\end{tikzpicture}
\stackrel{\eqref{diag:unitor from d},\eqref{eq:000-1}}{=\joinrel=\joinrel=\joinrel=\joinrel=\joinrel=\joinrel=}
\begin{tikzpicture}[anchorbase]
   \draw[blue, very thick] (4.6, 8.23)  to (4.6, 6.7) node[below] {$\ \rI_{\mi}$};
      \draw[very thick, dotted] (4, 7.7)[out=-90,in=-180] to (4.6, 7.2) (4.1, 7.3)node[below] {$\mathbbm{1}'_{\mi}$};
   \draw[very thick, red] (4, 9.3)node[above] {$\ \rI'_{\mi}$} to (4, 7.77);
      \draw[dotted ,very thick] (4.6, 8.3)[out=90, in=0] to (4, 8.8) (4.6, 8.6)node[above] {$\mathbbm{1}'_{\mi}$};
\end{tikzpicture}
\ \ \ \stackrel{\eqref{diag:sliding}}{=\joinrel=\joinrel=}\ \ \
\begin{tikzpicture}[anchorbase]
   \draw[blue, very thick] (4.6, 7.65)  to (4.6, 6.7) node[below] {$\ \rI_{\mi}$};
    \draw[dotted, very thick] (4.6, 8.23)  to (4.6, 7.72);
      \draw[very thick, dotted] (4, 7.71)[out=-90,in=-180] to (4.6, 7.2) (4.1, 7.3)node[below] {$\mathbbm{1}'_{\mi}$};
   \draw[very thick, red] (4, 9.3)node[above] {$\ \rI'_{\mi}$} to (4, 8.35);
     \draw[dotted, very thick] (4, 8.27)  to (4, 7.76);
      \draw[dotted ,very thick] (4.6, 8.3)[out=90, in=0] to (4, 8.8) (4.6, 8.6)node[above] {$\mathbbm{1}'_{\mi}$};
\end{tikzpicture}
\ \ \ \ \ =\ \ \ \ \
\begin{tikzpicture}[anchorbase]
   \draw[blue, very thick] (4.6, 7.13)  to (4.6, 6.7) node[below] {$\ \rI_{\mi}$};
      \draw[very thick, dotted] (4, 8)[out=-90,in=-180] to (4.6, 7.5)(4.1, 7.55)node[below] {$\mathbbm{1}'_{\mi}$};
       \draw[dotted, very thick] (4.6, 7.18)  to (4.6, 7.93);
        \draw[dotted, very thick] (4, 8.07)  to (4, 8.85);
   \draw[very thick, red] (4, 9.3)node[above] {$\ \rI'_{\mi}$} to (4, 8.87);
      \draw[dotted ,very thick] (4.6, 7.96)[out=90, in=0] to (4, 8.54)(4.6, 8.35)node[above] {$\mathbbm{1}'_{\mi}$};
\end{tikzpicture}
\ \ \ \stackrel{\eqref{eq:05}}{=\joinrel=\joinrel=}\ \ \
\begin{tikzpicture}[anchorbase]
   \draw[blue, very thick] (4, 7.24)  to (4, 6.7) node[below] {$\ \rI_{\mi}$};
        \draw[dotted, very thick] (4, 8.63)  to (4, 7.37) (4.4,8);
   \draw[very thick, red] (4, 9.3)node[above] {$\ \rI'_{\mi}$} to (4, 8.76);
\end{tikzpicture}
\end{equation*}
Here the third equality holds by the naturality of the unitors for $\mathbbm{1}_\mi$,
see Proposition~\ref{prop:1 unit in abeln}. For the last equality, \eqref{eq:05} is used twice, see the last part of the diagram computation in the proof of Proposition~\ref{prop:02}.
The claim follows.
\end{proof}

\subsection{Lifting bilax $2$-representations}\label{sec:lifted-2-rep}

Let $\cC=(\cC,\,\rI=\{\rI_{\mi}|\,\mi\in\ob\cC\},\,\rI'=\{\rI'_{\mi}|\,\mi\in\ob\cC\})$
be a fiax ca\-te\-gory and
$\bM$ a non-zero bilax $2$-representation of $\cC$ in the sense that
$\bM(\mi)$ is non-zero, for some $\mi\in\ob\cC$.
If $\bM(\rF)=0$, for all $1$-morphisms $\rF\in\cC$, we will call $\bM$ a {\em trivial} bilax $2$-representation of $\cC$.
None of the trivial bilax $2$-representations of $\cC$ can be lifted to a non-zero bilax $2$-representation of $\widehat{\cC}$
since, for any non-zero bilax $2$-representation $\bN$ of $\widehat{\cC}$, we have $\bN(\mathbbm{1}_{\mi})\cong \id_{\bN(\mi)}\neq 0$
due to the isomorphism $\mathbbm{1}_{\mi}\mathbbm{1}_{\mi}\xrightarrow{\cong}\mathbbm{1}_{\mi}$ and \eqref{eq:bilax-2-rep}.

For any non-zero bilax $2$-representation $\bM$ of $\cC$, define a bilax-unital $2$-functor $\cC\,\bM:\cC\to\mathbf{Cat}$ by setting
\[\cC\,\bM(\mi):=\mathrm{add}(\{\bM(\rF)X\ |\ X\in\bM(\mj),\ \rF\in\cC(\mj,\mi),\  \mj\in\ob\cC \})\subseteq \bM(\mi),\]
for every object $\mi\in\ob\cC$.
Hence $\cC\,\bM$ is a bilax $2$-subrepresentation of $\bM$. For any fiat $2$-category $\cC$,  viewed as a fiax category $(\cC,\,\rI=\rI'=\{\mathbbm{1}_{\mi}|\,\mi\in\ob\cC\})$, and for
any bilax $2$-representation $\bM$ of $\cC$, we have $\bM=\cC\,\bM$.

Let $\bM$ be a non-zero bilax $2$-representation of a fiax category $\cC=(\cC,\,\rI=\{\rI_{\mi}|\,\mi\in\ob\cC\},\,\rI'=\{\rI'_{\mi}|\,\mi\in\ob\cC\})$.
Consider the projective abelianization $\overline{\bM}$ as a bilax $2$-representation of $\overline{\cC}$.
Restricting $\overline{\bM}$ to $\widehat{\cC}$, we obtain a functor, $\overline{\bM}|_{\widehat{\ccC}}:\widehat{\cC}\to\mathbf{Cat}$,
which might not be a bilax $2$-representation of $\widehat{\cC}$ unless $\cC$ is a fiat $2$-category with weak or strict identity $1$-morphisms.
Define a weak $2$-functor $\widehat{\bM}:\widehat{\cC}\to\mathbf{Cat}$ by setting
\[\widehat{\bM}(\mi):=\mathrm{add}(\{\overline{\bM}(\rF) X\ |\ X\in\bM(\mj),\ \rF\in\cC(\mj,\mi),\ \mj\in\ob\cC  \})
\subseteq \bM(\mi)\subseteq \overline{\bM}(\mi),\]
for every object $\mi\in\ob\widehat{\cC}$. Then $\widehat{\bM}$ is a $2$-representation of $\widehat{\cC}$ since
\[\widehat{\bM}(\mathbbm{1}_{\mi})(\overline{\bM}(\rF)X)=\overline{\bM}(\mathbbm{1}_{\mi})(\overline{\bM}(\rF)X)=\overline{\bM}(\mathbbm{1}_{\mi}\rF)X\cong
\overline{\bM}(\rF)X,\]
for $X\in\bM(\mj)$ and $\rF\in\cC(\mj,\mi)$. It follows directly from the definition that $\cC\,\bM=\widehat{\bM}|_{\ccC}$.

On the other hand, for any non-zero $2$-representation $\bN$ of $\widehat{\cC}$,  precomposing $\bN$ with the embedding $\cC\hookrightarrow\widehat{\cC}$,
we obtain a non-zero bilax $2$-representation $\bN|_{\ccC}$ of $\cC$. Thanks to the fact that $\rI_{\mi}\to\mathbbm{1}_{\mi}$ is an epimorphism and $\bN(\mathbbm{1})\cong\id_{\bN(\mi)}$, we know that $\bN(\rI_{\mi})\neq0$. Therefore, we further get a non-zero bilax $2$-representation $\cC\,(\bN|_{\ccC})$
of $\cC$ given by setting
\[\cC\,(\bN|_{\ccC})(\mi):=\mathrm{add}(\{\bN(\rF)X\ |\ X\in\bN(\mj),\ \rF\in\cC(\mj,\mi),\ \mj\in\ob\cC \})\subseteq \bN(\mi),\]
for every object $\mi\in\ob\cC$. Since each $l_{\mj,\mi,\rF}$ is split, for any $1$-morphism $\rF$,
the bilax $2$-representation $\cC\,(\bN|_{\ccC})$ automatically satisfies the condition $\cC\,\cC\,(\bN|_{\ccC})=\cC\,(\bN|_{\ccC})$.
Indeed, in each category $\widehat{\cC}(\mi,\mi)$ we have an epimorphism $\xi_{\mi}:\rI_{\mi}\to\mathbbm{1}_{\mi}$ which implies that
$\cC\,(\bN|_{\ccC})(\mi)=\bN(\mi)$ as categories, for every object $\mi\in\ob\cC$.

\begin{thm}\label{thm:bijection}
Let $\cC$ be a fiax category.
There is a bijection between the set of bilax
$2$-representations $\bM$ of $\cC$ satisfying $\cC\,\bM=\bM$ and
the set of $2$-representations of the associated fiat $2$-category $\widehat{\cC}$
(with weak units).
\end{thm}

\begin{proof}
Note that the zero bilax $2$-representation $\bM$ of a fiax category $\cC$ automatically satisfies $\cC\,\bM=\bM$. Define a
map $\Psi$ from the set of bilax
$2$-representations $\bM$ of a fiax category $\cC$ satisfying $\cC\,\bM=\bM$ to
the set of $2$-representations of the associated $\widehat{\cC}$ by sending $\bM$
to $\widehat{\bM}$. Define a
map $\Omega$ on the other direction by sending $\bN$ to $\cC\,(\bN|_{\ccC})$.
It follows from the definitions that $\Omega\Psi=\id$ and $\Psi\Omega=\id$.
\end{proof}

By the proof of~Theorem~\ref{thm:bijection}, we have:

\begin{cor}\label{cor:00}
Let $\cC$ be a fiax category.
There is a biequivalence between the $2$-category of bilax
$2$-rep\-res\-en\-ta\-ti\-ons $\bM$ of $\cC$ satisfying that $\cC\,\bM=\bM$ and
the $2$-category of $2$-representations of the associated fiat $2$-categories
$\widehat{\cC}$ (with weak units).
\end{cor}

As a bonus, it is clear that the bijection~in Theorem~\ref{thm:bijection} also induced a bijection between the corresponding sets of equivalence classes.

\begin{rem}\label{rem:oplaxside2}
In the set-up of injective abelianizations, one obtains the dual statements of Theorem~\ref{thm:bijection} and Corollary~\ref{cor:00} with respect to the associated fiat $2$-category $\widecheck{\cC}$ with weak units.
\end{rem}

\subsection{Lifting (co)algebras $1$-morphisms}\label{s3.4}

Recall that, for any fiat $2$-category $\cC$ with a choice of a family of morphisms $\gamma_{\mi}:\rI_{\mi}\to\mathbbm{1}_{\mi}, \gamma'_{\mi}:\mathbbm{1}_{\mi}\to\rI'_{\mi}$, we have the bilax-unital $2$-category $\cC[\gamma,\gamma']$. Following the construction before Corollary~\ref{cor:01}, we obtain that $\widehat{\cC[\gamma,\gamma']}$ is biequivalent to $\cC$ as a $2$-category with weak units.

From now on, we assume that $\cC=(\cC,\,\rI=\{\rI_{\mi}|\,\mi\in\ob\cC\},\,\rI'=\{\rI'_{\mi}|\,\mi\in\ob\cC\})$ is a fiax category (unless explicitly stated otherwise).
By the previous paragraph, the following proposition gives, in some sense,
a converse to Lemma~\ref{algebraisalgebra}.

\begin{prop}\label{algebra lifts to 2cat}
Any algebra $1$-morphism $\rF=(\rF,\,\mu_{\rF},\,\eta_{\rF})$ in $\cC(\mi,\mi)$ with respect to the lax unit $\rI_{\mi}$
is also an algebra $1$-morphism in $\widehat{\cC}(\mi,\mi)$ with respect to $\mathbbm{1}_{\mi}$.
\end{prop}

\begin{proof}
We claim that $\eta_{\rF}:\rI_{\mi}\to\rF$ equalizes $l_{\mi,\mi,\rI_{\mi}}$ and $r_{\mi,\mi,\rI_{\mi}}$. Since the $2$-morphism
$$r_{\mi,\mi,\, \rI_{\mi}\rI_{\mi}}=\id_{\rI_{\mi}}r_{\mi,\mi,\rI_{\mi}}:\rI_{\mi}\rI_{\mi}\rI_{\mi}\to\rI_{\mi}\rI_{\mi}$$ is epic, it suffices to show that
\[\eta_{\rF}\circ_{\rv}l_{\mi,\mi,\rI_{\mi}}\circ_{\rv}r_{\mi,\mi,\,\rI_{\mi}\rI_{\mi}}=\eta_{\rF}\circ_{\rv}r_{\mi,\mi,\rI_{\mi}}\circ_{\rv}r_{\mi,\mi,\,\rI_{\mi}\rI_{\mi}},\]
that is, in diagrammatic terms, that we have the following:
\begin{equation*}
\begin{tikzpicture}[anchorbase]
\draw[very thick] (4, 7.5) to (4, 8.5)node[above] {$\rF$};
\draw[dashed  , very thick] (4, 7.4) to (4, 6)node[below] {$\ \rI_{\mi}$};
\draw[dashed  , very thick] (3.3, 6)[out=90, in=-180]node[below] {$\ \rI_{\mi}$} to (4, 7.1);
      \draw[dashed  , very thick] (4.5, 6)[out=90, in=0]node[below] {$\ \rI_{\mi}$} to (4, 6.7);
\end{tikzpicture}
\quad=\quad
\begin{tikzpicture}[anchorbase]
\draw[very thick] (3, 7.5) to (3, 8.5)node[above] {$\rF$};
\draw[dashed  , very thick] (3.7, 6)node[below] {$\ \rI_{\mi}$} [out=90, in=0] to (3, 7.1);
\draw[dashed  , very thick] (3, 6)node[below] {$\ \rI_{\mi}$} to (3, 7.4);
      \draw[dashed  , very thick] (4.2, 6)[out=90, in=0]node[below] {$\ \rI_{\mi}$} to (3.6, 6.67);
\end{tikzpicture}
\end{equation*}
Indeed, by a diagrammatic computation, we have:
\begin{equation}
\begin{split}
\begin{tikzpicture}[anchorbase]
\draw[very thick] (4, 7.5) to (4, 8.5)node[above] {$\rF$};
\draw[dashed  , very thick] (4, 7.4) to (4, 6)node[below] {$\ \rI_{\mi}$};
\draw[dashed  , very thick] (3.3, 6)[out=90, in=-180]node[below] {$\ \rI_{\mi}$} to (4, 7.1);
      \draw[dashed  , very thick] (4.5, 6)[out=90, in=0]node[below] {$\ \rI_{\mi}$} to (4, 6.7);
\end{tikzpicture}
&\stackrel{\eqref{diag:unitornatural}}{=\joinrel=\joinrel=}
\begin{tikzpicture}[anchorbase]
\draw[very thick] (4, 7.5) to (4, 8.5)node[above] {$\rF$};
\draw[dashed  , very thick] (4, 7.4) to (4, 6)node[below] {$\ \rI_{\mi}$};
\draw[dashed  , very thick] (3.1, 6)[out=90, in=-180]node[below] {$\ \rI_{\mi}$} to (4, 7.8);
      \draw[dashed  , very thick] (4.5, 6)[out=90, in=0]node[below] {$\ \rI_{\mi}$} to (4, 6.7);
\end{tikzpicture}\stackrel{\eqref{diag:algunit}}{=\joinrel=\joinrel=}
\begin{tikzpicture}[anchorbase]
\draw[very thick] (3.5, 8.5)node[above] {$\rF$} to (3.5, 8.1);
   \draw[very thick] (4,7.6) [out=90, in=0] to (3.5, 8.1);
    \draw[very thick] (3, 7.65)[out=90, in=-180]  to (3.5, 8.1); \draw (2.8, 7.55) node[above]{$\rF$};
    \draw[very thick] (4, 7.25) to (4, 7.6); \draw (4.25, 7.55)node[above]{$\rF$};
\draw[dashed  , very thick] (4, 7.17) to (4, 6)node[below] {$\ \rI_{\mi}$};
\draw[dashed  , very thick] (3, 6)node[below] {$\ \rI_{\mi}$} to (3, 7.6);
    \draw[dashed  , very thick] (4.5, 6)[out=90, in=0]node[below] {$\ \rI_{\mi}$} to (4, 6.7);
\end{tikzpicture}
\stackrel{\eqref{diag:sliding}}{=\joinrel=\joinrel=}
\begin{tikzpicture}[anchorbase]
\draw[very thick] (3.5, 8.5)node[above] {$\rF$} to (3.5, 8.1);
   \draw[very thick] (4,7.65) [out=90, in=0] to (3.5, 8.1);
    \draw[very thick] (3, 7.6)[out=90, in=-180]  to (3.5, 8.1); \draw (2.8, 7.55) node[above]{$\rF$};
    \draw[very thick] (3, 7.25) to (3, 7.6); \draw (4.25, 7.55)node[above]{$\rF$};
\draw[dashed  , very thick] (3, 7.17) to (3, 6)node[below] {$\ \rI_{\mi}$};
\draw[dashed  , very thick] (4, 6)node[below] {$\ \rI_{\mi}$} to (4, 7.6);
    \draw[dashed  , very thick] (4.5, 6)[out=90, in=0]node[below] {$\ \rI_{\mi}$} to (4, 6.7);
\end{tikzpicture}\\
&\stackrel{\eqref{diag:algunit}}{=\joinrel=\joinrel=}
\begin{tikzpicture}[anchorbase]
\draw[very thick] (3, 7.5) to (3, 8.5)node[above] {$\rF$};
\draw[dashed  , very thick] (3.8, 6)node[below] {$\ \rI_{\mi}$} [out=90, in=0] to (3, 7.8);
\draw[dashed  , very thick] (3, 6)node[below] {$\ \rI_{\mi}$} to (3, 7.4);
      \draw[dashed  , very thick] (4.3, 6)[out=90, in=0]node[below] {$\ \rI_{\mi}$} to (3.73, 6.67);
\end{tikzpicture}\stackrel{\eqref{diag:unitornatural}}{=\joinrel=\joinrel=}
\begin{tikzpicture}[anchorbase]
\draw[very thick] (3, 7.5) to (3, 8.5)node[above] {$\rF$};
\draw[dashed  , very thick] (3.7, 6)node[below] {$\ \rI_{\mi}$} [out=90, in=0] to (3, 7.1);
\draw[dashed  , very thick] (3, 6)node[below] {$\ \rI_{\mi}$} to (3, 7.4);
      \draw[dashed  , very thick] (4.2, 6)[out=90, in=0]node[below] {$\ \rI_{\mi}$} to (3.6, 6.67);
\end{tikzpicture}\quad.
\end{split}
\end{equation}
Therefore, there exists a unique $2$-morphism $\zeta_{\rF}:\mathbbm{1}_{\mi}\to \rF$ such that the triangle in the diagram
\begin{equation}\label{eq:06}
\xymatrix{\rI_{\mi}\rI_{\mi}\ar@<.6ex>[rrr]^{l_{\mi,\mi,\rI_{\mi}}}
\ar@<-.4ex>[rrr]_{r_{\mi,\mi,\rI_{\mi}}}&&&\rI_{\mi}\ar[rr]^{\xi_{\mi}}\ar[drr]_{\eta_{\rF}}&&\mathbbm{1}_{\mi}
\ar@{-->}[d]^{\zeta_{\rF}}\\
&&&&&\rF}
\end{equation}
commutes. Now, we claim that $(\rF,\mu_{\rF},\zeta_{\rF})$ is an algebra in $\widehat{\cC}(\mi,\mi)$. The associativity of $\mu_{\rF}$ is clear and we only need to verify the unitality. Consider the diagram:
\[\xymatrix@R=2.5pc{&&\rF\ar[d]^{\cong}\ar@/^1.5pc/@{=}[ddrr]&&\\
&&\mathbbm{1}_{\mi}\rF\ar[d]^{\zeta_{\rF}\id_{\rF}}\ar[drr]^{\cong}&&\\
\rI_{\mi}\rF\ar[rr]^{\eta_{\rF}\id_{\rF}}\ar@/_1.5pc/[rrrr]_{l_{\mi,\mi,\rF}}\ar@/^1.5pc/[uurr]^{l_{\mi,\mi,\rF}}\ar[urr]^{\xi_{\mi}\id_{\rF}}&&\rF\rF\ar[rr]^{\mu_{\rF}}&&\rF,}\]
where the left top triangle commutes due to \eqref{eq:001} with $\theta_{\mi}(\rF):\mathrm{U}_{\mi}(\rF)\to\rF$
replaced by the $\mathbbm{1}_{\mi}\rF\xrightarrow{\cong}\rF$;
the right top triangle and the big outer triangle along the boundary of the whole diagram commute by definitions.
Note that the left unitor of $\widehat{\cC}$ gives the isomorphism $\widehat{l}_{\mi,\mi,\rF}:\mathbbm{1}_{\mi}\rF\xrightarrow{\cong}\rF$. Then we have
\[\mu_{\rF}\circ_{\rv}(\zeta_{\rF}\id_{\rF})\circ_{\rv}(\widehat{l}_{\mi,\mi,\rF})^{\mone}\circ_{\rv}l_{\mi,\mi,\rF}=
\mu_{\rF}\circ_{\rv}(\zeta_{\rF}\id_{\rF})\circ_{\rv}(\xi_{\mi}\id_{\rF})=\mu_{\rF}\circ_{\rv}(\eta_{\rF}\id_{\rF})=l_{\mi,\mi,\rF}.\]
Thanks to the fact that $l_{\mi,\mi,\rF}$ is epic, we obtain $\mu_{\rF}\circ_{\rv}(\zeta_{\rF}\id_{\rF})\circ_{\rv}(\widehat{l}_{\mi,\mi,\rF})^{\mone}=\id_{\rF}$ which yields
$\mu_{\rF}\circ_{\rv}(\zeta_{\rF}\id_{\rF})=\widehat{l}_{\mi,\mi,\rF}$.
The left middle triangle commutes by \eqref{eq:06}; the bottom triangle commutes due to the left unitality of the algebra $(\rF,\mu_{\rF},\eta_{\rF})$.
This establishes the left unitality and the right unitality can be proved by a similar argument.
\end{proof}

The dual statement of Proposition~\ref{algebra lifts to 2cat} is the following.

\begin{prop}\label{coalgebra lifts to 2cat}
Any coalgebra $1$-morphism $\rF=(\rF,\,\Delta_{\rF},\,\epsilon_{\rF})$ in $\cC(\mi,\mi)$ with respect to the oplax unit $\rI'_{\mi}$
is also a coalgebra $1$-morphism in $\widecheck{\cC}(\mi,\mi)$ with respect to $\mathbbm{1}'_{\mi}$.
\end{prop}

For later use, we denote by $(\rF,\Delta_{\rF},\zeta'_{\rF})$ the new coalgebra $1$-morphism in $\widecheck{\cC}(\mi,\mi)$ and have the
following commutative diagram:
\begin{equation}\label{eq:06'}
\xymatrix{\mathbbm{1}'_{\mi}\ar[rr]^{\xi'_{\mi}}&&\rI'_{\mi}.\\
\rF\ar[urr]_{\epsilon_{\rF}}\ar[u]^{\zeta'_{\rF}}&&}
\end{equation}

\begin{cor}\label{modulelift}
Let $\rF=(\rF,\,\mu_{\rF},\,\eta_{\rF})$ and $\rG=(\rG,\,\mu_{\rG},\,\eta_{\rG})$ be
any algebra $1$-morphisms in $\cC(\mi,\mi)$ and $\cC(\mk,\mk)$, respectively.
\begin{enumerate}[$($i$)$]
\item\label{modulelift-1}
If $\rM=(\rM,\upsilon_{\rM})$ is a left $\rF$-module in $\cC(\mj,\mi)$, then $\rM=(\rM,\upsilon_{\rM})$ is also a left $\rF$-module in $\widehat{\cC}(\mj,\mi)$, where now $\rF$ is an algebra $1$-morphism in $\widehat{\cC}(\mi,\mi)$ given by Proposition~\ref{algebra lifts to 2cat}.
\item\label{modulelift-2}
If $\rM=(\rM,\tau_{\rM})$ is a right $\rF$-module in $\cC(\mi,\mj)$, then $\rM=(\rM,\tau_{\rM})$ is also a right $\rF$-module in $\widehat{\cC}(\mi,\mj)$, where now $\rF$ is an algebra $1$-morphism in $\widehat{\cC}(\mi,\mi)$ given by Proposition~\ref{algebra lifts to 2cat}.
\item\label{modulelift-3}
If $\rM=(\rM,\upsilon_{\rM},\tau_{\rM})$ is an $\rF$-$\rG$-bimodule in $\cC(\mk,\mi)$, then $\rM=(\rM,\upsilon_{\rM},\tau_{\rM})$ is also an $\rF$-$\rG$-bimodule in $\widehat{\cC}(\mk,\mi)$, where now $\rF$ and $\rG$ are the algebra $1$-morphisms in $\widehat{\cC}(\mi,\mi)$ and $\widehat{\cC}(\mk,\mk)$, respectively, given by Proposition~\ref{algebra lifts to 2cat}.
\end{enumerate}
\end{cor}

\begin{proof}
Recall that $(\rF,\mu_{\rF},\zeta_{\rF})$ is an algebra in $\widehat{\cC}(\mi,\mi)$. It is clear that $\upsilon_{\rM}$ satisfies the associativity since the multiplication map $\mu_{\rF}$ is still the same. We only prove claim~\ref{modulelift-1}, as claim~\ref{modulelift-2}
is proved similarly and claim~\ref{modulelift-3} follows from claims~\ref{modulelift-1} and \ref{modulelift-2}. Consider the diagram:
\begin{equation}\label{eq:07}
\xymatrix@R=3pc{&&\rI_{\mi}\rM\ar[dll]_{\xi_{\mi}\id_{\rM}}\ar@/^3.2pc/[dd]^{l_{\mj,\mi,\rM}}\ar[d]^{\eta_{\rF}\id_{\rM}}\\
\mathbbm{1}_{\mi}\rM\ar[drr]^{\widehat{l}_{\mj,\mi,\rM}}_{\cong}\ar[rr]^{\ \zeta_{\rF}\id_{\rM}}&&\rF\rM\ar[d]^{\upsilon_{\rM}}\\
&&\rM,}
\end{equation}
where we use the same notation as in Lemma~\ref{algebra lifts to 2cat}.
The top left triangle in \eqref{eq:07} commutes due to~\eqref{eq:06} while the right triangle commutes thanks to the unitality of the left module structure of $\rM=(\rM,\upsilon_{\rM})$ in $\cC(\mj,\mi)$. The big outer triangle along the boundary of the whole diagram \eqref{eq:07} commutes due to \eqref{eq:001} with $\rF$ replaced by $\rM$ and $\theta_{\mi}(\rF):\mathrm{U}_{\mi}(\rF)\to\rF$
by $\widehat{l}_{\mj,\mi,\rM}:\mathbbm{1}_{\mi}\rM\xrightarrow{\cong}\rM$. Therefore we have
\[\upsilon_{\rM}\circ_{\rv}(\zeta_{\rF}\id_{\rM})\circ_{\rv}(\xi_{\mi}\id_{\rM})=\upsilon_{\rM}\circ_{\rv}(\eta_{\rF}\id_{\rM})
=l_{\mj,\mi,\rM}=\widehat{l}_{\mj,\mi,\rM}\circ_{\rv}(\xi_{\mi}\id_{\rM}).\]
As $\xi_{\mi}\id_{\rM}$ is epic, we obtain $\upsilon_{\rM}\circ_{\rv}(\zeta_{\rF}\id_{\rM})=\widehat{l}_{\mj,\mi,\rM}$. This proves the unitality of $\upsilon_{\rM}$.
\end{proof}

The dual statement of Corollary~\ref{modulelift} is as follows.

\begin{cor}\label{comodulelift}
Let $\rF=(\rF,\,\Delta_{\rF},\,\epsilon_{\rF})$ and $\rG=(\rG,\,\Delta_{\rG},\,\epsilon_{\rG})$ be any coalgebra $1$-morphisms in $\cC(\mi,\mi)$ and $\cC(\mk,\mk)$, respectively.
\begin{enumerate}[$($i$)$]
\item\label{comodulelift-1}
If $\rM=(\rM,\lambda_{\rM})$ is a left $\rF$-comodule in $\cC(\mj,\mi)$, then $\rM=(\rM,\lambda_{\rM})$ is also a left $\rF$-comodule in $\widecheck{\cC}(\mj,\mi)$, where now $\rF$ is a coalgebra $1$-morphism in $\widecheck{\cC}(\mi,\mi)$ given  by Proposition~\ref{coalgebra lifts to 2cat}.
\item\label{comodulelift-2}
If $\rM=(\rM,\rho_{\rM})$ is a right $\rF$-comodule in $\cC(\mi,\mj)$, then $\rM=(\rM,\rho_{\rM})$ is also a right $\rF$-comodule in $\widecheck{\cC}(\mi,\mj)$, where now $\rF$ is a coalgebra $1$-morphism in $\widecheck{\cC}(\mi,\mi)$ given by Proposition~\ref{coalgebra lifts to 2cat}.
\item\label{comodulelift-3}
If $\rM=(\rM,\lambda_{\rM},\rho_{\rM})$ is an $\rF$-$\rG$-bicomodule in $\cC(\mk,\mi)$, then $\rM=(\rM,\lambda_{\rM},\rho_{\rM})$ is also an $\rF$-$\rG$-bicomodule in $\widecheck{\cC}(\mk,\mi)$, where now  $\rF$ and  $\rG$ are coalgebra $1$-morphisms in $\widecheck{\cC}(\mi,\mi)$ and $\widecheck{\cC}(\mk,\mk)$, respectively, given by Proposition~\ref{algebra lifts to 2cat}.
\end{enumerate}
\end{cor}

\subsection{(Co)tensor products over (co)algebra 1-morphisms}\label{s3.5}

Let $\rF=(\rF,\,\mu_{\rF},\,\eta_{\rF})\in\cC(\mi,\mi)$ be an algebra $1$-morphisms, $\rM=(\rM,\tau_{\rM})$ a right $\rF$-module in $\cC(\mi,\mk)$ and $\rN=(\rN,\upsilon_{\rN})$ a left $\rF$-module in $\cC(\mj,\mi)$.

\begin{defn}
The {\em tensor product} of $\rM$ and $\rN$ over the algebra $1$-morphism $\rF$ is defined as the coequalizer of the two $2$-morphisms in the diagram
\[\xymatrix{\rM\rF\rN\ar@<.6ex>[rrr]^{\tau_{\rM}\id_{\rN}}
\ar@<-.4ex>[rrr]_{\id_{\rM}\upsilon_{\rN}}&&&\rM\rN}.\]
This coequalizer is a $1$-morphism in $\overline{\cC}(\mj,\mk)$
which will be denoted by $\rM\boxtimes_{\rF}\rN$.
\end{defn}

Below we list some basic properties of $\boxtimes$, cf. e.g. \cite[Section~3.3]{MMMZ}
for the classical setup.

\begin{prop}\label{prop:tensor-product}
Let $\rF=(\rF,\,\mu_{\rF},\,\eta_{\rF})\in\cC(\mi,\mi)$ be an algebra $1$-morphism.
\begin{enumerate}[$($i$)$]
\item\label{lem:tensor-product1} For any left $\rF$-module $\rN=(\rN,\upsilon_{\rN})$ in $\cC(\mj,\mi)$, we have $\rF\boxtimes_{\rF}\rN\cong\rN$ as left $\rF$-modules.
\item\label{lem:tensor-product2} For any right $\rF$-module $\rM=(\rM,\tau_{\rM})$ in $\cC(\mi,\mk)$, we have $\rM\boxtimes_{\rF}\rF\cong\rM$ as right $\rF$-modules.
\end{enumerate}
\end{prop}

\begin{proof}
We only prove the statement~\ref{lem:tensor-product1} since the statement~\ref{lem:tensor-product2} can be proved similarly.
By the associativity of $\upsilon_{\rN}$, i.e., \eqref{diag:moduleassoc-unit}, we know that the $2$-morphism $\upsilon_{\rN}$ equalizes $\mu_{\rF}\id_{\rN}$ and $\id_{\rF}\upsilon_{\rN}$.
Thus, by the universal property of coequalizers, there exists a unique $2$-morphism $\varphi:\rF\boxtimes_{\rF}\rN\to \rN$ such that the right top triangle in the  diagram
\begin{equation}\label{eq:000}
\xymatrix{\rF\rF\rN\ar@<.6ex>[rrr]^{\mu_{\rF}\id_{\rN}}
\ar@<-.4ex>[rrr]_{\id_{\rF}\upsilon_{\rN}}&&&\rF\rN\ar[drr]^{\upsilon_{\rN}}\ar[rr]^{\pi}&&\rF\boxtimes_{\rF}\rN\ar@{-->}[d]^{\varphi}\\
&&&\mathbbm{1}_{\mi}\rN\ar[rr]^{\widehat{l}_{\mj,\mi,\rN}}_{\cong}\ar[u]^{\zeta_{\rF}\id_{\rN}}&&\rN,}
\end{equation}
commutes, i.e., $\varphi\circ_{\rv}\pi=\upsilon_{\rN}$.
The left bottom triangle in \eqref{eq:000} commutes due to the unitality of $\upsilon_{\rN}$, see Corollary~\ref{modulelift}~\ref{modulelift-1}.
Now we claim that $\varphi$ and $\pi\circ_{\rv}(\zeta_{\rF}\id_{\rN})\circ_{\rv}(\widehat{l}_{\mj,\mi,\rN})^{\mone}$
are mutually inverse of each other. On the one hand, we have
\[\varphi\circ_{\rv}\pi\circ_{\rv}(\zeta_{\rF}\id_{\rN})\circ_{\rv}(\widehat{l}_{\mj,\mi,\rN})^{\mone}=\upsilon_{\rN}\circ_{\rv}(\zeta_{\rF}\id_{\rN})\circ_{\rv}(\widehat{l}_{\mj,\mi,\rN})^{\mone}
=\widehat{l}_{\mj,\mi,\rN}\circ_{\rv}(\widehat{l}_{\mj,\mi,\rN})^{\mone}=\mathrm{id}_{\rN}.\]
On the other hand,
it is enough to prove
\begin{equation}\label{eq:000-0}
\pi\circ_{\rv}(\zeta_{\rF}\id_{\rN})\circ_{\rv}(\widehat{l}_{\mj,\mi,\rN})^{\mone}\circ_{\rv}\varphi\circ_{\rv}\pi=\pi,
\end{equation}
since the above equality implies that $\pi\circ_{\rv}(\zeta_{\rF}\id_{\rN})\circ_{\rv}(\widehat{l}_{\mj,\mi,\rN})^{\mone}\circ_{\rv}\varphi=\id_{\rF\boxtimes_{\rF}\rN}$ due to the universal property of coequalizers. Consider the diagram
\begin{equation}\label{eq:08}
\xymatrix@R=2.5pc{\rF\rN\ar[d]_{\upsilon_{\rN}}\ar[rr]^{(\widehat{l}_{\mj,\mi,\rF\rN})^{\mone}}_\cong
&&\mathbbm{1}_{\mi}\rF\rN\ar[d]_{\id_{\mathbbm{1}_{\mi}}\upsilon_{\rN}}\ar[rr]^{\zeta_{\rF}\id_{\rF\rN}}&&
\rF\rF\rN\ar[d]_{\id_{\rF}\upsilon_{\rN}}\ar[rr]^{\mu_{\rF}\id_{\rN}}&&\rF\rN\ar[d]^{\pi}\\
\rN\ar[rr]^{(\widehat{l}_{\mj,\mi,\rN})^{\mone}}_{\cong}&&\mathbbm{1}_{\mi}\rN\ar[rr]^{\zeta_{\rF}\id_{\rN}}&&\rF\rN\ar[rr]^{\pi}&&\rF\boxtimes_{\rF}\rN
}
\end{equation}
where the leftmost square commutes by the naturality of the left unitor $\widehat{l}_{\mj,\mi,{}_-}$;
the middle square commutes by the interchange law and the rightmost square commutes by definition of tensor products.
Therefore we have
\begin{equation*}
\begin{split}
\pi\circ_{\rv}(\zeta_{\rF}\id_{\rN})\circ_{\rv}(\widehat{l}_{\mj,\mi,\rN})^{\mone}\circ_{\rv}\varphi\circ_{\rv}\pi
&=\pi\circ_{\rv}(\zeta_{\rF}\id_{\rN})\circ_{\rv}(\widehat{l}_{\mj,\mi,\rN})^{\mone}\circ_{\rv}\upsilon_{\rN}\\
&=\pi\circ_{\rv}(\mu_{\rF}\id_{\rN})\circ_{\rv}(\zeta_{\rF}\id_{\rF\rN})\circ_{\rv}(\widehat{l}_{\mj,\mi,\rF\rN})^{\mone}\\
&=\pi\circ_{\rv}\widehat{l}_{\mj,\mi,\rF\rN}\circ_{\rv}(\widehat{l}_{\mj,\mi,\rF\rN})^{\mone}=\pi.
\end{split}
\end{equation*}
Here the first equality follows from $\varphi\circ_{\rv}\pi=\upsilon_{\rN}$; the second equality follows from the commutativity of
the diagram \eqref{eq:08} and the third equality follows from the left unitality of the algebra $1$-morphism $(\rF,\mu_{\rF},\zeta_{\rF})$ in $\widehat{\cC}(\mi,\mi)$ and $\widehat{l}_{\mj,\mi,\rF\rN}=\widehat{l}_{\mi,\mi,\rF}\id_{\rN}$. The claim follows.
\end{proof}

Let $\rF=(\rF,\,\Delta_{\rF},\,\epsilon_{\rF})\in\cC(\mi,\mi)$ be a coalgebra $1$-morphism, $\rM=(\rM,\rho_{\rM})$ a right $\rF$-comodule in $\cC(\mi,\mk)$ and $\rN=(\rN,\lambda_{\rN})$ a left $\rF$-comodule in $\cC(\mj,\mi)$.
\begin{defn}
The {\em cotensor product} of $\rM$ and $\rN$ over the coalgebra $1$-morphism $\rF$ is defined as the equalizer of the two $2$-morphisms in the diagram
\[\xymatrix{\rM\rN\ar@<.6ex>[rrr]^{\rho_{\rM}\id_{\rN}}
\ar@<-.4ex>[rrr]_{\id_{\rM}\lambda_{\rN}}&&&\rM\rF\rN}.\]
This is a $1$-morphism in $\underline{\cC}(\mj,\mk)$, which we denote by $\rM\Box_{\rF}\rN$.
\end{defn}

The dual statement of Corollary~\ref{modulelift} is the following.

\begin{prop}\label{prop:cotensor-product}
Let $\rF=(\rF,\,\Delta_{\rF},\,\epsilon_{\rF})\in\cC(\mi,\mi)$ be a coalgebra $1$-morphism.
\begin{enumerate}[$($i$)$]
\item\label{lem:cotensor-product1} For any left $\rF$-comodule $\rN=(\rN,\lambda_{\rN})$ in $\cC(\mj,\mi)$, we have $\rF\Box_{\rF}\rN\cong\rN$ as left $\rF$-comodules.
\item\label{lem:cotensor-product2} For any right $\rF$-comodule $\rM=(\rM,\rho_{\rM})$ in $\cC(\mi,\mk)$, we have $\rM\Box_{\rF}\rF\cong\rM$ as right $\rF$-comodules.
\end{enumerate}
\end{prop}

\begin{rem}\label{rem1}
Note that (co)algebras, (co)modules, bi(co)modules and the respective homomorphisms in the bilax-unital $2$-category $\overline{\cC}$ (or $\underline{\cC}$)
are defined just as in $\cC$.
Moreover, all statements in this subsection also hold with respect to $\overline{\cC}$ (resp. $\underline{\cC}$) and $\widehat{\cC}$ (resp. $\widecheck{\cC}$).
\end{rem}

\subsection{Internal Hom}\label{s3.6}

In this subsection, we will switch to the set-up of the injective abelianization
$\underline{\cC}$ of a fiax category $\cC$ and use the language of coalgebras.
Note that all statements for coalgebra $1$-morphisms associated to oplax units
and unitors have the obvious dual statements for algebra $1$-morphisms associated
the lax units and unitors.

Let $\bM$ be a finitary bilax $2$-representation of a fiax category $\cC$. Denote by $\mathrm{vect}_{\Bbbk}$ the category of finite dimensional $\Bbbk$-vector
spaces. Note that each $\cC(\mj,\mi)$ is equivalent to the full subcategory of injective objects in $\underline{\cC}(\mj,\mi)$, see Subsection~\ref{sec:abelianization}. For any
$X\in\bM(\mi)$ and $Y\in\bM(\mj)$, consider the unique (up to isomorphism)
left exact functor from $\underline{\cC}(\mj,\mi)$ to $\mathrm{vect}_{\Bbbk}$ given by
\[\rF\mapsto \Hom_{\bM(\mi)}(X,\bM(\rF)Y),\qquad \text{ where }\quad\rF\in\cC(\mj,\mi).\]
This left exact functor is representable. This means that there exists a unique (up to isomorphism)
$1$-morphism $\ihom(Y,X)\in\underline{\cC}(\mj,\mi)$, called the {\em internal hom} from $Y$ to $X$, such that there is a natural isomorphism
$$\Hom_{\bM(\mi)}(X,\bM(\rF)Y)\cong \Hom_{\underline{\ccC}(\mj,\mi)}(\ihom(Y,X),\rF),\quad
\text{ for }\quad\rF\in\cC(\mj,\mi).$$ Similarly to \cite[Lemma~4.2]{MMMT}, the above natural isomorphism can be extended to any $\rF\in\underline{\cC}(\mj,\mi)$, namely,
\begin{equation}\label{eq:09}
\Hom_{\underline{\bM}(\mi)}(X,\underline{\bM}(\rF)Y)\cong \Hom_{\underline{\ccC}(\mj,\mi)}(\ihom(Y,X),\rF), \quad\text{ for }\quad\rF\in\underline{\cC}(\mj,\mi)
\end{equation}
In fact, the internal hom assignment
\[\ihom({}_-,{}_-):\bM(\mj)^{\operatorname{op}} \times \bM(\mi) \to \underline{\cC}(\mj,\mi)\]
is, naturally, a bifunctor.

\begin{prop}\label{end is coalg}
For any object $X\in\bM(\mi)$, the $1$-morphism $\rC_X:=\ihom(X,X)$ has the structure
of a coalgebra in $\underline{\cC}(\mi,\mi)$.
\end{prop}

\begin{proof}
We follow \cite[Lemma~4.3]{MMMT} and \cite[Section~7.9]{EGNO}.
We have the coevaluation map:
\[\mathrm{coev}_{X,X}:X\to\underline{\bM}(\rC_X)X,\]
given by the pre-image of $\id_{\rC_X}$ under the isomorphism~\eqref{eq:09}, by setting $\rF=\rC_X$ and $X=Y$.
The comultiplication map
\[\Delta_{\rC_X}: \rC_X\to\rC_X\rC_X\]
is given by the image of $(\underline{\bM}(\id_{\rC_X})\mathrm{coev}_{X,X})\circ_{\rv}\mathrm{coev}_{X,X}$ under the isomorphism
\[\Hom_{\underline{\bM}(\mj)}(X,\underline{\bM}(\rC_X\rC_X)X)\cong \Hom_{\underline{\ccC}(\mi,\mj)}(\rC_X,\rC_X\rC_X)\]
since $\underline{\bM}(\rF)\underline{\bM}(\rG)=\underline{\bM}(\rF\rG)$, for any $\rF\in\underline{\cC}(\mi,\mj)$ and $\rG\in\underline{\cC}(\mk,\mi)$.
The counit map
\[\epsilon_{\rC_X}: \rC_X\to\rI'_{\mi}\]
is given by the image of $(u'_{\mi})_X:X\to\bM(\rI'_{\mi})X=\underline{\bM}(\rI'_{\mi})X$ under the isomorphism \eqref{eq:09}, by setting $\rF=\rI'_{\mi}$ and $X=Y$.
Note that, by the naturality of the isomorphism~\eqref{eq:09}, we have the following two commutative diagrams:
\begin{equation}\label{eq:10}
\xymatrix@C=3.5pc@R=2.5pc{X\ar[r]^{\mathrm{coev}_{X,X}\qquad}\ar[d]_{\mathrm{coev}_{X,X}}&\underline{\bM}(\rC_X)X\ar[r]^{\underline{\bM}(\Delta_{\rC_X})_X\ }
&\underline{\bM}(\rC_X\rC_X)X\ar@{=}[d]\\
\underline{\bM}(\rC_X)X\ar[rr]^{\underline{\bM}(\id_{\rC_X})\mathrm{coev}_{X,X}\ }&&\underline{\bM}(\rC_X)\underline{\bM}(\rC_X)X}\quad\text{and}\quad
\xymatrix@C=2.5pc@R=2.5pc{X\ar[rr]^{\mathrm{coev}_{X,X}\qquad}\ar[drr]_{(u'_{\mi})_X}&&
\underline{\bM}(\rC_X)X\ar[d]^{\underline{\bM}(\epsilon_{\rC_X})_X}\\
&&\underline{\bM}(\rI'_{\mi})X.}
\end{equation}
The coassociativity and the right counitality, as in \cite[Section~7.9]{EGNO}, can be checked by a direct computation
using the natural isomorphism \eqref{eq:09} and \eqref{eq:bilax-2-rep}. Consider the commutative diagram
\begin{equation}\label{eq:11}
\xymatrix@C=3pc@R=2.5pc{&&\underline{\bM}(\rC_X)X\ar[rr]^{\underline{\bM}(\Delta_{\rC_X})_X}&&\underline{\bM}(\rC_X\rC_X)X\ar@{=}[d]
\ar@/^5pc/[dddd]^{\underline{\bM}(\epsilon_{\rC_X}\id_{\rC_X})_X}\\
X\ar[drr]_{(u'_{\mi})_X}\ar[rr]^{\mathrm{coev}_{X,X}}\ar[dddrr]_{\mathrm{coev}_{X,X}}\ar[urr]^{\mathrm{coev}_{X,X}}&&\underline{\bM}(\rC_X)X
\ar[rr]^{\underline{\bM}(\id_{\rC_X})\mathrm{coev}_{X,X}\quad }\ar[d]_{\underline{\bM}(\epsilon_{\rC_X})_X}&&\underline{\bM}(\rC_X)\underline{\bM}(\rC_X)X
\ar[d]_{\underline{\bM}(\epsilon_{\rC_X})\id_{\underline{\bM}(\rC_X)_X}}\\
&&\underline{\bM}(\rI'_{\mi})X\ar[rr]_{\id_{\underline{\bM}(\rI'_{\mi})}\mathrm{coev}_{X,X}\quad}&&\underline{\bM}(\rI'_{\mi})\underline{\bM}(\rC_X)X\ar@{=}[dd]\\
&&&&\\
&&\underline{\bM}(\rC_X)X\ar[uurr]|-{\ (u'_{\mi})_{\underline{\bM}(\rC_X)X}}\ar[rr]^{\underline{\bM}(l'_{\mi,\mi,\rC_X})_X}&&\underline{\bM}(\rI'_{\mi}\rC_X)X.}
\end{equation}
Here the top pentagon and the left middle triangle commute thanks to \eqref{eq:10}. The middle square and the bottom left square commute by the
naturalities of $\underline{\bM}(\epsilon_{\rC_X})$ and $u'_{\mi}$ respectively.
The rightmost square commutes by definition and the bottom middle triangle commutes due to~\eqref{def:bilax-2-rep}.
Therefore the two paths along the boundary of the whole diagram~\eqref{eq:11} coincide with each other which implies that their images under the the isomorphism
\eqref{eq:09} by setting $\rF=\rI'_{\mi}\rC_X$ and $X=Y$ also coincide. Hence the left counitality follows.
\end{proof}

\subsection{Categories of comodules}\label{sec:coalgcomod}

For any object $X\in\bM(\mi)$, we have the bilax $2$-representation $\comod_{\underline{\ccC}}\text{-}\rC_X$  of $\cC$ (see the end of Subsection~\ref{s2.5} for definitions). Then we have the functor
\[\begin{split}
    \Theta_\mj: \bM(\mj)&\to (\comod_{\underline{\ccC}}\text{-}\rC_X)(\mj),\\
     Y&\mapsto \ihom (X,Y),
\end{split}
\]
for each $\mj\in\ob\cC$ with the obvious assignment on morphisms.

We have a result similar to \cite[Theorem 4.7]{MMMT} in the following general setting:

\begin{thm}\label{thm4.7}
Let $\bM$ be a finitary bilax $2$-representation of a fiax category $\cC$ generated by some non-zero object $X\in\bM(\mi)$.
Then the functors $\Theta_\mj$ give rise to an equivalence of bilax $2$-representations of $\cC$ between $\underline{\bM}$
and $\comod_{\underline{\ccC}}\text{-}\rC_X$. This equivalence can be restricted to an equivalence of finitary bilax $2$-representations of $\cC$
between $\bM$ and $\operatorname{inj}_{\underline{\ccC}}\text{-}\rC_X$.
\end{thm}

The rest of the subsection is devoted to the proof of Theorem~\ref{thm4.7}.

\begin{lem}\label{Crepmorph}
For any $1$-morphism $\rF\in \cC(\mj,\mk)$ and any object $Y\in\bM(\mj)$, we have a natural isomorphism
\begin{equation}\label{eq:003}
\ihom(X,\bM(\rF)Y)\cong \rF\ihom(X,Y)
\end{equation}
in $\underline{\cC}$. Furthermore, the functors $\Theta_\mj$ form a bilax $2$-natural transformation $\Theta$ of bilax $2$-rep\-re\-sen\-ta\-ti\-ons.
\end{lem}

\begin{proof} The first statement is proved similarly to \cite[Lemma~4.4]{MMMT}. We have the sequence of natural isomorphisms:
\begin{equation}
    \begin{split}
        \Hom_{\underline{\ccC}}(\ihom(X,\bM(\rF)Y),\,\rG)&\cong \Hom_{\underline{\bM}(\mk)}(\underline{\bM}(\rF)Y,\,\underline{\bM}(\rG)X)\\
        &\cong\Hom_{\underline{\bM}(\mj)}(Y,\,\underline{\bM}(\rF^\star)\underline{\bM}(\rG)X)\\
        &=\Hom_{\underline{\bM}(\mj)}(Y,\,\underline{\bM}(\rF^\star\rG)X)\\
        &\cong\Hom_{\underline{\ccC}}(\ihom(X,Y),\,\rF^\star\rG)\\
        &\cong\Hom_{\underline{\ccC}}(\rF\ihom(X,Y),\,\rG),
    \end{split}
\end{equation}
for any $1$-morphism $\rG\in\underline{\cC}(\mi,\mk)$.
To obtain the second statement, one needs to prove that~\eqref{eq:003} satisfies the coherence conditions~\eqref{eq:0001}, \eqref{eq:0002} and~\eqref{eq:0003}.
The proof of~\eqref{eq:0001} is similar to that of \cite[Lemma~4.4]{MMMT}. We only
prove~\eqref{eq:0002} since~\eqref{eq:0003} is proved using a similar computation.
It suffices to prove that, for any $Y\in\bM(\mj)$ and $\rG\in\underline{\cC}(\mi,\mj)$,
the map
\begin{equation*}
{}_-\circ_{\rv}\,l_{\mi,\mj,\,\ihom(X,Y)}:\Hom_{\underline{\ccC}}(\ihom(X,Y),\rG)\to\Hom_{\underline{\ccC}}(\rI_{\mj}\ihom(X,Y),\rG)
\end{equation*}
is equal the composite
\begin{equation*}
\begin{split}
\xymatrix@C=1.6pc{\Hom_{\underline{\ccC}}(\ihom(X,Y),\,\rG)\xrightarrow{\varphi_1}
\Hom_{\underline{\bM}(\mj)}(Y,\,\underline{\bM}(\rG)X)\ar[rr]^{\qquad\qquad\ {}_-\circ_{\rv}\,(u_{\mj})_Y }&&\Hom_{\underline{\bM}(\mj)}(\underline{\bM}(\rI_{\mj})Y,\,\underline{\bM}(\rG)X)}\\
\xrightarrow{\varphi_2}\Hom_{\underline{\ccC}}(\ihom(X,\underline{\bM}(\rI_{\mj})Y),\,\rG)
\xrightarrow{\varphi_{2}^{\mone}}\Hom_{\underline{\bM}(\mj)}(\underline{\bM}(\rI_{\mj})Y,\,\underline{\bM}(\rG)X)\\
\xrightarrow{\varphi_{3}}\Hom_{\underline{\bM}(\mj)}(Y,\,\underline{\bM}(\rI'_{\mj})\underline{\bM}(\rG)X)=
\Hom_{\underline{\bM}(\mj)}(Y,\,\underline{\bM}(\rI'_{\mj}\rG)X)\\
\xrightarrow{\varphi_{4}}\Hom_{\underline{\ccC}}(\ihom(X,Y),\,\rI'_{\mj}\rG)
\xrightarrow{\varphi_{5}}\Hom_{\underline{\ccC}}(\rI_{\mj}\ihom(X,Y),\,\rG),
\end{split}
\end{equation*}
where we note that $(\rI_{\mj})^\star=\rI'_{\mj}$ and each $\varphi_{i}$ is a natural isomorphism.
Therefore, the equality we need to prove is the following:
\begin{equation}\label{eq:0004}
(\varphi_4^{\mone}\circ_{\rv}\varphi_5^{\mone})(\gamma\circ_{\rv}\,l_{\mi,\mj,\,\ihom(X,Y)})=\varphi_3(\varphi_1(\gamma)\circ_{\rv}\,(u_{\mj})_Y),
\end{equation}
for any $2$-morphism $\gamma:\ihom(X,Y)\to\rG$ in $\underline{\cC}$. Denote by
\[\alpha:\rI'_{\mj}\to \rI'_{\mj}\rI_{\mj}\quad\text{and}\quad
    \beta:\rI_{\mj}\rI'_{\mj} \to \rI_{\mj}\]
the adjunction morphisms for the adjoint pair $(\rI_{\mj},\rI'_{\mj})$. Consider the diagram
\begin{equation}\label{eq:0005}
\xymatrix@C=3.5pc@R=2.5pc{Y\ar[d]_{(u'_{\mj})_Y}\ar[rr]^{\mathrm{coev}_{X,Y}\qquad\quad}&&\bM(\ihom(X,Y))X\ar[d]^{(u'_{\mj})_{\bM(\ihom(X,Y))X}}\\
\bM(\rI'_{\mj})Y\ar[d]_{\bM(\alpha)_Y}\ar[rr]^{\id_{\bM(\rI'_{\mj})}\mathrm{coev}_{X,Y}\qquad\quad}
&&\bM(\rI'_{\mj})\bM(\ihom(X,Y))X\ar[d]^{\bM(\alpha)_{\bM(\ihom(X,Y))X}}\\
\bM(\rI'_{\mj}\rI_{\mj})Y\ar@{=}[d]\ar[rr]^{\id_{\bM(\rI'_{\mj}\rI_{\mj})}\mathrm{coev}_{X,Y}\qquad\quad}&&\bM(\rI'_{\mj}\rI_{\mj})\bM(\ihom(X,Y))X\ar@{=}[d]\\
\bM(\rI'_{\mj})\bM(\rI_{\mj})Y\ar[d]_{\id_{\bM(\rI'_{\mj})}(u_{\mj})_Y}\ar[rr]^{\id_{\bM(\rI'_{\mj})\bM(\rI_{\mj})}\mathrm{coev}_{X,Y}\qquad\quad}
&&\bM(\rI'_{\mj})\bM(\rI_{\mj})\bM(\ihom(X,Y))X\ar[d]^{\id_{\bM(\rI'_{\mj})}(u_{\mj})_{\bM(\ihom(X,Y))X}}\\
\bM(\rI'_{\mj})Y\ar[rr]^{\id_{\bM(\rI'_{\mj})}\mathrm{coev}_{X,Y}\qquad\quad}&&\bM(\rI'_{\mj})\bM(\ihom(X,Y))X,}
\end{equation}
where $\mathrm{coev}_{X,Y}:Y\to\bM(\ihom(X,Y))X$ is the pre-image of $\id_{\ihom(X,Y)}$ under the isomorphism \[\Hom_{\underline{\bM}(\mj)}(Y,\,\underline{\bM}(\ihom(X,Y))X)\cong \Hom_{\underline{\ccC}(\mi,\mj)}(\ihom(X,Y),\,\ihom(X,Y)).\]
In~\eqref{eq:0005}, the top and the second squares commute due to the naturalities of $u'_{\mj}$ and $\bM(\alpha)$ respectively;
the third square commutes by definition and the bottom square commutes due to the naturality of $u_{\mj}$.
Therefore the whole diagram~\eqref{eq:0005} commutes.
Computing the left hand-side of \eqref{eq:0004}, we have
\[\begin{split}
(\varphi_4^{\mone}\circ_{\rv}\varphi_5^{\mone})(\gamma\circ_{\rv}\,l_{\mi,\mj,\,\ihom(X,Y)})&=
\varphi_4^{\mone}\Big((\id_{\rI'_{\mj}}\big(\gamma\circ_{\rv}\,l_{\mi,\mj,\,\ihom(X,Y)})\big)\circ_{\rv}
\big(\alpha\id_{\ihom(X,Y)}\big)\circ_{\rv}l'_{\mi,\mj,\,\ihom(X,Y)}\Big)\\
&=
\varphi_4^{\mone}\big((\id_{\rI'_{\mj}}\gamma)\circ_{\rv}(\id_{\rI'_{\mj}}l_{\mi,\mj,\,\ihom(X,Y)})\circ_{\rv}(\alpha\id_{\ihom(X,Y)})\circ_{\rv}
l'_{\mi,\mj,\,\ihom(X,Y)}\big)\\
&=\bM(\id_{\rI'_{\mj}}\gamma)_X\circ_{\rv}\bM(\id_{\rI'_{\mj}}l_{\mi,\mj,\,\ihom(X,Y)})_X\circ_{\rv}\bM(\alpha\id_{\ihom(X,Y)})_X\\
&\quad\circ_{\rv}\bM(l'_{\mi,\mj,\,\ihom(X,Y)})_X\circ_{\rv}\mathrm{coev}_{X,Y}\\
&=\bM(\id_{\rI'_{\mj}}\gamma)_X\circ_{\rv}(\id_{\bM(\rI'_{\mj})}(u_{\mj})_{\bM(\ihom(X,Y))X})\circ_{\rv}\bM(\alpha\id_{\ihom(X,Y)})_X\\
&\quad\circ_{\rv}(u'_{\mj})_{\bM(\ihom(X,Y))X}\circ_{\rv}\mathrm{coev}_{X,Y}\\
&=\bM(\id_{\rI'_{\mj}}\gamma)_X\circ_{\rv}(\id_{\bM(\rI'_{\mj})}\mathrm{coev}_{X,Y})\circ_{\rv}(\id_{\bM(\rI'_{\mj})}(u_{\mj})_Y)\circ_{\rv}
\bM(\alpha)_Y\circ_{\rv}(u'_{\mj})_Y,
\end{split}\]
where the first equality holds due to the proof of Proposition~\ref{propadj}; the second equality holds by the interchange law;
the third equality holds by the naturality of the isomorphism~\eqref{eq:09}, that is, we have the commutative diagram:
\begin{equation}\label{eq:0006}
\xymatrix@C=3pc@R=2.5pc{\Hom_{\underline{\ccC}(\mi,\mj)}(\ihom(X,Y),\,\ihom(X,Y))\ar[rr]^{\cong}\ar[d]_{\delta\circ{\rv}{}_-}
&&\Hom_{\underline{\bM}(\mj)}(Y,\,\underline{\bM}(\ihom(X,Y))X)\ar[d]^{\bM(\delta)_X\circ{\rv}{}_-}\\
\Hom_{\underline{\ccC}(\mi,\mj)}(\ihom(X,Y),\,\rH)\ar[rr]^{\cong }&&\Hom_{\underline{\bM}(\mj)}(Y,\,\underline{\bM}(\rH)X);}
\end{equation}
for any $\delta\in\Hom_{\underline{\ccC}(\mi,\mj)}(\ihom(X,Y),\,\rH)$;
the fourth equality holds by~\eqref{eq:bilax-2-rep} and the fifth equality holds due to the commutativity of ~\eqref{eq:0005}.
Computing the right hand-side of \eqref{eq:0004}, we have
\[\begin{split}
\varphi_3(\varphi_1(\gamma)\circ_{\rv}\,(u_{\mj})_Y)&=\varphi_3(\bM(\gamma)_X\circ_{\rv}\mathrm{coev}_{X,Y}\circ_{\rv}\,(u_{\mj})_Y)\\
&=\bM(\id_{\rI'_{\mj}}\gamma)_X\circ_{\rv}(\id_{\bM(\rI'_{\mj})}\mathrm{coev}_{X,Y})\circ_{\rv}(\id_{\bM(\rI'_{\mj})}(u_{\mj})_Y)\circ_{\rv}
\bM(\alpha)_Y\circ_{\rv}(u'_{\mj})_Y,
\end{split}\]
where the first equality holds by~\eqref{eq:0006} and the last equality holds due to the proof of Proposition~\ref{propadj}.
This proves the equality~\eqref{eq:0004}.
\end{proof}

\begin{lem}\label{lem:removingA}
Let $\rC$ be a coalgebra 1-morphism in $\underline{\cC}(\mi,\mi)$ for some $\mi\in\ob\cC$. For any $\rM\in (\comod_{\underline{\ccC}}\text{-}\rC)(\mj)$ and any $1$-morphism $\rF\in\cC(\mi,\mj)$,
we have an isomorphism
\begin{equation}\label{eq:12}
\Hom_{\comod_{\underline{\ccC}}\text{-}\rC}(\rM,\rF\rC)\cong\Hom_{\underline{\ccC}}(\rM,\rF).
\end{equation}
\end{lem}

\begin{proof}
By~Proposition~\ref{coalgebra lifts to 2cat}, the coalgebra $1$-morphism $\rC$ is also a coalgebra $1$-morphism in $\widecheck{\cC}(\mi,\mi)\subseteq\underline{\cC}(\mi,\mi)$ with the comultiplication map $\Delta_{\rC}$ and the new counit map $\zeta'_{\rC}$.
Define a $\Bbbk$-linear map $\Phi$ from the left-hand side to the right-hand side of~\eqref{eq:12} by sending any right $\rC$-comodule homomorphism
$\alpha:\rM\to\rF\rC$ to the composite
\[\xymatrix{\rM\ar[rr]^{\alpha}&&\rF\rC\ar[rr]^{\id_{\rF}\zeta'_{\rC}}&&\rF\mathbbm{1}'_{\mi}\ar[rr]^{(\widecheck{r}'_{\mi,\mj,\rF})^{\mone}}_{\cong}&&\rF}.\]
Define a $\Bbbk$-linear map $\Psi$ from the right-hand side to the left-hand side of~\eqref{eq:12} by sending
any $2$-morphism $\beta:\rM\to\rF$ to the composite
\[\xymatrix{\rM\ar[rr]^{\rho_{\rM}}&&\rM\rC\ar[rr]^{\beta\id_{\rC}}&&\rF\rC,}\]
which is, obviously, a right $\rC$-comodule homomorphism.

We now claim that $\Psi\Phi=\id$ and $\Phi\Psi=\id$. On the one hand, consider the diagram
\begin{equation}\label{eq:13}
\xymatrix@R=2.5pc{\rM\ar[rr]^{\rho_{\rM}}\ar[drr]^{\alpha}&&\rM\rC\ar[rr]^{\alpha\id_{\rC}}&&\rF\rC\rC\ar[rr]^{\id_{\rF}\zeta'_{\rC}\id_{\rC}}
&&\rF\mathbbm{1}'_{\mi}\rC\ar[rr]^{(\widecheck{r}'_{\mi,\mj,\rF})^{\mone}\id_{\rC}}\ar[drr]|-{\id_{\rF}(\widecheck{l}'_{\mi,\mi,\rC})^{\mone}}&&\rF\rC\ar@{=}[d]\\
&&\rF\rC\ar[urr]_{\id_{\rF}\Delta_{\rC}}\ar@{=}[rrrrrr]&&&&&&\rF\rC,}
\end{equation}
where the left square commutes if $\alpha$ is a right $\rC$-comodule homomorphism;
the middle square commutes thanks to the left counitality of the coalgebra $1$-morphism $(\rC,\,\Delta_{\rC},\,\zeta'_{\rC})$ in $\widecheck{\cC}(\mi,\mi)$; and
the right triangle commutes by the condition~\eqref{diag:lrcompatible}.
The commutativity of~\eqref{eq:13} implies that $\Psi\Phi(\alpha)=\alpha$,
for any right $\rC$-comodule homomorphism
$\alpha:\rM\to\rF\rC$. On the other hand, consider the diagram
\begin{equation}\label{eq:14}
\xymatrix@R=2.5pc{\rM\ar[rr]^{\rho_{\rM}}\ar[drr]_{\widecheck{r}'_{\mi,\mj,\rM}}&&\rM\rC\ar[rr]^{\beta\id_{\rC}}\ar[d]^{\id_{\rM}\zeta'_{\rC}}&&
\rF\rC\ar[rr]^{\id_{\rF}\zeta'_{\rC}}
&&\rF\mathbbm{1}'_{\mi}\ar[rr]^{(\widecheck{r}'_{\mi,\mj,\rF})^{\mone}}&&\rF\\
&&\rM\mathbbm{1}'_{\mi}\ar[urrrr]|-{\beta\id_{\mathbbm{1}'_{\mi}}}\ar[rrrrrr]^{(\widecheck{r}'_{\mi,\mj,\rM})^{\mone}}&&&&&&\rM\ar[u]_{\beta},}
\end{equation}
where the left triangle commutes thanks to Corollary~\ref{comodulelift}~\ref{comodulelift-2}; the top middle square commutes by the interchange law; and
the right square commutes due to the naturality of $\widecheck{r}'_{\mi,\mj, {}_-}$. The commutativity of~\eqref{eq:14} implies that $\Phi\Psi(\beta)=\beta$, for any
$2$-morphism $\beta:\rM\to\rF$. The claim follows.
\end{proof}

\begin{lem}\label{imageinjective} For any $Y\in\bM(\mj)$, the object $\Theta_\mj(Y)=\ihom(X,Y)$ is injective in $(\comod_{\underline{\ccC}}\text{-}\rC_X)(\mj)$.
That is, the functor $\Theta_\mj$ factors through $(\operatorname{inj}_{\underline{\ccC}}\text{-}\rC_X)(\mj)$, for each $\mj\in\ob\cC$.
\end{lem}

\begin{proof}
By Lemma~\ref{prop:fiaxsplit} and the fact that the object $X$ generates $\bM$,
it is enough to prove the statement for $Y=\bM(\rF)X$, where $\rF\in\cC(\mi,\mj)$.
By Lemma~\ref{Crepmorph}, we have the isomorphism
\[\ihom(X,\bM(\rF)X)\cong \rF\ihom(X,X)=\rF\rC_X.\]
By Lemma~\ref{lem:removingA}, we obtain
\[\Hom_{\comod_{\underline{\ccC}}\text{-}\rC_X}({}_-,\rF\rC_X)\cong\Hom_{\underline{\ccC}}({}_-,\rF).\]
Therefore the injectivity of $\rF\rC_X$ follows from the injectivity of $\rF$ in $\underline{\cC}$. The proof is complete.
\end{proof}

\begin{proof}[Proof of Theorem~\ref{thm4.7}]
Using Lemma~\ref{Crepmorph}, Lemma~\ref{lem:removingA} (by setting $\rC:=\rC_X$), Lemma~\ref{imageinjective},
Theorem~\ref{thm4.7} is proved mutatis mutandis \cite[Theorem~4.7]{MMMT}.
\end{proof}

\subsection{Morita-Takeuchi theory}\label{s3.8}

Let $\cC=(\cC,\,\rI=\{\rI_{\mi}|\,\mi\in\ob\cC\},\,\rI'=\{\rI'_{\mi}|\,\mi\in\ob\cC\})$
be a fiax category, and $\rF=(\rF,\,\Delta_{\rF},\epsilon_{\rF})\in\underline{\cC}(\mi,\mi)$
and $\rG=(\rG,\,\Delta_{\rG},\epsilon_{\rG})\in\underline{\cC}(\mj,\mj)$ two coalgebra
$1$-morphisms.

\begin{prop}\label{moritathm}
As bilax $2$-representations of $\underline{\cC}$, $\comod_{\underline{\ccC}}\text{-}\rF$ and $\comod_{\underline{\ccC}}\text{-}\rG$ are equivalent if and only if there exist bi-injective bicomodule 1-morphisms $\rM={}_{\rF}\rM_{\rG}$ and $\rN={}_{\rG}\rN_{\rF}$ in $\underline{\cC}$ and bicomodule $2$-isomorphisms
\[f:\xymatrix{\rF\ar[r]^{\cong\ \ \ }&\rM\Box_{\rG} \rN}\quad\text{and}\quad g:\xymatrix{\rG\ar[r]^{\cong\ \ \ }&\rN\Box_{\rF}\rM,}\]
such that the following diagrams commute:
\begin{equation*}
\xymatrix{\rM\ar[rr]^{\cong}\ar[d]_{\cong}&&\rM\Box_{\rG}\rG\ar[d]^{\id_{\rM}\Box\, g}\\
\rF\Box_{\rF}\rM\ar[rr]^{f\Box\id_{\rM}\quad}&&\rM\Box_{\rG}\rN\Box_{\rF}\rM}\quad\text{and}\quad
   \xymatrix{\rN\ar[rr]^{\cong}\ar[d]_{\cong}&&\rN\Box_{\rF}\rF\ar[d]^{\id_{\rN}\Box\, f}\\
\rG\Box_{\rG}\rN\ar[rr]^{g\Box\id_{\rM}\quad}&&\rN\Box_{\rF}\rM\Box_{\rG}\rN,}
\end{equation*}
where the isomorphism are from
Proposition~\ref{prop:cotensor-product}.
\end{prop}

In this case, $\rF$ and $\rG$ are said to be {\em Morita-Takeuchi equivalent}.

\begin{proof}
Mutatis mutandis \cite[Theorem~5.1]{MMMT}.
\end{proof}

\subsection{Correspondence between coalgebras and 2-representations}\label{s3.new}

The two previous subsections say that bilax 2-representations of $\cC$
with certain conditions (e.g. finitary and having a generator) correspond to coalgebra 1-morphisms in $\underline{\cC}$, and the equivalence of bilax 2-representations gives Morita-Tacheuchi equivalence of the corresponding coalgebras.
If $\cC$ is a fiat 2-category, then all coalgebra 1-morphisms in $\underline{\cC}$ arise in this way since $\operatorname{inj}_{\underline{\ccC}}\dash\rC$ is finitary for any coalgebra $\rC$ in $\cC$ (we do not know a reference for this exact statement, so we prove it in
Corollary~\ref{fiatfinitary}, another proof is expected to appear in \cite{MMMTZ3}),
so that we have a correspondence between finitary 2-represntations of $\cC$ and all coalgebras in $\cC$.
For a fiax category $\cC$, however, it seems that we need to restrict to certain coalgebras in $\underline{\cC}$ to have such a correspondence.
To characterize such coalgebras, let $\rC$ be a coalgebra in $\underline{\cC}(\mi,\mi)$ and consider the bilax 2-representation \[\cC\rC:=\add_{\rC}\{\rF\rC\ |\ \rF\in\cC(\mi,\mj),\mj\in\ob\cC\}\] of $\cC$. That is, each $\cC\rC(\mj)$ is the additive idempotent split closure of  $\{\rF\rC\ |\,\rF\in\cC(\mi,\mj)\}$ in $(\comod_{\underline{\ccC}}\dash \rC)(\mj)$, for each $\mj\in\ob\cC$.

\begin{lem}\label{lemIAI}
Let $\rC=(\rC,\Delta_\rC,\epsilon_\rC)$ be a coalgebra in $\underline{\cC}(\mi,\mi)$. Then for any 1-morphism $\rF$ in $\cC(\mi,\mj)$ for some $\mj\in\ob\cC$, the 1-morphism $\rF\rC \rF^\star$ has the canonical structure of a coalgebra in $\underline{\cC}(\mj,\mj)$.
\end{lem}

\begin{proof}
Assume that the $2$-morphisms $\alpha_{\rF}:\rI'_{\mi}\to \rF^{\star}\rF$ and $\beta_{\rF}:\rF\rF^{\star} \to \rI_{\mj}$
are the adjunction morphisms for the adjoint pair $(\rF,\rF^{\star})$. Then the comultiplication is given by the composition
\[\xymatrix@C=2.5pc{\rF\rC\rF^\star\ar[rr]^{\id_\rF\Delta_\rC\id_{\rF^\star}\ }&&\rF\rC\rC\rF^\star\ar[rr]^{\id_\rF r'_{\mi,\mi,\rC}\id_{\rC\rF^\star}\ }&&\rF\rC\rI'_{\mi}\rC\rF^\star\ar[rr]^{\id_{\rF\rC}\alpha_\rF \id_{\rC\rF^\star}\ }&&\rF\rC\rF^\star\rF\rC\rF^\star}\]
which is by \eqref{diag:oplaxlrcompatible} the same as
\[\xymatrix@C=2.5pc{\rF\rC\rF^\star\ar[rr]^{\id_\rF\Delta_\rC\id_{\rF^\star}\ }&&\rF\rC\rC\rF^\star\ar[rr]^{\id_{\rF\rC} l'_{\mi,\mi,\rC}\id_{\rF^\star}\ }&&\rF\rC\rI'_{\mi}\rC\rF^\star\ar[rr]^{\id_{\rF\rC}\alpha_\rF \id_{\rC\rF^\star}\ }&&\rF\rC\rF^\star\rF\rC\rF^\star}\]
To define the counit map, recall that $\rC$ is also a coalgebra with respect to $\mathbbm{1}'_{\mi}$ by the proof of Proposition \ref{coalgebra lifts to 2cat} (see the proof Proposition \ref{algebra lifts to 2cat}) where the new counit is denoted by $\zeta'_\rC$.
Now, the counit for $\rF\rC\rF^\star$ is given by
\[\xymatrix@C=2.5pc{\rF\rC\rF^\star\ar[rr]^{\id_\rF\zeta'_\rC\id_{\rF^\star}\ }&&\rF\mathbbm{1}'_{\mi}\rF^\star\ar[rr]^{\ (\check{r}'_{\mi,\mj,\rF})^{\mone} \id_{\rF^\star}\ }&&\rF\rF^\star\ar[r]^{\beta_\rF}&\rI_{\mj}\ar[r]^{(\xi'_{\mj})^\star}&\mathbbm{1}'_{\mj}\ar[r]^{\xi_{\mj}}&\rI'_{\mj},}\]
noting that we have $(\check{r}'_{\mi,\mj,\rF})^{\mone} \id_{\rF^\star}=\id_{\rF}(\check{l}'_{\mj,\mi,\rF^\star})^{\mone}$ by \eqref{diag:oplaxlrcompatible} for the second map.

Both associativity and unitality follow, for example, similarly to \cite[Lemma~9]{MMMTZ}.
\end{proof}

In particular, from Lemma \ref{lemIAI} it follows that the $1$-morphism $\rI_{\mi}\rC\rI'_{\mi}$ has the canonical
structure of a coalgebra.

\begin{lem}\label{prop11}
Let $0\neq X\in\bM(\mi)$ and $\rF\in\cC(\mi,\mj)$ for some $\mj\in\ob\cC$ be such that $\bM(\rF)X\neq 0$.
Then we have \[\ihom(\bM (\rF)X,\bM (\rF)X)\cong \rF\rC_{X} \rF^\star\]
and the coalgebra structure in Proposition \ref{end is coalg} agrees with the coalgebra structure in Lemma \ref{lemIAI}.
\end{lem}

\begin{proof}
Mutatis mutandis the proof of \cite[Proposition 11]{MMMTZ}.
\end{proof}

\begin{prop}\label{CA}
The bilax 2-representation $\cC\rC$ is finitary and is equivalent to $\operatorname{inj}_{\underline{\ccC}}\dash \rI_{\mi}\rC\rI'_{\mi}$.
\end{prop}

\begin{proof}
The first claim follows from $\cC$ being finitary. Indeed,
to see that $\cC\rC$ has finitely many indecomposable 1-morphisms, it is enough to note that $\rF\rC$ decomposes into finitely many summands in $\underline{\cC}$ and has less summands in $\comod_{\underline{\ccC}}\dash \rC$.
Similarly, the morphism spaces in $(\comod_{\underline{\ccC}}\dash \rC)(\mj)$, for any $\mj\in\ob\cC$, are smaller than those in $\cC(\mi,\mj)$ and thus finite dimensional.

For the second claim, we apply Theorem~\ref{thm4.7} to $\cC\rC$ by setting $X=\rI_{\mi}\rC\in\cC\rC$. The assumptions of Theorem~\ref{thm4.7} are satisfied: $\cC\rC=\cC X$
(because the lax unit $\rI_{\mi}$ is split) and $\cC \rC$ is finitary by the above.
By Lemma~\ref{prop11}, the internal End agrees with $\rI_{\mi}\rC\rI'_{\mi}$ as coalgebras.
Thus Theorem~\ref{thm4.7} implies $\cC \rC\simeq \operatorname{inj}_{\underline{\ccC}}\dash \rI_{\mi}\rC\rI'_{\mi}$ as bilax $2$-representations of $\cC$.
\end{proof}

\begin{prop}\label{injA finite}
Let $\rC$ be a coalgebra $1$-morphism in $\underline{\cC}(\mi,\mi)$. Then the following are equivalent.
\begin{enumerate}[$($i$)$]
\item\label{cond-inja-1} $\operatorname{inj}_{\underline{\ccC}}\dash \rC\simeq \cC\rC$ (as bilax $2$-representations of $\cC$);
\item\label{cond-inja-2} $\rC$ is Morita-Takeuchi equivalent to $\rI_{\mi}\rC\rI'_{\mi}$;
\item\label{cond-inja-3} $\operatorname{inj}_{\underline{\ccC}}\dash \rC$ is finitary and $\cC \operatorname{inj}_{\underline{\ccC}}\dash \rC=\operatorname{inj}_{\underline{\ccC}}\dash \rC$.
\end{enumerate}
We call such a coalgebra {\em finitary}.
\end{prop}
\begin{proof}
\ref{cond-inja-1}$\Leftrightarrow$\ref{cond-inja-2}$\Rightarrow $\ref{cond-inja-3} follows from Proposition~\ref{CA}.

For \ref{cond-inja-3}$\Rightarrow $\ref{cond-inja-2}, we apply Theorem $\ref{thm4.7}$ to $\operatorname{inj}_{\underline{\ccC}}\dash \rC$ by setting $X=\rI_{\mi}\rC$ (note that $\rI_{\mi}\rC\in \operatorname{inj}_{\underline{\ccC}}\dash \rC$ by Lemma \ref{lem:removingA}).
Then by in Lemma~\ref{prop11}, we have $\ihom(X,X)\cong\rI_{\mi}\rC\rI'_{\mi}$ and so \ref{cond-inja-2} follows.
\end{proof}

\begin{prop}\label{if A inj}
Let $\rC$ be a coalgebra $1$-morphism in $\underline{\cC}(\mi,\mi)$.
If $\rC\in(\operatorname{inj}_{\underline{\ccC}}\dash \rC)(\mi)$, then $\rC$ is finitary.
\end{prop}

\begin{proof}
Note that $\rC\in(\operatorname{inj}_{\underline{\ccC}}\dash \rC)(\mi)$ implies that $l'_{\mi,\mi,\rC}:\rC\to \rI'_{\mi}\rC$ is split as a $\rC$-comodule map, thus we have $\rC\in\cC \rC$.
Since $\cC\rC$ is finitary, as proved in Proposition \ref{CA}, we can apply Theorem \ref{thm4.7} to $X=\rC$. Then similarly to Lemma~\ref{prop11}, we can show that $\ihom(\rC,\rC)\cong \rC$ and we obtain the desired equivalence $\operatorname{inj}_{\underline{\ccC}}\dash \rC\simeq \cC\rC$.
\end{proof}

\begin{cor}\label{fiatfinitary}
If $\cC$ is a fiat 2-category, then any coalgebra $\rC\in\underline{\cC}(\mi,\mi)$ for any $\mi\in\ob\cC$ is finitary.
\end{cor}

\begin{proof}
Since $\rI_\mi=\rI'_\mi$ is the (strict) identity 1-morphism $\mathbbm{1}_{\mi}$ in this case, Proposition~\ref{CA} implies $\cC\rC\simeq\operatorname{inj}_{\underline{\ccC}}\dash \rC$. The claim now follows from Proposition~\ref{if A inj} since $\rC=\mathbbm{1}_\mi\rC\in\cC\rC$.
\end{proof}

The above motivates the following questions, for which we have no answers at the moment.

\begin{ques}
Is there a coalgebra 1-morphism that is not finitary? Is there a finitary coalgebra 1-morphism that is not injective? Is there a relation between the condition on $\operatorname{inj}_{\underline{\ccC}}\dash \rC$ to be finitary and the condition $\cC (\operatorname{inj}_{\underline{\ccC}}\dash \rC)=\operatorname{inj}_{\underline{\ccC}}\dash \rC$?
\end{ques}

Recall from  Subsection~\ref{sec:lifted-2-rep} and Remark~\ref{rem:oplaxside2} that a bilax $2$-representation $\bM$ of $\cC$ with the condition $\cC\,\bM=\bM$
can be lifted to a $2$-representation of $\widecheck{\cC}$.
On the other hand, by Theorem~\ref{thm4.7}, such a bilax $2$-representation, if finitary, is equivalent to
$\operatorname{inj}_{\underline{\ccC}}\dash \rC_X$, for some coalgebra $1$-morphism $\rC_X$ in $\underline{\cC}$.
The same is true for $\widecheck{\cC}$. Moreover, in $\widecheck{\cC}$, any coalgebra is finitary by Corollary \ref{fiatfinitary}.
To put all these in one picture, we generalize the notion of a coalgebra 1-morphism to a collection of coalgebra 1-morphisms $(\rC_\mi)_{\mi\in\ob\ccC}$, that is,
we say that $(\rC_\mi)_{\mi\in\ob\ccC}$ is finitary if each $\rC_\mi$ is finitary and so on.
Then we have the following picture.

\[
\begin{tikzpicture}
  \matrix (m) [matrix of math nodes,row sep=5em,column sep=-1em,minimum width=2em]
  {\left\{\begin{array}{cc} \mbox{finitary bilax $2$-representations}\\ \mbox{ of $\cC$ with $\cC\bM=\bM$}
\end{array}\right\}/\simeq &&   \left\{\begin{array}{cc} \mbox{finitary $2$-representations}\\ \mbox{ of $\widecheck{\cC}$}
\end{array}\right\}/\simeq\\
      \{\text{finitary coalgebras in $\underline{\cC}$}\}/ \sim &\ \ \ \ \ \ \ \ \ \ \ \ \ \ \ \ \ \ \ \ \ \ \ \ \ \ \ \ \ \ \ & \{\text{coalgebras in $\underline{(\widecheck{\cC})}$}\}/ \sim\\};
  \path[-stealth]
    (m-1-1) edge node [above] {$\bM\mapsto \widecheck{\bM}$} (m-1-3)
    (m-1-3) edge node [below] {$\cC(\bM|_{\ccC})\mapsfrom \bM$} (m-1-1)
    (m-1-1) edge node [left] {Theorem~\ref{thm4.7}} (m-2-1)
    (m-1-3) edge node [right] {} (m-2-3)
    (m-2-1) edge node [above right] {Proposition~\ref{moritathm}} (m-1-1)
    (m-2-1) edge node [below right] {Proposition~\ref{injA finite}} (m-1-1)
    (m-2-3) edge node [below right] {\cite[Section~7]{EGNO}} (m-1-3)
    (m-2-3) edge node [above right] {\cite{MMMT}} (m-1-3)
    (m-2-3) edge node [left]{Corollary \ref{fiatfinitary}} (m-1-3)
    ;
\end{tikzpicture}
\]

Here ``$\simeq$'' denotes the equivalence of bilax $2$-representations and ``$\sim$'' the Morita-Takeuchi equivalence of coalgebras.
The bijection in the top row is a consequence of Theorem~\ref{thm:bijection}.

\subsection{Cells and Duflo 1-morphisms} \label{s3.9}

We complete this section with some results on the
combinatorial structures of bilax-unital $2$-categories,
following \cite[Section~4]{MM1} and \cite[Section~3]{MM2}.

\begin{defn} For any two indecomposable $1$-morphisms $\rF,\rG$ in a $2$-semicategory $\cC$, we say
$\rF\leq_{L} \rG$ if there exists a $1$-morphism $\rH$ in $\cC$ such that $\rG$ is isomorphic to a direct summand of $\rH\rF$. The preorder $\leq_L$ is called the {\em left preorder}. The equivalence classes with respect to $\leq_{L}$ are called the {\em left cell}.
We denoted by $\sim_L$ the corresponding equivalence relation, in the sense that, for any two $1$-morphism $\rF,\rG$, we have $\rF\sim_L\rG$ if and only if,
both $\rF\leq_{L}\rG$ and $\rG\leq_L\rF$.

Similarly, one can define the {\em right} and the {\em two-sided} preorders,
denoted by $\leq_R$ and $\leq_J$ respectively, and the corresponding {\em right}
and {\em two-sided} cells, denoted by $\mathcal R$ and $\mathcal J$ respectively.

A two-sided cell $\mathcal J$ is called {\em strongly regular}
if the intersection of each left cell in $\mathcal J$ and each right cell in $\mathcal J$ consists of exactly one element.
A $2$-semicategory $\cC$ is called {\em strongly regular} if all its two-sided cells are strongly regular.
\end{defn}

Now assume that $\cC=(\cC,\,\rI=\{\rI_{\mi}|\,\mi\in\ob\cC\},\rI'=\{\rI_{\mi}'|\,\mi\in\ob\cC\})$ is a fiax category.
For each object $\mi$, we have one of  the following two cases:
\begin{itemize}
    \item $\cC$ contains the identity $1$-morphism $\mathbbm{1}_{\mi}$ which implies that $l_{\mi,\mi,\mathbbm{1}_{\mi}}=r_{\mi,\mi,\mathbbm{1}_{\mi}}=\id_{\mathbbm{1}_{\mi}}$. In this case, it is clear that $\cC(\mi,\mi)=\widehat{\cC}(\mi,\mi)$.
    \item $\cC$ does not contain any $1$-morphism isomorphic to the identity $1$-morphism $\mathbbm{1}_{\mi}$. In this case, each category $\widehat{\cC}(\mi,\mi)$ has, up to
    isomorphism, exactly one indecomposable $1$-morphism more than $\cC(\mi,\mi)$, namely $\mathbbm{1}_{\mi}$. The $1$-morphism $\mathbbm{1}_{\mi}$ forms a left, right, and two-sided cell on its own and this cell is minimal with respect to the corresponding preorders.
\end{itemize}

Therefore, in either case, the cell structures in $\cC$ are preserved under the embedding $\cC\hookrightarrow \widehat{\cC}$.
Let $\mathcal L$ be a left cell in $\cC$. Viewing $\mathcal L$ as a left cell in $\widehat{\cC}$ and following \cite[Proposition 17]{MM1}
one can define the {\em Duflo $1$-morphism} $\rD_{\mathcal L}$ in $\mathcal L$
to be the unique $1$-morphism, up to isomorphism, such that the minimal submodule of $P_{\mathbbm{1}_{\mi}}$ in the category $\overline{\bP_{\mi}}(\mi)$, the quotient by which is annihilated by all $1$-morphisms in $\mathcal L$, has simple top $L_{\rD_{\mathcal L}}$.
By construction, we have a $2$-morphism
\begin{equation}\label{dufloto1}
   d_{\mathcal{L}}: \rD_{\mathcal{L}}\to \mathbbm{1}_{\mi}
\end{equation}
such that any non-zero $2$-morphism $\rF\to \mathbbm{1}_{\mi}$ factors through $d_{\mathcal{L}}$, if $\rF\in \add(\mathcal L)$.

An interesting question is that whether one can define Duflo $1$-morphism by only using the bilax-unital structure of $\cC$ itself without
going through $\widehat{\cC}$. So far we didn't find such a definition.

\section{Examples}\label{s4}

\subsection{The category $\ca$}\label{sec:example1}

Let $A$ be a finite dimensional, self-injective, basic $\Bbbk$-algebra and $\{e_1,\ldots e_n\}$ a complete set of primitive orthogonal idempotents of $A$.
For any $1\leq i,j\leq n$, denote by ${}_i\mathtt{A}_j$ a basis of the $\Bbbk$-space $e_iAe_j$
chooses with respect to the socle filtration and which contains the element $e_i$, when $i=j$.
Set $\displaystyle\mathtt{A}:=\bigcup_{1\leq i,j\leq n}{}_i\mathtt{A}_j$. Let $\mathrm{tr}$ be the unique linear map from $A$ to $\Bbbk$ such that, for all $a\in \mathtt{A}$,
we have
\[\mathrm{tr}(a)=\begin{cases}
1, & \text{if }a\in \mathrm{soc}(A)\bigcap \mathtt{A};\\
0, & \text{otherwise}.\end{cases}\]
For $a\in\mathtt{A}$, we denote by $a^*$ the unique element in $A$ which satisfies
\[\mathrm{tr}(ba^*) =
\begin{cases}
1, & \text{if }b=a;\\
0, & \text{if }b\in \mathtt{A}\setminus\{a\}.
\end{cases}\]
Define $\ca$ to be a bilax-unital $2$-category such that
\begin{itemize}
    \item $\ob \ca=\{1,\ldots, n\}$ (each object $i$ is a small category equivalent to $\add(Ae_i)$);
    \item $1$-morphisms in each category $\ca(i,j)$ are functors isomorphic to the ones given by tensoring with the projective $A$-$A$-bimodules in
    $\add (Ae_{i}\otimes_{\Bbbk} e_{j} A) $, that is, the additive closure of $Ae_i\otimes_{\Bbbk} e_j A$ in the category of $A$-$A$-bimodules;
    \item $2$-morphisms are given by natural transformations between those functors which are in one to one correspondence with $A$-$A$-bimodule homomorphisms;
    \item the lax unit $\rI_{i}$ is isomorphic to the functor of tensoring with the $A$-$A$-bimodule $Ae_i\otimes_{\Bbbk} e_iA$ and
     the associated left and right lax unitor maps are given by the natural transformations which correspond to the following $A$-$A$-bimodule homorphisms, respectively:
        \[\begin{split}\begin{array}{rcl}
        Ae_i\otimes_{\Bbbk} e_iA\otimes_{A}Ae_i\otimes_{\Bbbk} e_jA&\to &Ae_i\otimes_{\Bbbk} e_jA\\
          \displaystyle\sum_{s}a_se_i\otimes e_ib_s\otimes   \displaystyle\sum_{t}c_te_i\otimes e_jd_t&\mapsto&
          \displaystyle\sum_{s,t}a_se_ib_sc_te_i\otimes e_jd_t,
        \end{array}\\
        \begin{array}{rcl}
        Ae_i\otimes_{\Bbbk} e_jA\otimes_{A}Ae_j\otimes_{\Bbbk} e_jA&\to &Ae_i\otimes_{\Bbbk} e_jA\\
       \displaystyle\sum_{s}a_se_i\otimes e_jb_s\otimes   \displaystyle\sum_{t}c_te_j\otimes e_jd_t&\mapsto&
       \displaystyle\sum_{s,t}a_se_i\otimes e_jb_sc_te_jd_t;
        \end{array}\end{split}
        \]
        \item the oplax unit $\rI'_{i}$ is isomorphic to the functor of tensoring with the $A$-$A$-bimodule $Ae_i\otimes_{\Bbbk} e_iA$ and
     the associated left and right oplax unitor maps are given by the natural transformations which correspond to the following $A$-$A$-bimodule homorphisms, respectively:
        \[\begin{split}\begin{array}{rcl}
       Ae_i\otimes_{\Bbbk} e_jA&\to & Ae_i\otimes_{\Bbbk} e_iA\otimes_{A}Ae_i\otimes_{\Bbbk} e_jA\\
        \displaystyle\sum_{s}b_se_i\otimes e_jc_s&\mapsto&\displaystyle\sum_{s}\sum_{a\in\mathtt{A}}b_sae_i\otimes e_ia^*\otimes e_i\otimes e_jc_s,
        \end{array}\\
        \begin{array}{rcl}
       Ae_i\otimes_{\Bbbk} e_jA&\to & Ae_i\otimes_{\Bbbk} e_jA\otimes_{A}Ae_j\otimes_{\Bbbk} e_jA \\
       \displaystyle\sum_{s}b_se_i\otimes e_jc_s&\mapsto&\displaystyle\sum_{s}\sum_{a\in\mathtt{A}}b_se_i\otimes e_j\otimes ae_j\otimes e_ja^*c_s.
        \end{array}\end{split}
        \]
\end{itemize}
We emphasize the major difference between $\ca$ and  the $2$-category $\cC_A$ from \cite[Subsection~7.3]{MM1}: in $\ca$ we do not take any functors isomorphic to the
regular $A$-$A$-bimodule $_{A}A_{A}$ or any variation of it.

Since $A$ is finite dimensional, $\ca$ is a finitary bilax-unital $2$-category. Denote by $\rF_{ij}$ the functor isomorphic to tensoring with the $A$-$A$-bimodule $Ae_i\otimes_{\Bbbk} e_jA$.
Note that $\ca$ has only one two-sided cell which is strongly regular. Indeed, we have $n$ left cells, that is,
$\mathcal L_i=\{\rF_{ji}| 1\leq j\leq n\}, 1\leq i\leq n$, and $n$ right cells, that is, $\mathcal R_i=\{\rF_{ij}| 1\leq j\leq n\}, 1\leq i\leq n$.

\begin{prop}\label{prop:cafiax}
If $A$ is weakly symmetric, then $\ca$ is a fiax category.
\end{prop}

\begin{proof}
If $A$ is weakly symmetric, then we have the adjoint pair $(Ae_i\otimes_{\Bbbk}e_jA\otimes_A{}_-, Ae_j\otimes_{\Bbbk} e_iA\otimes_A{}_-)$ in the $\ca$ with adjunction morphisms given by
\[\begin{array}{rcl}
Ae_j\otimes_{\Bbbk}e_jA&\to&Ae_j\otimes_{\Bbbk} e_iA\otimes_A Ae_i\otimes_{\Bbbk}e_jA\\
\displaystyle\sum_{s}b_se_j\otimes e_jc_s&\mapsto& \displaystyle\sum_{s}b_se_j\otimes e_i\otimes e_i\otimes e_jc_s
\end{array}\]
and
\[\begin{array}{rcl}
Ae_i\otimes_{\Bbbk}e_jA\otimes_A Ae_j\otimes_{\Bbbk} e_iA&\to&Ae_i\otimes_{\Bbbk}e_iA\\
 \displaystyle\sum_{s}a_se_i\otimes e_jb_s\otimes   \displaystyle\sum_{t}c_te_j\otimes e_id_t&\mapsto&
\displaystyle\sum_{s,t}\mathrm{tr}(b_sc_t)a_se_i\otimes e_id_t.
\end{array}\]
This allows one to define a weak involution $({}_-)^\star$ in the obvious way.
Alternatively, one can refer to the proof of Proposition~\ref{cjfiax} for details in a more general case. The claim follows.
\end{proof}

For the converse of Proposition~\ref{prop:cafiax}, note that $\ca$, in addition to consisting of one two-sided cell, has the property that any of its nontrivial $2$-ideal contains some identity $2$-morphism $\id_{\rF}$. Following \cite{MM2}, we call this property being {\em $\mathcal J$-simple}.
Now, if $\cC$ is a fiax category with one two-sided cell which is strongly regular and $\mathcal J$-simple, then we can find a finite dimensional weakly symmetric basic algebra $A$ so that $\cC$ is biequivalent to $\ca$. This follows from a similar result for fiat categories from \cite{MM3}. Note that $\widehat{\cC}$ is $\mathcal J$-simple, has the same strongly regular two-sided cell as $\cC$, and each of the added 1-morphisms forms a two-sided cell by itself.
Such a fiat $2$-category is biequivalent to $\widehat{\ca}$ by \cite[Theorem~13]{MM3}. And the biequivalence is restricted to a biequivalence (of bilax-unital 2-categories) between $\cC$ and $\ca$.
We emphasize that this biequivalence does not take into account the weak involution $({}_-)^\star$, that is, the weak involution $({}_-)^\star$ may not commute with the biequivalence.
Even in the fiat setting of \cite{MM3}, we do not have a complete classification of possible weak involutions.

\subsection{The category $\cj$}\label{sec:example2}

Let $\cC$ be a fiat $2$-category and $\mathcal J$ a two-sided cell in $\cC$.
Let $\mathcal L_1,\cdots \mathcal L_n$, where $n$ is a positive integer, be a complete list of left cells in $\mathcal J$.
Denote by~$\mathcal R_i=\mathcal L_i^\star$.
By \cite[Proposition~17]{MM1}, the intersection $\mathcal H_i=\mathcal L_i \bigcap \mathcal R_i$ contains exactly one Duflo 1-morphism, denoted by $\rD_i$.
Let $d_i:D_i\to\mathbbm{1}_{\mi_i}$ be the corresponding defining $2$-morphism,
cf. Subsection~\ref{s3.9}.
Using \eqref{dufloto1}, we define two natural transformations given by
\[d_i\id_{\rF}:\rD_i\rF\to \rF\quad\text{and}\quad \id_{\rG}d_i:\rG\rD_i\to\rG,\]
where $\rF,\rG$ run through all $1$-morphisms whenever $\rD_i\rF$ and $\rG\rD_i$ make sense. One also have the dual natural transformations given by
$(d_i\id_{\rF})^{\star}=\id_{\rF^{\star}}(d_i)^\star$ and $(\id_{\rG}d_i)^{\star}=(d_i)^\star\id_{\rG^\star}$.

Consider the $2$-semicategory $\cC(\mathcal J)$ consisting of all the objects of sources and targets of $1$-morphisms in $\mathcal J$ and
\begin{equation}\label{C(J)}
    \cC(\mathcal J):= \add(\{\rF\ |\ \rF\geq_{J} \mathcal J\}) / \cI,
\end{equation}
where $\cI$ is the unique maximal $2$-ideal which does not contain any $\id_{\rF}$, for $\rF\in\mathcal{J}$.
Define $\cj$ to be a bilax-unital $2$-category such that
\begin{itemize}
    \item $\ob \cj=\{1,\ldots, n\}$;
    \item $\cj(i,j):=\add(\mathcal R_i\bigcap \mathcal L_j)$ for any $1\leq i,j\leq n$, that is, the $2$-full $2$-semisubcategory of $\cC(\mathcal J)$ consisting of objects in $\add(\mathcal R_i\bigcap \mathcal L_j)$;
    \item the horizontal composition is inherited from $\cC$;
    \item each $\rD_i$ is a lax unit with the lax unitors $l_{j,i,\rF}=d_i\id_{\rF}$ and $r_{i,k,\rG}=\id_{\rG}d_i$ for any $1$-morphisms $\rF\in\cj(j,i)$ and $\rG\in\cj(i,k)$;
    \item each $(\rD_i)^{\star}$ is an oplax unit with the oplax unitors $l'_{j,i,\rF}=(r_{i,j,\rF^{\star}})^{\star}$ and $r'_{i,k,\rG}=(l_{k,i,\rG^{\star}})^{\star}$ for any $1$-morphisms $\rF\in\cj(j,i)$ and $\rG\in\cj(i,k)$.
\end{itemize}
The fact that $\cj$ is a finitary bilax-unital $2$-category due to the definition and the fact that $\cC$ is finitary.

\begin{prop}\label{cjfiax}
The finitary bilax-unital $2$-category $\cj$ is fiax.
\end{prop}

\begin{proof}
The weak involution $({}_-)^\star$ on $\cC$ restricts to a weak involution on $\cj$,
which we still denoted by $({}_-)^{\star}$.
We need to prove that $(\rF,\rF^{\star})$ is an adjoint pair in $\cj$, for any $1$-morphism $\rF\in\cj(i,j)$.
In the fiat $2$-category $\cC$, we have the following adjunction morphisms
\[\alpha:\mathbbm{1}_{\mi_i}\to \rF^{\star}\rF\quad\text{and}\quad\beta:\rF\rF^{\star}\to \mathbbm{1}_{\mi_j},\]
for the adjoint pair $(\rF,\rF^{\star})$. The latter adjunction morphisms also belong to $\cC(\mathcal{J})$.
Since $\rF\rF^\star\in \add(\mathcal{L}_j)$, the $2$-morphism $\beta$ factors through $d_j$, that is, there exists a
$2$-morphism $\beta':\rF\rF^{\star}\to\rD_j$ such that $\beta=d_j\circ_{\rv}\beta'$. By the dual argument, the $2$-morphism $\alpha$
factors through $(d_i)^\star$, that is, there exists a $2$-morphism $\alpha':\rD_i\to\rF^\star\rF$ such that $\alpha=\alpha'\circ_{\rv}(d_i)^\star$.
We claim that $\alpha'$ and $\beta'$ give  adjunction morphisms for the adjoint pair $(\rF,\rF^{\star})$ in $\cj$.
By definition, the diagram
\[\xymatrix@C=2.5pc{\rF\ar@{=}[d]\ar[rr]^{r'_{i,j,\rF}=\id_{\rF}(d_i)^{\star}\ }&& \rF\rD_i\ar[rr]^{\id_{\rF}\alpha'}&&\rF\rF^\star\rF\ar@{=}[d]\ar[rr]^{\beta'\id_{\rF}}&& \rD_j\rF\ar[rr]^{l_{i,j,\rF}=d_j\id_{\rF}}&& \rF\ar@{=}[d]\\
\rF\ar[rrrr]^{\id_{\rF}\alpha}&&&&\rF\rF^\star\rF\ar[rrrr]^{\beta\id_{\rF}}&&&&\rF,}\]
commutes. This implies that the path along the consecutive arrows from $\rF$ to $\rF$ in the top row equals $(\id_{\rF}\alpha)\circ_{\rv}(\beta\id_{\rF})$ and hence equals $\id_{\rF}$ by adjunction. The other condition can be proved similarly and the claim follows.
\end{proof}

\begin{rem}
Note that $\cC(\mathcal J)$ and $\cC_\mathcal{J}$ might have different number of objects. Even though they have the same $1$-
and $2$-morphisms, $\cC(\mathcal J)$ might not be fiax in general.
\end{rem}

\subsection{Comparing the first two examples}\label{sec:example}

Let $A$ be a finite dimensional, basic, connected and weakly symmetric algebra.
Denote by $\cC_A$ the $2$-category whose object set consists of only one object $\clubsuit$ (identified with a small category equivalent to the category of left $A$-modules), $1$-morphisms given by functors isomorphic to those tensoring with $A$-$A$-bimodules in $\add(\{A,\,A\otimes_{\Bbbk} A\})$ and
$2$-morphisms given by natural transformations between those functors. The $2$-category $\cC_A$ is fiat with at most two two-sided cells, that is,
$\mathcal{J}_0$ consisting of the identity $1$-morphism which corresponds to $A$ and the two-sided cell $\mathcal{J}$ consisting of all $1$-morphisms
which correspond to indecomposable projective $A$-$A$-bimodules. Let $\{e_1,\ldots,e_n\}$ be a complete set of primitive orthogonal idempotents of $A$.
Then the associated $(\cC_A)_{\mathcal{J}}$ from Subection~\ref{sec:example2}
coincides with $\ca$ in Subection~\ref{sec:example1}.

\subsection{Soergel bimodules}\label{sec:example4}

Let $\Bbbk=\mathbb C$.
For any finite Coxeter group $W=(W,S)$, one can define the $2$-category $\cS$ of Soergel bimodules over the coinvariant algebra of $W$ as in \cite{KMMZ}.
Note that $\cS$ is a fiat $2$-category with one object. The $1$-morphisms are indexed by the element in $W$ and  left/right/two-sided cells in $\cS$ are given by the Kazhdan-Lusztig left/right/two-sided cells in the Hecke algebra of $(W,S)$; see \cite{So3} and \cite{EW} for details.
The Duflo $1$-morphisms are the so-called Duflo involutions which are self-adjoint. Therefore, in the associated fiax category $\cS_{\mathcal{J}}$ with respect to any two-side cell $\mathcal{J}$, we have that $\rI_{\mi}=\rI_{\mi}^\star=\rI_{\mi}'$, for each object $\mi\in\ob\cS_{\mathcal{J}}$.

\subsection{$G$-symmetric projective bimodules}\label{sec:example5}

Let $G$ be any finite abelian group and $A$ a finite dimensional, weakly symmetric, $\Bbbk$-algebra with $G$-action. Assume that the characteristic of $\Bbbk$ does not divide $|G|$. One can define the $2$-category $\mathscr{G}_A$ of projective symmetric bimodules, cf. \cite{MMZ}. Note that it is a fiat $2$-category with $n+1$ two-sided cells, where $n$ is the number of connected components of $A$. In this case, the Duflo $1$-morphism of a left cell
is not self-dual in general. Therefore, in the associated fiax category $(\mathscr{G}_A)_{\mathcal{J}}$ with respect to the maximal two-sided cell, we have $\rI_{\mi}\not\cong\rI'_{\mi}$, for some object $\mi$. One can also refer to \cite[Subsection~9.3]{KMMZ} for
an example in the case $|G|=2$.

\subsection{Projective modules over group algebras}\label{sec:example6}

Let $G$ be a finite group and $\Bbbk$ an algebraically closed field.
Consider the category $G$-mod of all finite dimensional $G$-modules (over $\Bbbk$).
If the characteristic of $\Bbbk$ divides $|G|$, the category
$G$-mod might have infinitely many isomorphism classes of indecomposable
objects and so be non-finitary. The category $G$-mod has the usual structure of a monoidal category
with respect to the usual tensor product of $G$-modules. The unit object for
this monoidal structure is the trivial $G$-module. The category
$G$-mod has both involution and adjunctions, see \cite[Example~2.10.13]{EGNO}.

Consider the category $G$-proj of projective $G$-modules
(or, more precisely, a small category equivalent to it). This category is always
finitary. If the characteristic of $\Bbbk$ does not divide $|G|$, the categories
$G$-mod and $G$-proj coincide (and are semi-simple), in particular,
$G$-proj is a fiat $2$-category. If the characteristic of $\Bbbk$ divides $|G|$,
the monoidal structure on $G$-mod equips $G$-proj with an associative
tensor product. Further, $G$-proj is a fiax category in which the lax (resp., oplax) unit is the projective cover (resp., injective envelope) of the trivial $G$-module. (Note that the projective cover and injective envelope coincide in $G$-mod). The unitors are given by
the horizontal composition with the canonical projection (resp., injection)
between the trivial module and its projective cover (resp., injective envelope).

We can in fact think of $G$-proj as an instance of the construction from Subsection \ref{sec:example2}. Let $\cC$ be (a small category equivalent to) $G$-mod. Then the indecomposable projectives (which are also injectives) in $\cC$ form a two-sided cell $\mathcal J$.
Note that, to define $\cC_{\mathcal J}$ as a $2$-semicategory, we do not need $\cC$ to be finitary.
If further $\mathcal J$ is finite, we can give a natural bilax-unital structure to $\cC_{\mathcal J}$, since the Duflo $1$-morphisms (which serve as lax units) exist in $\mathcal J$.
In this example, we know that the projective cover of the trivial module is naturally a lax unit in $G$-proj, which makes $\cC_\mathcal{J}$ fiax.

\vspace{7mm}

\noindent
H.~K.: Department of Mathematics, Uppsala University, Box. 480,
SE-75106, Uppsala,\\ SWEDEN, email: {\tt hankyung.ko\symbol{64}math.uu.se}
\vspace{3mm}

\noindent
V.~M.: Department of Mathematics, Uppsala University, Box. 480,
SE-75106, Uppsala,\\ SWEDEN, email: {\tt mazor\symbol{64}math.uu.se}
\vspace{3mm}

\noindent
X.~Z.: Department of Mathematics, Uppsala University, Box. 480,
SE-75106, Uppsala,\\ SWEDEN, email: {\tt xiaoting.zhang09\symbol{64}hotmail.com}

\end{document}